\documentclass[english]{amsart}

\usepackage{epsfig}
\usepackage{amsmath}
\usepackage{amssymb}
\usepackage{amscd}
\usepackage{graphicx}
\usepackage{color}
\usepackage{verbatim}
\usepackage{enumerate}
\usepackage{bbm}
\usepackage{tikz}
\usepackage{verbatim}
\usepackage{dsfont}
\usepackage{url}
\usepackage{soul}
\usepackage{babel}
\usepackage{rotating}
\usepackage{amssymb}
\newcommand{\longnearrow}{
        \begin{turn}{20}
               \raisebox{-1ex}{$\Longrightarrow$}
        \end{turn}
}
\newcommand{\longsearrow}{
        \begin{turn}{-20}
                \raisebox{-1ex}{$\Longrightarrow$}
        \end{turn}
}

\newtheorem{thm}{Theorem}[section]
\newtheorem{lem}[thm]{Lemma}
\newtheorem{prop}[thm]{Proposition}
\newtheorem{fact}[thm]{Proposition}

\newtheorem{que}[thm]{Question}
\newtheorem{cor}[thm]{Corollary}
\newtheorem{defn}[thm]{Definition}

\theoremstyle{remark}
\newtheorem{rem}[thm]{Remark}

\newtheorem{exam}[thm]{Example}

\def \N {\mathbb N}
\def \S {\mathcal S}
\def \F {\mathcal F}

\def \D {\mathcal D}
\def \B {\mathcal B}

\def \K {\mathcal K}

\def \CT {\mathcal T}
\def \CS {\mathcal S}
\def \T {\mathsf T}
\def \TTT {\boldsymbol{\mathsf T}}
\def \Z {\mathbb Z}
\def \R {\mathbb R}

\def \Q {\mathcal Q}

\def \M {\mathcal M}
\def \mgx {\mathcal M_G(X)}
\def \mgxb {\mathcal M_G(\bar X)}
\def \megx {\mathcal M^{\mathsf e}_G(X)}

\def \mgl {\mathcal M_G(\Lambda^G)}
\def \P {\mathcal P}

\def \Fr {\mathsf{Fr}}

\def \Rep {\bar v}

\def \sq {sequence}

\def \xg {$(X,G)$}

\def \xmt {$(X,\Sigma,\mu,T)$}

\def \tl {topological}
\def \im {invariant measure}
\def \inv {invariant}
\def \ds {dynamical system}

\def \htop {h_{\mathsf{top}}}

\def \eps {\varepsilon}
\def \dist {\mathsf{dist}}

\title[Deterministic functions on amenable semigroups]{Deterministic functions on amenable semigroups and a generalization of the Kamae--Weiss theorem on normality preservation}
\author{Vitaly Bergelson}
\address{Department of Mathematics, Ohio State University, Columbus,
  OH 43210, USA}
\email{vitaly@math.ohio-state.edu}

\author{Tomasz Downarowicz}
\address{Faculty of Pure and Applied Mathematics, Wroc\l aw University
  of Science and Technology, Wybrze\.ze Wyspia\'nskiego 21, 50-370
  Wroc\l aw, Poland}
\email{Tomasz.Downarowicz@pwr.edu.pl}

\author{Joseph Vandehey}
\address{Department of Mathematics, The University of Texas at Tyler, 3900 University Blvd.
Tyler, TX 75799, USA}
\email{JVandehey@uttyler.edu}

\numberwithin{equation}{section}

\newcommand{\notet}[1]{\textcolor{red}{{#1}}}

\begin{document}
{\let\thefootnote\relax\footnote{\today}}
\thanks{The research of first two authors was supported by the NCN grant 2018/30/M/ST1/00061. Additionally, the research of the second author was supported by the Wroc\l aw University of Science and Technology grant 049U/0052/19.}
\keywords{Amenable semigroup action, deterministic function, normality-preserving set, subexponential complexity}

\subjclass[2010]{prim.: 37B05, 37C85, 37B10, sec.: 43A07, 20E07}

\begin{abstract}
A classical Kamae--Weiss theorem states that an increasing sequence $(n_i)_{i\in\N}$ of positive lower density is \emph{normality preserving}, i.e.\ has the property that for any normal binary sequence $(b_n)_{n\in\N}$, the sequence $(b_{n_i})_{i\in\N}$ is normal, if and only if $(n_i)_{i\in\N}$ is a deterministic sequence. Given a countable cancellative amenable semigroup $G$, and a F\o lner sequence $\F=(F_n)_{n\in\N}$ in $G$, we introduce the notions of normality preservation, determinism and subexponential complexity for subsets of $G$ with respect to $\F$, and show that for sets of positive lower $\F$-density these three notions are equivalent. The proof utilizes the apparatus of the theory of tilings of amenable groups and the notion of tile-entropy. We also prove that under a natural assumption on $\F$, positive lower $\F$-density follows from normality preservation. Finally, we provide numerous examples of normality preserving sets in various semigroups.
\end{abstract}

\maketitle

\tableofcontents

\section{Introduction}
The impetus for this paper comes from the desire to better understand and, if possible, extend to the context of general countable cancellative amenable semigroups, the Kamae--Weiss theorem \cite{W1,Kamae} (see also \cite{W2}) on normality preservation along deterministic sets.

In this paper we abide by the convention that $\N=\{1,2,\dots\}$. Given an integer $b\ge2$, a \sq\ $x=(x_n)_{n\in\N}\in\{0,1,\dots,b-1\}^\N$ is called normal to base $b$ if every word $w=\langle w_1,w_2,\dots,w_k\rangle\in\{0,1,\dots,b-1\}^k$ occurs in $x$ with frequency $b^{-k}$. Let us say that an increasing \sq\ $(n_i)_{i\in\N}\subset\N$ preserves normality if for any \sq\ $(x_n)_{n\in\N}$ that is normal to base $b$, the \sq\ $(x_{n_i})_{i\in\N}$ is also normal. Wall \cite{Wa} showed, using purely combinatorial techniques, that any periodic increasing sequence of integers preserves normality; Furstenberg \cite{Fu2} later reproved this result as a simple consequence of the theory of disjointness for measure preserving systems. Kamae and Weiss succeeded in characterizing the increasing sequences which preserve normality. Their theorem states, roughly speaking, that an increasing \sq\ of integers $(n_i)_{i\in\N}$ preserves normality if and only if it is deterministic and has positive lower density. (See also \cite{BV}.)

The notion of a deterministic sequence, which appears in the Kamae--Weiss theorem, has (at least) two equivalent definitions. One of them involves quasi-generic points for \im s on the shift space $\{0,1\}^\N$ and measure-theoretic entropy. We will freely identify an increasing \sq\ $(n_i)_{i\in\N}$ with its indicator function $y=\mathbbm1_{\{n_i:i\in\N\}}\in\{0,1\}^\N$. According to \cite[Definition 1.6]{W2}, a \sq\ $y\in\{0,1\}^\N$ is \emph{completely deterministic} if any \im, arising as an accumulation point of the \sq\ of averages $\bigl(\frac1n\sum_{i=0}^{n-1}\delta_{\sigma^i(y)}\bigr)_{n\in\N}$, has entropy zero under the action of the standard shift transformation $\sigma$ (here $\delta_z$ denotes the probability measure concentrated at $z$). The other definition is closely related to the concept of  subexponential subword complexity (see \cite[Lemma 8.9]{W2}): 

\emph{A \sq\ $y=(y_n)_{n\in\N}\in\{0,1\}^\N$ is completely deterministic if and only if for any $\eps>0$ there exists $k\in\N$ such that the collection of subwords of length $k$ appearing in $y$ can be divided in two families: the first family has cardinality smaller than $2^{\eps k}$, and the words from the second family appear in $y$ with frequencies summing up to at most $\eps$.}

Yet another approach to determinism, due to Rauzy \cite{Rauzy}, utilizes the concept of a noise function, essentially a measure of predictability. This approach will not be explored in this paper.
\medskip

While the Kamae--Weiss theorem has an elegant formulation for \sq s in $\N$, both the proof of the theorem (see \cite{W1,Kamae,W2}) and the definition of determinism, which is instrumental to it, are quite non-trivial. When one attempts to generalize this theorem to broader spaces than $\N$, one has, first of all, to properly define the notions of normality, normality along a subset, and determinism. A general approach to normality (which broadens the concept even for \sq s in $\N$) was developed in a recent article \cite{BDM}. In the current paper the authors make the next step and introduce and study in some depth the notions of normality along a set and that of determinism. Among other things, we show, skipping some technical details, that
\begin{enumerate}[(i)]
	\item determinism = normality preservation, and
	\item determinism = subexponential complexity.
\end{enumerate}
\medskip

Below is a more precise description of the main results obtained in this paper.
\medskip

$\bullet$ We study two types of deterministic real-valued functions on a countable cancellative amenable semigroup $G$: strongly deterministic and $\F$-deterministic (i.e.\ deterministic with respect to a fixed F\o lner \sq\ $\F=(F_n)_{n\in\N}$ in $G$). Strongly deterministic functions generalize classical ``deterministic \sq s'', i.e.\ functions on $\N$ arising by reading consecutive values of a continuous function along the orbit of point in a \ds\ with zero \tl\ entropy. The definition of $\F$-deterministic functions utilizes the notion of $\F$-quasi-generic points and measure-theoretic entropy. $\{0,1\}$-valued $\F$-deterministic functions generalize ``completely deterministic sets'' introduced by Weiss in \cite{W1} for $G=\N$ with the standard F\o lner \sq\ $F_n=\{1,2,\dots,n\}$. It is worth mentioning that our notion can be applied to any F\o lner \sq, leading to new classes of subsets already at the level of~$\N$.

$\bullet$ The notion of a ``normality-preserving set'' in $\N$ (also introduced in \cite{W1}) is very natural; however, it implicitly relies on the fact that for any infinite subset $A\subset\N$ there is a natural bijection from $A$ to $\N$ which sends the ``traces'' $F_n\cap A$ of the elements $F_n$ of the standard F\o lner \sq\ to elements of the same F\o lner \sq. For example, when $G=\N$ and $F_n=\{1,2,\dots, n\}$ (which we will think of as the classical case), then for any infinite subset $A\subset \N$, the natural bijection sends the $k$th smallest element of $A$ to $k$, so that if $|F_n\cap A|=m$, then the bijection naturally sends $F_n\cap A$ to $F_m$. This property fails (regardless of the F\o lner \sq) already for $\N^2$, probably the simplest semigroup beyond $\N$, let alone for more general semigroups. This is why it is a challenge to reasonably define ``normality along $A$'' (and, subsequently, normality-preserving subsets) in full generality for countable cancellative amenable semigroups. In this paper we propose three notions of ``normality along $A$''. Just like in the case of $\F$-deterministic sets, we work in the context of a fixed F\o lner \sq\ $\F$. Although the proposed notions are not mutually equivalent, the resulting three types of $\F$-normality-preserving sets turn out to coincide in the class of sets of positive lower $\F$-density. Moreover, in the classical case ($G=\N$ and $F_n=\{1,2,\dots,n\}$),
all three types of $\F$-normality preservation coincide with normality preservation in the sense of Weiss.

$\bullet$ Our first main achievement, and the evidence that our notions of both $\F$-normality preservation and $\F$-determinism are satisfactory, is a theorem in the spirit of Kamae--Weiss \cite{Kamae,W1} (see also \cite{W2}). We show that for any F\o lner \sq\ $\F$ in $G$, a set $A\subset G$, which has positive lower $\F$-density, preserves $\F$-normality if and only if it is $\F$-deterministic. Moreover, under additional (mild) assumptions on the F\o lner \sq, we show that $\F$-normality preservation implies positive lower $\F$-density. Since this strengthening applies, in particular, to the classical case, we slightly improve the original Kamae--Weiss theorem.\footnote{\notet{This improvement is also implicit in a recent paper \cite{WK}.}} 

$\bullet$ Our second main achievement is a characterization of finite-valued deterministic functions in terms of complexity. Strong determinism is characterized by the subexponential growth of complexity, which is a straightforward generalization of the corresponding well-known fact for actions of $\N$ or $\Z$. A complexity-based characterization of $\F$-determinism requires a careful definition of $\F$-complexity and its growth rate. We provide such a definition and successfully prove the corresponding characterization. The proof is complicated and involves the apparatus of the theory of tilings of amenable groups developed in \cite{DHZ} and the notion of tile-entropy introduced recently in \cite{DZ}.

\medskip
We conclude our paper with a series of examples of deterministic functions and normality-preserving sets in various semigroups:

$\bullet$ Some well-known classes of actions which have \tl\ entropy zero give rise to some novel examples of strongly deterministic functions. For instance we show strong determinism for certain functions on (or subsets of) $\Z^d$ defined in terms of multiple recurrence or in terms of so-called generalized polynomials. Strong determinism is proven via rather deep results connecting such functions with actions by translations on nilmanifolds. 

$\bullet$ We also give a new natural example of a strongly deterministic set in the non-commutative amenable group of finite permutations of a countable set. 

$\bullet$ Automatic sequences are a classical example of strongly deterministic subsets of $\N$. We describe two variations of automatic sequences in a broader semigroup context, both of which are strongly deterministic. This provides a second non-commutative example, as our results apply to the integer Heisenberg group.

$\bullet$ It is much harder to find concrete (and natural) examples of $\F$-deterministic sets which are not strongly deterministic. We provide examples of this kind in the rings of integers of algebraic extensions of $\mathbb Q$, as well as in lattices on $\mathbb{R}^n$, where we generalize a classical example of a completely (but not strongly) deterministic subset of $\N$ --- the set of square-free numbers.

\section{Preliminaries}

\subsection{Amenable (semi)groups, F\o lner \sq s}
Throughout this paper, $G$ denotes an infinite countable cancellative semigroup (or an infinite countable group). Most of the time, we will assume that the semigroup has a unit, denoted by $e$. Given a nonempty finite set $K\subset G$ and $\eps>0$ we will say that a finite set $F\subset G$ is
\emph{$(K,\eps)$-invariant} if
\[
\frac{|KF\triangle F|}{|F|}\le\varepsilon
\]
($\triangle$ stands for the symmetric difference of sets, and $|\cdot|$ denotes cardinality of a set).
\begin{defn}\label{Folner}{\color{white}.}

\begin{itemize}
	\item A \sq\ of finite sets $\F=(F_n)_{n\in\N}$ in $G$ is a \emph{F{\o}lner \sq}
	if for any nonempty finite $K\subset G$ and $\eps>0$, the sets $F_n$ are eventually 
	(i.e.\ except for finitely many of them) $(K,\eps)$-\inv.
	\item $G$ is amenable if it has a F\o lner \sq.
\end{itemize}

\end{defn}
\begin{rem}\label{sgip} We will often use the following jargon: when we say that a set $F$ has ``good invariance properties'', we will mean that there exists a (large) finite set $K\subset G$ and a (small) $\eps>0$ such that $F$ is $(K,\eps)$-\inv. A F\o lner \sq\ is any \sq\ of finite sets that eventually has arbitrarily good invariance properties. When saying that some condition holds for any set $F$ with ``sufficiently good invariance properties'' we will mean that there exists a finite set $K\subset G$ and an $\eps>0$ such that the condition holds for any $(K,\eps)$-invariant set $F$.
\end{rem}

A F{\o}lner \sq\ $\F=(F_n)_{n\in\N}$ will be called 
\begin{itemize}
	\item \emph{nested}, if for each $n\in\N$, $F_n\subset F_{n+1}$, 
	\item \emph{centered}, if $G$ has a unit $e$ and for each $n\in\N$, $e\in F_n$,
	\item \emph{disjoint}, if the sets $F_n$ are pairwise disjoint,
	\item \emph{exhaustive}, if $\bigcup_{n\in\N}F_n=G$.
\end{itemize}
The most frequently considered example of a F\o lner \sq\ appears in the additive semigroup of positive integers $\N$, and is given by $F_n=\{1,2,\dots,n\}$. It is nested and exhaustive. In $\N_0=\N\cup\{0\}$ one often uses the centered version 
$F_n=\{0,1,2,\dots,n-1\}$. A general F\o lner \sq\ need not have any of the properties listed above. Take for example $F_n=\{3^n,3^n+1,3^n+1,\dots,3^{n+1}+(-2)^n\}$~in~$\N_0$.

Given a F{\o}lner sequence $\F=(F_n)_{n\in\N}$ in $G$ and a set $A\subset G$, one
defines the \emph{upper and lower $\F$-densities} of $A$ by the
formulas
\begin{align*}
\overline d_{\F}(A)&=\limsup_{n\to\infty}\frac{|F_n\cap A|}{|F_n|},
\\
\underline d_{\F}(A)&=\liminf_{n\to\infty}\frac{|F_n\cap A|}{|F_n|}.
\end{align*}
If $\overline d_{\F}(A)=\underline d_{\F}(A)$, then we denote
the common value by $d_{\F}(A)$ and call it the
\emph{$\F$-density} of $A$. 

We will also use the notions of upper and lower Banach densities:
\begin{align*}
d^*(A)&=\sup_{\F}\overline d_{\F}(A),\\
d_*(A)&=\inf_{\F}\underline d_{\F}(A),
\end{align*}
where $\F$ ranges over all F\o lner \sq s in $G$.
If $d^*(A)=d_*(A)$, then we denote the common value by $d_{\mathsf B}(A)$ and call it the \emph{Banach density} of $A$. 

\begin{defn}\label{core} Let $F'\subset F$ and $K$ be nonempty finite subsets of $G$.
\begin{itemize}
	\item Given $\eps>0$, $F'$ is called a \emph{$(1-\eps)$-subset} of $F$ if 
	$\frac{|F'|}{|F|}\ge1-\eps$. 
	\item The \emph{$K$-core} of $F$ is the set $F_K=\{g\in F:Kg\subset F\}$. 
\end{itemize}
\end{defn}

The following lemma is elementary (see Lemma~2.6 in~\cite{DHZ} or Lemma~2.4 in~\cite{BDM}). 

\begin{lem}\label{estim}
Fix some $\eps>0$ and a finite subset $K\subset G$. 
If $F$ is $(K,\frac\varepsilon{|K|})$-invariant, then the $K$-core of $F$ is a
$(1-\eps)$-subset of $F$.
\end{lem}

\begin{defn}\label{equi}
We will say that two \sq s of sets, $(F_n)_{n\in\N}$ and $(F'_n)_{n\in\N}$, are \emph{equivalent} if $\frac{|F'_n\triangle
  F_n|}{|F_n|}\to 0$ (equivalently, $\frac{|F'_n\triangle
  F_n|}{|F'_n|}\to 0$ or $\frac{|F'_n\cap F_n|}{|F'_n\cup F_n|}\to 1$).
\end{defn}

It is immediate to see that if $\F=(F_n)_{n\in\N}$ is a F{\o}lner \sq\ and $\F'=(F_n')_{n\in\N}$ is an equivalent \sq\ of sets, then the latter is also a F{\o}lner \sq\ and the notions of (upper/lower) $\F$-density and $\F'$-density coincide.

\subsection{Semigroup actions, \im s}\label{2.2}
Let $X$ be a compact metric space. We will say that the semigroup $G$ \emph{acts on $X$} if there exists a homomorphism $\tau$ from $G$ to the semigroup $C(X,X)$ of continuous transformations of $X$ (with composition). If $G$ is a group, then, for any $g\in G$, $\tau(g)$ is a homeomorphism. If $G$ is a semigroup without a unit, $G$ embeds naturally in $G\cup\{e\}$. The unit $e$ always acts on $X$ as the identity map. So, whenever convenient, we will tacitly assume that $G$ has a unit. The action of $G$ on $X$ will be referred to as the \emph{\ds} $(X,\tau)$, alternatively denoted by $(X,G)$, if there is no ambiguity as to which particular action of $G$ on $X$ is considered. Instead of $\tau(g)(x)$ we will write $g(x)$. For subsets $A\subset G$ and $B\subset X$, $A(B)$ will denote the set $\{g(x):g\in A, x\in B\}$. If $G$ acts on $X$, it also acts on the space of Borel probability measures on $X$, denoted by $\M(X)$, endowed with the (compact and metrizable) weak* topology, by the formula $g(\mu)(B)=\mu(g^{-1}(B))$, where $B$ is a Borel subset of $X$. The set of \emph{\im s}, i.e.\ measures $\mu\in\M(X)$ satisfying $g(\mu)=\mu$ for all $g\in G$, will be denoted by $\mgx$. It is always a compact convex subset of $\M(X)$, and if $G$ is amenable, it is also nonempty. Indeed, it is not hard to see that for any F\o lner \sq\ $\F=(F_n)_{n\in\N}$ and any \sq\ of measures $(\nu_n)_{n\in\N}$ from $\M(X)$, any accumulation point of the \sq\ of Ces\`aro averages
\begin{equation}\label{domi}
\frac1{|F_n|}\sum_{g\in F_n}g(\nu_n)
\end{equation}
is an \im. 

An \im\ is \emph{ergodic} if any Borel measurable set $B$ which is \emph{\inv} (i.e.\ such that for each $g\in G$, $B\subset g^{-1}(B)$) has measure either zero or one.
It is well known that, whenever $\mgx$ is nonempty (which is the case if $G$ is amenable), the extreme points of $\mgx$ are precisely the ergodic measures, the set $\megx$ of all ergodic measures is a Borel measurable subset of $\mgx$, and any $\mu\in\mgx$ can be uniquely represented as the barycenter of a Borel probability measure $\xi_\mu$ supported by $\megx$: 
$$
\mu = \int\nu\,d\xi_\mu(\nu).
$$
More explicitly, the above formula means that for any measurable set $B\subset X$, we have
$$
\mu(B) = \int\nu(B)\,d\xi_\mu(\nu).
$$

\subsection{Generic points}\label{gen} We now fix a F\o lner \sq\ $\F=(F_n)_{n\in\N}$ in $G$. A point $x$ is \emph{$\F$-generic} for an \im\ $\mu$ if the \sq\ of Ces\`aro averages
\begin{equation}\label{ca}
\frac1{|F_n|}\sum_{g\in F_n}\delta_{g(x)}
\end{equation}
converges to $\mu$, where $\delta_z$ denotes the point-mass at $z$. 
In most cases, $\F$-generic points exist for all ergodic and some (but not all) other \im s. In certain specific systems (for example in the full shift) $\F$-generic points exist for all \im s. We will say that $x$ is \emph{$\F$-quasi-generic} for an \im\ $\mu$ if $\mu$ is an accumulation point of the \sq\ \eqref{ca}. Clearly, every point $x\in X$ is $\F$-quasi-generic for at least one \im. The following fact  plays a crucial role in our considerations.

\begin{fact}\label{miner} If $\mu\in\megx$ and $\F=(F_n)_{n\in\N}$ is a F\o lner \sq\ then there exists a sub\sq\ $\F_\circ=(F_{n_k})_{k\in\N}$ such that $\mu$-almost every point $x\in X$ is $\F_\circ$-generic for $\mu$ (in particular, $x$ is $\F$-quasi-generic for $\mu$).
\end{fact}
\begin{proof}
By the mean ergodic theorem (which holds for amenable semigroups and all F\o lner \sq s), for every $f\in L^2(\mu)$ we have
$$
\lim_{n\to\infty} \frac1{|F_n|}\sum_{g\in F_n}f(g(x)) = \int f\,d\mu
$$ 
(the convergence is in the norm of $L^2(\mu))$. For any fixed $f$ there is a sub\sq\ of $\F$ along which the convergence holds for $\mu$-almost every $x$. Using a diagonal argument we can find a sub\sq\ $\F_\circ$ which works for all functions from a dense (in the supremum norm) countable subset of $C(X)$. This implies that $\mu$-almost every point in $X$ is $\F_\circ$-generic for~$\mu$.
\end{proof}

\subsection{Subshifts, cylinders, the zero-coordinate partition}\label{s24} Let $\Lambda$ be a finite set called the \emph{alphabet}. The elements of $\Lambda$ are called \emph{symbols}. The space $\Lambda^G$, equipped with the (compact) product topology is called \emph{the symbolic space} and its elements $x=(x(g))_{g\in G}$ are referred to as \emph{symbolic elements} or just \emph{points} (in the symbolic space). A natural action $\sigma$ of $G$ on $\Lambda^G$, called the \emph{shift-action}, is defined as follows (recall that, by convention, $g(x)$ stands for $\sigma(g)(x)$):
\begin{equation}\label{shift}
\forall_{g,\,h\in G,\, x\in X}\ g(x)(h) = x(hg).
\end{equation}
The \ds\ $(\Lambda^G,G)$ is called \emph{the full shift over $\Lambda$}. By a \emph{subshift} we will mean any subsystem of the full shift, i.e.\ any closed \inv\ subset $X\subset\Lambda^G$, together with the (restricted to $X$) shift action of $G$. When dealing with subshifts we will never consider actions of $G$ other than the shift-action. For this reason we will often denote the subshift just by the letter $X$ (instead of $(X,\sigma)$ or $(X,G)$).

By a \emph{block} we will mean any element $B\in\Lambda^K$, where $K$ is a finite subset of $G$ called the \emph{domain} of the block. We will sometimes use \emph{equality of blocks modulo shift}\footnote{For example, when $G=\N$, the block $110$ at coordinates 1,2,3 equals the block $110$ at coordinates 7,8,9 modulo shift.} defined as follows. Let $B\in\Lambda^K$ and $C\in\Lambda^{Kg}$ for some $g\in G$. We will write $B\approx C$ if 
$$
\forall_{h\in K}\ C(hg) = B(h). 
$$
In particular, if $D\in\Lambda^F$, where $F$ is any subset of $G$ (including the case $F=G$ in which $D$ becomes $x\in\Lambda^G$) we will say that the block $B$ \emph{occurs in $D$ with an anchor at $g\in G$} if $Kg\subset F$ and $D|_{Kg}\approx B$, where $D|_{Kg}$ denotes the restriction of $D$ to $Kg$ and
by $D|_{Kg}\approx B$, we mean that 
$$
\forall_{h\in K}\ D(hg) = B(h).\footnote{For instance, the block $B\in\{0,1\}^{\{3,5,11\}}$, with values $B(3)=1$, $B(5)=1$ and $B(11)=0$, occurs in some $x\in\{0,1\}^\N$ with an anchor at 2 whenever $x(5)=1,\ x(7)=1,\ x(13)=0$.}
$$

With each block $B\in\Lambda^K$ we associate the \emph{cylinder set}
$$
[B] = \{x: x|_K = B\}.
$$
Each cylinder is closed and open (we will say \emph{clopen}) in the symbolic space.
The condition that $B$ occurs in $x$ with an anchor at $g$ is equivalent to $g(x)\in[B]$. 
Given a finite set $K\subset G$, the cylinders associated to all blocks $B\in\Lambda^K$ partition the symbolic space into clopen sets. Whenever this does not lead to a confusion, we will abuse the notation and skip the brackets in the denotation of cylinders, identifying blocks with the associated cylinders. With this convention, the above mentioned partition will be simply denoted by $\Lambda^K$. If $K=\{e\}$, the corresponding partition will be called the \emph{partition by symbols} or the \emph{zero-coordinate partition}\footnote{This terminology bears apparent hallmarks of $\N_0$-actions, but we decided to use it anyway.}, and be denoted by $\Lambda$. 
If $X$ is a subshift, a cylinder (associated to a block $B$)  will be understood as the intersection of $[B]$ with $X$. In such a case, we will often restrict our attention to blocks $B$ for which this intersection is nonempty, which is equivalent to the condition that $B$ \emph{occurs in $X$} (meaning that it occurs anchored at some coordinate in some element of $X$).

Cylinders play an important role in both the topological and measurable structure of the symbolic space (and of any subshift); they form a base for the topology and they generate the Borel sigma-algebra. The indicator functions of cylinders are linearly dense in $C(X)$, the Banach space of all (real or complex) continuous functions on $X$, equipped with the topology of uniform convergence. Thus, the weak* convergence of measures can be verified by testing it only on cylinders. In particular, a point $x\in\Lambda^G$ is generic for an \im\ $\mu\in\mgl$ if for each finite set $K\subset G$ (it suffices to take finite sets containing the unit, if there is one), and each block $B\in\Lambda^K$, we have
\begin{equation}\label{gener}
d_{\F}(\{g\in F_n: y|_{Kg}\approx B\})= \mu(B).
\end{equation}

\subsection{Entropy}\label{2.5} Let now $(X,\Sigma,\mu)$ be a standard probability space\footnote{A probability space $(X,\Sigma,\mu)$ is standard if it is measure-theoretically isomorphic to $([0,1],\bar\B,m)$, where $[0,1]$ is the unit interval, $m$ is a convex combination of the Lebesgue measure with a purely atomic probability measure and $\bar\B$ is the Borel sigma-algebra completed with respect to $m$. Every compact metric space with a completed Borel probablity measure is a standard probability space.} on which $G$ acts by measure-preserving transformations. Let $\P$ be a finite measurable partition of $X$. The \emph{Shannon entropy} of $\P$ (with respect to $\mu$) is the quantity
$$
H(\mu,\P)= -\sum_{B\in\P} \mu(B)\log(\mu(B))
$$
(throughout, $\log$ stands for  $\log_2$). Next, the \emph{dynamical entropy} of $\P$ is defined as
$$
h(\mu,\P) = \lim_{n\to\infty} \frac1{|F_n|}H(\mu,\P^{F_n}),
$$
where $\P^{F_n}=\bigvee_{g\in F_n}g^{-1}(\P)$. (Note that $g^{-1}(\P) = \{g^{-1}(B): B\in\P\}$ is again a finite measurable partition of $X$, and so is $\P^{F_n}$.) By the general theory of entropy in amenable semigroup actions, the above limit exists, does not depend on the F\o lner \sq, and in fact equals
\begin{equation}\label{dfr}
\inf_F \frac1{|F|}H(\mu,\P^{F}),
\end{equation}
where $F$ ranges over all finite subsets of $G$ (this \emph{infimum rule} can be found, e.g., in \cite{DFR}). Finally the (\emph{Kolmogorov--Sinai}) \emph{entropy} of the action $(X,\Sigma,\mu, G)$ is defined as follows:
$$
h(X,\Sigma,\mu,G) = \sup_\P h(\mu,\P),
$$
where $\P$ ranges over all finite measurable partitions of $G$. If there is no risk of ambiguity, we will abbreviate $h(X,\Sigma,\mu,G)$ as $h(\mu)$. By
\eqref{dfr}, $h(\mu)$ does not depend on the choice of the F\o lner \sq. A finite measurable partition $\P$ of $X$ is a \emph{generator} if the smallest sigma-algebra containing all partitions $g^{-1}(\P)$ ($g\in G$) coincides with $\Sigma$. The general form of the Kolmogorov--Sinai Theorem asserts that if $\P$ is a generator then  $h(\mu) = h(\mu,\P)$. In particular, if $X$ is a subshift, $\mu$ is an \im\ on $X$ and $\P=\Lambda$ is the zero-coordinate partition then $h(X,\Sigma,\mu,G) = h(\mu,\Lambda)$.

In this paper, we will also need the notion of \tl\ entropy. We will find it convenient to use the definition of topological entropy based on the Variational Principle (which is valid for actions of countable amenable semigroups). Assume that $X$ is a compact metric space on which $G$ acts by continuous transformations. We define the \emph{\tl\ entropy} of the \ds\ $(X,G)$ as
\begin{equation}\label{vp}
\htop(X,G)= \sup_{\mu\in\mgx}h(\mu).
\end{equation}
Note that if $\mu\in\mgx$ and $\xi$ is the (unique) probability measure on $\megx$ such that $\mu=\int\nu\,d\xi(\nu)$ then $h(\mu)=\int h(\nu)\,d\xi(\nu)$.
This implies that the range of the supremum in \eqref{vp} can be restricted to the set of $\megx$ of ergodic measures. This ``ergodic version'' of the variational principle will be utilized in Sections~\ref{determinism}~and~\ref{nine}.

A \emph{measure of maximal entropy} is a measure $\mu\in\mgx$ such that $h(\mu)=\htop(X,G)$. In a general dynamical system such a measure need not exist. However, if $X=\Lambda^G$ is the full shift, then it supports a distinctive invariant (and in fact ergodic) measure, which assigns to each cylinder $B\in\Lambda^K$ the value $|\Lambda|^{-|K|}$. It is called the \emph{uniform Bernoulli measure} and we will denote it by $\lambda$; it is the unique measure of maximal entropy (equal to $\log|\Lambda|)$ on $X$. This measure will play an important role in this note.

\subsection{Normality}\label{normality}
Let us fix a finite alphabet~$\Lambda$ and a F\o lner \sq\ $\F=(F_n)_{n\in\N}$ in $G$.

\begin{defn}\label{normal} A point $y\in\Lambda^G$ is \emph{$\F$-normal} if it is $\F$-generic for the uniform Bernoulli measure $\lambda$.
\end{defn}

By \eqref{gener}, the above definition has the following equivalent formulation in terms of densities: for every finite set $K\subset G$ and every block $B\in\Lambda^K$, we have
\begin{equation}\label{mol}
d_{\F}(\{g\in F_n: y|_{Kg}\approx B\})= |\Lambda|^{-|K|}.
\end{equation}

\begin{rem}
Theorem 4.2 in \cite{BDM} asserts that if the cardinalities of the sets $F_n$ strictly increase (which is a very weak and natural restriction) then $\lambda$-almost every point is $\F$-normal.
\end{rem}

\smallskip
In the special case of $G=\N$ and $F_n=\{1,2,\dots,n\}$, $\F$-normality will be called \emph{classical normality}. Let us say that $W$ is an \emph{initial block} if $W\in\Lambda^{\{0,1,\dots,m-1\}}$ for some $m\in\N$. Since cylinders associated to initial blocks generate the Borel sigma-algebra in $\Lambda^{\N}$, in order to determine the normality of a \sq\ $x\in\Lambda^\N$ it suffices to verify \eqref{mol} for initial blocks. The following statement summarizes the above discussion:

\begin{fact}\label{class}
A \sq\ $y\in\Lambda^{\N}$ is classically normal if and only if for every $m\in\N$ and every initial block $W\in\Lambda^{\{0,1,\dots,m-1\}}$, we have
$$
\lim_{n\to\infty} \frac1n |\{i=1,2,\dots,n: y|_{\{i,i+1,\dots,i+m-1\}}\approx W\}| = |\Lambda|^{-m}.
$$
\end{fact}

A set $A\subset G$, such that the indicator function $y=\mathbbm 1_A\in\{0,1\}^G$ is $\F$-normal, will be called an \emph{$\F$-normal set}.

\section{Normality-preserving subsets of amenable semigroups}

In this section we will define the property of an infinite set $A\subset G$ to be ``$\F$-normality preserving''. To do this, we first need to define $\F$-normality (of a symbolic element $y\in\Lambda^G$) along a set $A$, and there are several different ways to do so. As we will show later, all of them lead to equivalent notions of $\F$-normality preservation.
%Later, once $\F$-normality along $A$ is defined (see below for several ways how this can be done), we will define an $\F$-normality-preserving set $A\subset G$ as a set such that every $y\in\Lambda^G$ which is $\F$-normal is also $\F$-normal along $A$ (see Defintion \ref{np} below). 

\subsection{The classical case}\label{clacs}
To find inspiration for a definition of $\F$-normality along $A$, we will first take a closer look at the classical case of $G=\N$ and $F_n=\{1,2,\dots,n\}$. We choose an infinite subset $A\subset\N$ and we enumerate it increasingly: $A=\{a_1,a_2,\dots\}$. This allows us to view $A$ as a sub\sq\ of $\N$. The simplest (and weakest) notion of normality is the following:

\begin{defn}[Simple normality along $A$]\label{snA} A \sq\ $y\in\Lambda^\N$ is \emph{simply normal along $A=\{a_1,a_2,\dots\}$} (or just \emph{simply normal}, if $A=\N$) if, for every $b\in\Lambda$,
$$
\lim_{n\to\infty} \frac1n |\{k=1,2,\dots,n: y(a_k)=b\}| = \frac1{|\Lambda|}.
$$
\end{defn}

Simple normality along a set is a very weak notion. For example when $A=\N$, simple normality of $y\in\{0,1\}^\N$ means merely that $y$ consists of 50\% zeros and 50\% ones. 

Definition \ref{snA} naturally leads to the notion of normality along a set $A\subset\N$. It is this notion which is present in both Wall's theorem (\cite{Wa}) and in Kamae--Weiss theorem (\cite{W1}, \cite{Kamae}), alluded to in the Introduction (see also Theorem \ref{KW} in section~\ref{determinism}). 

\begin{defn}[Normality along $A\subset\N$]\label{onA} A \sq\ $y\in\Lambda^\N$ is \emph{normal along $A$} if the \sq\ $(y(a_k))_{k\in\N}$ generates $\lambda$, meaning that for every $m\in\N$ and every initial block $W\in\Lambda^{\{0,1,\dots,m-1\}}$, we have
\begin{multline*}
\lim_{n\to\infty} \frac1n |\{k=1,2,\dots,n: \\y(a_k)=W(0), y(a_{k+1})=W(1),\dots,y(a_{k+m-1})=W(m-1)\}|=|\Lambda|^{-m}.
\end{multline*}
\end{defn}

%Joe: It seems we are somewhat abusing notation above, as the $K$-``shape" of $a_k, a_{k+1},\dots, a_{k+m-1}$ need not be the same as $W$. In other words, $\{a_k,a_{k+1},\dots, a_{k+m-1}\}$ is not an element of $\Lambda^{\{0,1,\dots,m-1\}g}$ for some $g$.
%Tom: This can be fixed only in one way: writing equalities term by term. We did so, but it created a "bad box". But let's not worry and let the editors handle the bad box.

The problem with generalizing Definition \ref{onA} to other semigroups is that it uses very special properties of the semigroup $(\N,+)$, such as orderability and the fact that the initial order-intervals (i.e.\ sets $\{i: 1\le i\le n\}$) form a F\o lner \sq. Moreover, this definition crucially uses the fact that any infinite subset $A\subset\N$ is order-isomorphic to $\N$. Notice that already in $\N^2$ (which is orderable by the lexicographical order), the initial order-intervals do not form a F\o lner \sq.\footnote{With respect to the lexicographical order, the initial intervals in $\N^2$ have the form 
$$
\{(i,j):(1,1)\le(i,j)\le(m,n)\}=([1,m-1]\times\N)\cup(\{m\}\times[1,n]). 
$$
Unless $m=1$, these sets are infinite, so they cannot be elements of any F\o lner \sq\ (and those with $m=1$ clearly do not have the asymptotic invariance property).}

To cope with this problem, we introduce two new notions of normality along a set $A\subset\N$, which can (and will) be generalized to countable cancellative amenable semigroups, and which retain the spirit of the classical notion.

\begin{defn}[Orbit-normality along $A$]\label{nA} 
A \sq\ $y\in\Lambda^\N$ is \emph{orbit-normal along $A$} if the \sq\ $(\sigma^{a_k}(y))_{k\in\N}$ generates the uniform Bernoulli measure $\lambda$ in the following sense:
$$
\lim_{n\to\infty}\frac1n\sum_{k=1}^n\delta_{\sigma^{a_k}(y)}=\lambda
$$
(with respect to the weak* convergence).
\end{defn}
That is to say, for every $m\in\N$ and every initial block $W\in\Lambda^{\{0,1,2,\dots,m-1\}}$, we have
$$
\lim_{n\to\infty} \frac1n |\{k=1,2,\dots,n: y|_{\{a_k+1,a_k+2,\dots,a_k+m\}}\approx W\}| = |\Lambda|^{-m}.
$$ 
\smallskip

Recall that $A_K$ denotes the $K$-core of $A$ (see Definition \ref{core}). In case $G=\N$,
we have $A_K=\{k\in\N:K+k\subset A\}$.

\begin{defn}[Block-normality along $A$]\label{bnA} A \sq\ $y\in\Lambda^\N$ is \emph{block-normal along $A$} if for every finite set $K\subset\N$ such that 
\begin{equation}\label{tut}
\liminf_{n\to\infty}\frac{|\{1,2,\dots,n\}\cap A_K\}|}{|\{1,2,\dots,n\}\cap A|}>0,
\end{equation}
and every block $B$ over $K$, we have
\begin{equation}\label{tu}
\lim_{n\to\infty} \frac{|\{k\in\{1,2,\dots,n\}\cap A_K, \  y|_{K+k}\approx B\}|}{|\{1,2,\dots,n\}\cap A_K|} = |\Lambda|^{-|K|}.
\end{equation}
\end{defn}
The requirement \eqref{tut} says, roughly speaking, that the growth of the $K$-core of $A$ is ``proportional'' to the growth of $A$ (in such a case we will say that $K$ is \emph{visible} in $A$). This allows us to interpret the limit in \eqref{tu} as the ratio of two ``relative densities'' in $A$:
$$
\lim_{n\to\infty} \frac{\hphantom{a}\frac{|\{k\in\{1,2,\dots,n\}\cap A_K, \  y|_{K+k}\approx B\}|}{|\{1,2,\dots,n\}\cap A|}\hphantom{a}}{\frac{|\{1,2,\dots,n\}\cap A_K|}{|\{1,2,\dots,n\}\cap A|}} = |\Lambda|^{-|K|}.
$$
Note that the formula \eqref{tu} makes no sense for invisible sets $K$. As an extreme example, consider $A=2\N$ and $K=\{1,2\}$. Clearly, no shift of $K$ is contained in $A$, so the formula \eqref{tu} for $K$ is meaningless. An ostensible disadvantage of Definition \ref{bnA} is that there may be very few visible sets $K$. For example, if $A=\{n^2:n\in\N\}$ then the only visible sets are singletons and block-normality along $A$ reduces to simple normality along $A$. On the other hand, as we will see later, this definition will be useful in dealing with normality preservation along sets having positive lower density.
\smallskip

Note that in Definition \ref{onA} all the coordinates $a_k,a_{k+1},\dots,a_{k+m-1}$ belong to $A$, and while they do not necessarily form a ``connected interval'' (i.e.\ an interval of the form $\{a_k,a_k+1,\dots,a_k+m-1\}$) they do form a ``relatively connected'' subset of $A$, i.e.\ the intersection of $A$ with an interval. 

In Definition \ref{nA}, just like in the case of classical normality (see Proposition \ref{class}), we count occurrences of $W$ in $y$ along ``connected intervals'', where only the starting position $a_k$ (the anchor) must belong to $A$.

Finally, in Definition \ref{bnA}, we count occurrences of $B$ in $y$, whose domains $K+k$ are entirely contained in $A$ and are not necessarily ``relatively connected'' in $A$. 

\smallskip
One can show that the above four notions of normality along $A$ are not equivalent (yet all of them trivially imply simple normality along $A$).

\subsection{Generalizations to amenable semigroups}
The definitions \ref{snA}, \ref{nA} and \ref{bnA} can be extended to any infinite countable cancellative amenable semigroup $G$ with a fixed F\o lner \sq\ $\F=(F_n)_{n\in\N}$. Assume that $A$ is a subset of $G$ satisfying the following condition (effective throughout the rest of the paper):
\begin{equation}\label{infi}
\lim_{n\to\infty}|F_n\cap A|=\infty.\footnote{In the classical case of $G=\N$ and $F_n=\{1,2,\dots,n\}$, the condition \eqref{infi} is fulfilled for any infinite set $A$.}
\end{equation}
As before, $\lambda$ denotes the uniform Bernoulli measure on $\Lambda^G$, where $\Lambda$ is a (fixed throughout the rest of the paper) finite alphabet.

\begin{defn}[Simple $\F$-normality along a set $A$]\label{GSnA}
An element $y\in\Lambda^G$ is \emph{simply $\F$-normal along $A$} if, for each $a\in\Lambda$,
$$
\lim_{n\to\infty} \frac{|\{g\in (F_n\cap A): y(g)=a\}|}{|F_n\cap A|} = \frac1{|\Lambda|}.
$$
\end{defn}

\begin{defn}[Orbit-$\F$-normality along a set $A$]\label{GnA} 
An element $y\in\Lambda^G$ is \emph{orbit-$\F$-normal along $A$} if the set $\{g(y):g\in A\}$ $\F$-generates the uniform Bernoulli measure $\lambda$ in the following sense:
$$
\lim_{n\to\infty}\frac1{|F_n\cap A|}\sum_{g\in F_n\cap A}\delta_{g(y)}=\lambda
$$
(with respect to the weak* convergence).
\end{defn}

\begin{defn}[Block-$\F$-normality along a set $A$]\label{GbnA} An element $y\in\Lambda^G$ is \emph{block-$\F$-normal along $A$} if for every finite set $K\subset G$ which is \emph{$\F$-visible in $A$}, i.e.\ such that 
$$ 
\liminf_{n\to\infty} \frac{|F_n\cap A_K|}{|F_n\cap A|} >0,
$$
and every block $B\in\{0,1\}^K$, we have
$$
\lim_{n\to\infty} \frac{|\{g\in F_n\cap A_K, \  y|_{Kg}\approx B\}|}{|F_n\cap A_K|} = |\Lambda|^{-K}.
$$
\end{defn}

As we have mentioned at the end of subsection \ref{clacs}, the above notions of $\F$-normality along a set are not mutually equivalent, even in the classical case. 

We are now in a position to formulate the definition of a normality preserving~set.
\begin{defn}\label{np}{\color{white}.}
\begin{enumerate}[(i)]
	\item A set $A\subset\N$ \emph{preserves normality} if for every 
	$y\in\Lambda^\N$ which is classically normal, $y$ is also normal along $A$ (in the sense 
	of Definition \ref{onA}).
	\item Let $\F=(\F_n)_{n\in\N}$ denote a F\o lner \sq\ in a countable cancellative amenable  semigroup $G$. A set $A\subset G$ \emph{preserves} simple $\F$-normality (respectively, orbit-$\F$-normality, block-$\F$-normality) if for every $y\in\Lambda^G$ which is $\F$-normal, $y$ is also simply $\F$-normal (respectively, orbit-$\F$-normal, block-$\F$-normal) along~$A$. 
\end{enumerate}
\end{defn}

\section{Determinism}\label{determinism} In this section we introduce and discuss two natural notions of determinism for real bounded functions defined on countable cancellative amenable semigroups. 

\begin{defn}\label{std} Let $G$ be a countable cancellative amenable semigroup. A function $f:G\to\R$ will be called \emph{strongly} \emph{deterministic} if there exists a \ds\ $(X,G)$ of \tl\ entropy zero, a continuous function $\varphi:X\to\R$ and a point $x\in X$ such that $f(g)=\varphi(g(x))$ ($g\in G$).
\end{defn}

For $G=\N$, the above notion (under the name ``deterministic'') has recently enjoyed  popularity in the works related to Sarnak's famous \emph{M\"obius disjointness conjecture} (see \cite{Sa}). The following theorem provides a useful characterization of strongly deterministic functions. %We remark that although until now we have worked with the shift action on $\Lambda^G$, where $\Lambda$ is a finite set, the same definition applies to $Y^G$ where $Y$ is some compact metric space. In the theorem below, we will be working with the shift action on $I^G$, where $I$ is a closed bounded interval in $\R$.

\begin{thm}\label{banal}
A function $f:G\to\R$ is strongly deterministic if and only if it is bounded and the orbit closure of $f$ under the shift action of $G$ on $I^G$, where $I$ is the closure of the range of $f$, has \tl\ entropy zero.\footnote{
In total analogy with the formula \eqref{shift}, the shift action of $G$ on $I^G$ is defined as follows: if $g\in G$ and $x\in I^G$, then $g(x) = y, \text{ \ where \ }\forall_{h\in G}\ y(h) = x(hg)$.} 
\end{thm}

\begin{proof}If $f$ is strongly deterministic, then we have at our disposal a zero entropy system $(X,G)$, a continuous function $\varphi:X\to\R$ and a point $x\in X$ such that $f(g)=\varphi(g(x))$. Observe that $x\mapsto(\varphi(g(x)))_{g\in G}$ is a continuous map from $X$ to $I^G$ (henceforth denoted by $\Phi:X\to I^G$). Moreover, $\Phi$ is a \tl\ factor map from $(X,G)$ to the shift-action of $G$ on the image $\Phi(X)$, i.e.\ for any $h\in G$ and $x\in X$, the following diagram is commutative. 
$$
\begin{matrix}
&x&\overset{\Phi}\longrightarrow&(\varphi(g(x)))_{g\in G}\\ \vspace{-10pt} \\
\text{action by\!\!\!\!}&h\downarrow\ \ \ &&\downarrow \text{shift by }h\\ \vspace{-10pt} \\
&h(x)&\overset{\Phi}\longrightarrow&(\varphi(gh(x)))_{g\in G}
\vspace{8pt}
\end{matrix}
$$
Finally, by Definition \ref{std}, $f$ is the image by $\Phi$ of some $x\in X$. Thus the shift-orbit closure of $f$ is a \tl\ factor of the orbit closure of $x$. Because entropy zero is inherited by subsystems and factors, the proof of the implication in one of the direction is finished.

Now suppose that $f$ is bounded and the action of $G$ restricted to the (compact) orbit closure $X_f=\overline{\{g(f):g\in G\}}$ has entropy zero. Define $\varphi:X_f\to\R$ by the formula $\varphi(x) = x(e)$. (This is one of the places where it is convenient to assume that $e\in G$.) It remains to observe that $f\in X_f$ and, for each $g\in G$, we have $\varphi(g(f))=f(g)$. This proves that $f$ is strongly deterministic.
\end{proof}

\begin{thm}\label{vit1}
Let \xg\ be a \ds\ with \tl\ entropy zero. Let $\varphi:X\to\R$ be a bounded function and let $D_\varphi$ denote the set of discontinuity points of $\varphi$. If $\mu(D_\varphi)=0$ for all $\mu\in\mgx$, then, for any $x\in X$, the function $f:G\to\R$ given by $f(g)=\varphi(g(x))$ is strongly deterministic.
\end{thm}

\begin{proof}
With each point $x\in X$ we associate the element $\Phi(x)=(\varphi(g(x)))_{g\in G}\in I^G$, where $I$ is the closure of the range of $\varphi$. We now create a \tl\ extension $(\bar X,G)$ of $(X,G)$, where $\bar X=\overline{\{(x,\Phi(x))\in X\times I^G\}}$ and $g(x,y)=(g(x),g(y))$ (on the second coordinate we apply the shift as it is defined on $I^G$). The factor map from $\bar X$ to $X$ is simply the projection on the first axis. Note that every point $x\in X$ whose orbit never visits $D_\varphi$ has a one-element preimage in this extension, namely the pair $(x,\Phi(x))$. Since $D_\varphi$ has measure zero for all invariant measures on $X$, the projection $\pi_1$ to the first coordinate is a measure-theoretic isomorphism between the measure-preserving systems $(\bar X,\nu,G)$ and $(X,\pi_1(\nu),G)$ for each $\nu\in\M_G(\bar X)$. Since $\htop(X,G)=0$, we have by the variational principle that $h(\mu)=0$ for all $\mu\in\mgx$. Thus $h(\nu)=0$ for all $\nu\in\M_G(\bar X)$. Applying the variational principle one more time, we conclude that $\htop(\bar X,G)=0$ as well. Now, define $\bar\varphi:\bar X\to\R$ by $\bar\varphi((x,z))=z(e)$.  (Again, it is convenient to assume that $G$ has a unit.) This is clearly a continuous function on $\bar X$. For each $x\in X$, the pair $(x,\Phi(x))$ is an element of $\bar X$ and thus the function $\bar f:G\to\R$ given by $\bar f(g)=\bar\varphi(g(x,\Phi(x)))$ is strongly deterministic. Finally note that
$$
\bar f(g)=\bar\varphi(g(x,\Phi(x)))=g(\Phi(x))(e)=\Phi(x)(g)=\varphi(g(x))=f(g),
$$ 
so $\bar f\equiv f$ and hence $f$ is strongly deterministic. 
\end{proof}

When $f$ is a strongly deterministic $\{0,1\}$-valued function, we will call the set $\{g\in G: f(g)=1\}$ a \emph{strongly deterministic set}. The following theorem provides a useful dynamical characterization of strongly deterministic sets via visiting times of an orbit of a point to a set with ``small boundary'' in a zero entropy system:

\begin{thm}\label{vit} Let $A$ be a subset of $G$. The following conditions are equivalent:
\begin{enumerate} 
\item $A$ is strongly deterministic.
\item There exists a zero entropy system $(X,G)$, a set $C\subset X$ such that $\mu(\partial C)=0$ for every \im\ $\mu\in\mgx$ ($\partial C$ denotes the boundary of $C$) and a point $x_0\in X$ such that
$$
A=\{g\in G: g(x_0)\in C\}.
$$
\item There exists a zero entropy system $(X,G)$, a clopen set $C\subset X$ and a point $x_0\in X$ such that
$$
A=\{g\in G: g(x_0)\in C\}.
$$
\end{enumerate}
\end{thm}

\begin{proof}
If $A\subset G$ is strongly deterministic then, by Theorem \ref{banal}, the action of the shift on the orbit closure $Y$ of $\mathbbm 1_A$ has \tl\ entropy zero. Putting $x_0=\mathbbm 1_A$, we have $A=\{g\in G: g(x_0)\in[1]\}$. Since the cylinder $[1]$ is clopen in $Y$, (1) implies (3), which trivially implies (2). The implication (2)$\implies$(1) is a special case of Theorem \ref{vit1} because the indicator function of $C$ is discontinuous only on $\partial C$. 
\end{proof}

\begin{rem}\label{dd}
(i) It is easy to see from the definition of upper and lower Banach densities that the set $A$ has positive lower (resp.\ upper) Banach density if and only if $\inf_\mu\mu(C)>0$ (resp. $\sup_\mu\mu(C)>0$), where $\mu$ ranges over all \im\ supported by the orbit closure of $\mathbbm 1_A$ and $C$ is the set appearing in (2) or (3). 

\noindent(ii) Sets $A$ with Banach density 0 or 1 are ``trivially deterministic''. The orbit closure in $\{0,1\}^G$ of the indicator function of such a set carries a unique invariant measure which is the point mass concentrated at a fixpoint
(either the constant 0 function or the constant 1 function).
\end{rem}

We will introduce now a weaker notion of determinism which will play a crucial role in this paper.

\begin{defn}Let $G$ be a countable cancellative amenable semigroup in which we fix a F\o lner \sq\ $\F$. Let $I\in\R$ be a compact set. A function $f:G\to I$ will be called \emph{$\F$-deterministic} if for every \im\ $\mu$ on $I^G$, for which $f$ is an $\F$-quasi-generic point under the shift action of $G$, the system $(I^G,\mu,G)$ has Kolmogorov--Sinai entropy zero. A set $A\subset G$ will be called  \emph{$\F$-deterministic} if the indicator function $\mathbbm 1_A$ is $\F$-deterministic.
\end{defn}

When $G=\N$ and $F_n=\{1,2,\dots,n\}$, the above notion of an $\F$-deterministic set is the same as that of a ``completely deterministic set'' in \cite{W2}. We will retain this terminology when dealing with the classical case. Now that the notions of normality preservation and complete determinism have been introduced, we can give a precise formulation of the Kamae--Weiss Theorem alluded to in the introduction.
\begin{thm}[Kamae--Weiss Theorem]\label{KW}
Let $A\subset\N$ be a set of positive lower density. Then $A$ preserves normality if and only if $A$ is completely deterministic. 
\end{thm}

Note that strongly deterministic functions are $\F$-deterministic for any F\o lner \sq\ $\F$ in $G$. This follows from Theorem \ref{banal} and from the fact that, by the variational principle, all \im s supported by a system with \tl\ entropy zero have Kolmogorov--Sinai entropy zero. On the other hand, the notion of $\F$-determinism is essentially weaker than that of strong determinism. For example, consider a minimal symbolic system $X\subset\Lambda^G$ with positive \tl\ entropy, such that for some ergodic measure $\mu$ the measure-preserving system $(X,\mu,G)$ has Kolmogorov--Sinai entropy zero.\footnote{
For the proof of existence of such systems, when $G$ is a countable amenable group, see \cite{FH}. This result can be generalized to countable cancellative amenable semigroups via the method of natural extensions.
} 
By Proposition \ref{miner}, $\mu$-almost every point is $\F$-generic, for some F\o lner \sq\ $\F$. Now, every such point (viewed as a function from $G$ to $\Lambda$) is obviously $\F$-deterministic, while, since $(X,G)$ is minimal and has positive \tl\ entropy, no point of $X$ is strongly deterministic. For another proof of the existence of $\F$-deterministic but not strongly deterministic functions see Propositon~\ref{trivia}.

A well-known (and quite natural) example of a completely, but not strongly, deterministic set is provided by the set $S\subset\N$ of square-free numbers. Let $\mathbbm 1_S =\omega\in\{0,1\}^\N$ and let $X_S=\overline{\{\sigma^n(\omega),n\in\N\}}\subset\{0,1\}^\N$ (where $\sigma$ denotes the shift transformation on $\{0,1\}^\N$). One can show (see \cite[Theorem 8]{Sa}) that the \tl\ system $(X_S,\sigma)$ has positive \tl\ entropy. On the other hand, the point $\omega$ is generic for the so-called Mirsky measure $\nu$ (introduced and studied by L. Mirsky \cite{Mi}), which has the property that the measure-preserving system $(X_S,\nu,\sigma)$ has rational discrete spectrum (and hence has zero Kolmogorov-Sinai entropy (see \cite[Theorem 9]{Sa}).

More examples of deterministic sets (of both types) will be provided in Section~\ref{nine}. 
\smallskip

An important property of the class of $\F$-deterministic functions is that it is closed under uniform limits. In the theorem below $\F$ is a fixed F\o lner \sq\ in $G$.

\begin{thm}\label{ulodf}
For each $n\ge1$, let $f_n:G\to \R$ be an $\F$-deterministic function and assume that the \sq\ $(f_n)_{n\ge1}$ converges uniformly to a function $f_0:G\to\R$. Then $f_0$ is also $\F$-deterministic.
\end{thm}

\begin{proof}By definition, each strongly deterministic function is bounded, and hence $f_0$ is bounded as a uniform limit of bounded functions. Consider the (compact) product space $\mathfrak X=\prod_{n\ge 0}I_n^G$,  where, for each $n\ge 0$, $I_n\subset\R$ is the closure of the range of $f_n$. The semigroup $G$ acts on $\mathfrak X$ coordinatewise, by the shifts
$$
g((\phi_n)_{n\ge 0}) = (g(\phi_n))_{n\ge 0} =(\psi_n)_{n\ge 0}, \ \ g\in G,
$$  
where, for each $n\ge 0$, $\psi_n:G\to I_n$ is defined by $\psi_n(h)=\phi_n(hg)$, $h\in G$.
We will restrict our attention to the orbit closure $\mathfrak F=\overline{\{g((f_n)_{n\ge 0}):g\in G\}}$ of the element $(f_n)_{n\ge 0}\in\mathfrak X$. For $n\ge 0$, we will denote by $\mathfrak F_n$ the projection of $\mathfrak F$ onto the $n$th coordinate. Note that $\mathfrak F_n$ equals the orbit closure $\overline{\{g(f_n):g\in G\}}$ of $f_n\in I_n^G$. The projection $\mathfrak F_{\N}$ of $\mathfrak F$ onto $\prod_{n\ge1}I_n^G$ coincides with the orbit closure $\mathfrak F=\overline{\{g((f_n)_{n\ge 1}):g\in G\}}$ of the element $(f_n)_{n\ge1}$. Clearly, every \im\ $\mu$ on $\mathfrak F_\N$ for which this element is $\F$-quasi-generic is a joining\footnote{
In ergodic theory, a joining of a \sq\ of measure preserving systems $(X_n,\mu_n,G)$, $n\in\N$, is a system of the form $(X,\mu,G)$, where $X=\prod_{n\ge1}X_n$, the action of $G$ is coordinatewise, and $\mu$ is any \im\ on $X$ such that for every $n\ge1$ the marginal of $\mu$ on $X_n$ equals $\mu_n$.
} 
of measures $\mu_n$ on $I_n^G$ for which $f_n$ is $\F$-quasi-generic, $n\ge 1$. Because any joining of measures of entropy zero has entropy zero, we conclude that $h(\mu)=0$. We will show that $\mathfrak F_0$ is a \tl\ factor of $\mathfrak F_\N$ via a map which sends the element $(f_n)_{n\ge1}$ to $f_0$. This will imply that any measure $\mu_0$ for which $f_0$ is $\F$-quasi-generic is the image by a \tl\ factor map of a measure $\mu$ on $\mathfrak F_\N$ for which the element $(f_n)_{n\ge1}$ is $\F$-quasi-generic. Since we have already shown that any such measure $\mu$ has entropy zero, this will imply that $\mu_0$ also has entropy zero, and so, it will follow that $f_0$ is $\F$-deterministic.

It remains to build the desired factor map. For every element $(\varphi_n)_{n\ge 0}\in\mathfrak F$ there exists a \sq\ $(g_k)_{k\ge 1}$ in $G$ such that
$$
(\varphi_n)_{n\ge 0} = \lim_{k\to\infty} g_k((f_n)_{n\ge 0}),
$$
that is, for every $n\ge 0$ and $h\in G$, we have the convergence
$$
\varphi_n(h)=\lim_{k\to\infty}f_n(hg_k).
$$
We have assumed the uniform (with respect to $h\in G$) convergence $\lim_{n\to\infty} f_n(h)=f_0(h)$. So we can write
$$
\varphi_0(h)=\lim_{k\to\infty}f_0(hg_k)=\lim_{k\to\infty}\lim_{n\to\infty}f_n(hg_k),
$$
where the limit $\lim_{n\to\infty}f_n(hg_k)$ is uniform with respect to $h\in G$.
Since uniform limits commute with pointwise limits, we get
$$
\varphi_0(h)=\lim_{n\to\infty}\lim_{k\to\infty}f_n(hg_k)=\lim_{n\to\infty}\varphi_n(h),
$$
where $\lim_{n\to\infty}\varphi_n(h)$ is again uniform with respect to $h\in G$. The above equality has two consequences: \begin{enumerate}
\item each element $(\varphi_n)_{n\in\N}\in\mathfrak F_\N$ is a uniformly convergent \sq\ of functions, 
\item $\lim_{n\to\infty}\varphi_n$ belongs to $\mathfrak F_0$. 
\end{enumerate}
Thus, the assignment $(\varphi_n)_{n\in\N}\mapsto \lim_{n\to\infty}\varphi_n$ defines a map from $\mathfrak F_\N$ to $\mathfrak F_0$, continuous with respect to the product topology. Clearly, this map commutes with the shift action of $G$, and so it is a \tl\ factor map. The image of $(f_n)_{n\ge1}$ by this map is $f_0$, as desired.
\end{proof}

\begin{rem}\label{rr}
It is not difficult to see that sums and products of $\F$-deterministic functions are $\F$-deterministic. Indeed, let $f_1, f_2$ be $\F$-deterministic functions and let $I_1,I_2$ denote the closures of their ranges, respectively. Arguing as in the proof of Theorem \ref{ulodf}, we obtain that the pair $(f_1,f_2)$ is $\F$-deterministic in the system $(I_1\times I_2)^G$. The coordinatewise summation (resp., multiplication) is a \tl\ factor map from $I_1^G\times I_2^G$ to $(I_1+I_2)^G$ (resp., to $(I_1\cdot I_2)^G$) such that the pair $(f_1,f_2)$ is mapped to $f_1+f_2$ (resp., to $f_1f_2$), and hence $f_1+f_2$ (resp., $f_1f_2$) is $\F$-deterministic. Extending the scalars to complex numbers, we see that $\F$-deterministic functions on $G$ form a commutative $L^\infty$-algebra with involution, which by the Gelfand Representation Theorem is isomorphic to the algebra $C(X_{\F})$ of continuous functions on a (nonmetrizable) Hausdorff compactification $X_{\F}$ of $G$. This compactification has the following properties:
\begin{enumerate}
	\item The action of $G$ on itself by right multiplication extends to a continuous action on $X_{\F}$. 
	\item The elements of $X_{\F}$ which correspond to elements of $G$ are, in the system $(X_{\F},G)$,	$\F$-quasi-generic \emph{only} for measures of entropy zero. 
\end{enumerate}
The system $(X_{\F},G)$ is the maximal compactification of $G$ satisfying (1) and (2),
in the sense that any other compactification with these properties is a \tl\ factor of
$(X_{\F},G)$.
\end{rem}

The next theorem shows that, in a way, strong determinism is an ``extreme'' form of $\F$-determinism.

\begin{thm}\label{fsq}
A function $f:G\to I$, where $I\subset\R$ is a compact set, is strongly deterministic if and only if it is $\F$-deterministic for every F\o lner \sq\ $\F$ in $G$.
\end{thm}

The proof utilizes the following lemma:
\begin{lem}
Consider a \tl\ \ds\ \xg and a point $x\in X$. For any ergodic measure $\mu\in\mgx$ supported by the orbit closure of $x$, there exists a F\o lner \sq\ $\F$ in $G$ such that $x$ is $\F$-generic for $\mu$. 
\end{lem}

\begin{proof} By Proposition \ref{miner}, the measure $\mu$ has an $\F'$-generic point $y\in \overline{\{g(x):g\in G\}}$ for some F\o lner \sq\ $\F'=(F'_n)_{n\in\N}$. For each $n\in\N$, let $\eps_n$ be so small that 
$$
d(z,z')<\eps_n\implies
\dist\Bigl(\frac1{|F'_n|}\sum_{h\in F'_n}\delta_{h(z)},\frac1{|F'_n|}\sum_{h\in F'_n}\delta_{h(z')}\Bigr)<\frac1n
$$
(where $d$ is a metric on $X$ and $\dist$ is some metric on $\M(X)$ compatible with the weak* topology),
and let $g_n\in G$ be such that $d(g_n(x),y)<\eps_n$. We define $F_n=F_n'g_n$ and note that $\F=(F_n)_{n\in\N}$ is a F\o lner \sq\ in $G$. Now, 
$$
\lim_{n\to\infty} \frac1{|F_n|}\sum_{h\in F_n}\delta_{h(x)}=
\lim_{n\to\infty} \frac1{|F'_n|}\sum_{h\in F'_n}\delta_{hg_n(x)}=
\lim_{n\to\infty} \frac1{|F'_n|}\sum_{h\in F'_n}\delta_{h(y)}=\mu.
$$
We have shown that $x$ is $\F$-generic for $\mu$.
\end{proof}

\begin{proof}[Proof of Theorem \ref{fsq}]
Sufficiency has been already explained earlier in this section. For necessity, suppose that $f$ is not strongly deterministic. Lack of strong determinism means (by Theorem \ref{banal}) that the shift action of $G$ on the orbit closure $X_f$ of $f$ in $I^G$ has positive \tl\ entropy. By the ergodic version of the variational principle, there exists an ergodic measure $\mu$ supported by $X_f$, such that the measure-preserving system $(X_f,\mu,G)$ has positive Kolmogorov-Sinai entropy. By the above lemma, there exists a F\o lner \sq\ $\F$ such that $f$ is $\F$-generic for $\mu$. So $f$ is not $\F$-deterministic.
\end{proof}

Theorem \ref{fsq} combined with Remark \ref{rr} implies that the family of strongly deterministic functions is closed under finite sums and products. This, together with Theorem \ref{ulodf}, implies the following result.

\begin{cor}\label{uf}
The family of strongly deterministic functions is an algebra closed under uniform limits. 
\end{cor}

We conclude that complex-valued strongly deterministic functions on $G$ form a commutative $L^\infty$-algebra with involution, which, by the Gelfand Representation Theorem, is isomorphic to the algebra $C(X)$ of continuous functions on a (nonmetrizable) Hausdorff  compactification $X$ of $G$. The action of $G$ on itself by right translations extends to a continuous action of $G$ on $X$. The system $(X,G)$ is the universal zero entropy system in the sense that any other \tl\ system $(Y,G)$ with \tl\ entropy zero is a \tl\ factor of \xg. %(See also \cite{Ge}).} 
\smallskip

We will say that $f:G\to\R$ is a \emph{negligible function}, if, for every F\o lner \sq\ $\F=(F_n)_{n\in\N}$ in $G$,
\begin{equation}\label{null}
\lim_{n\to\infty}\frac 1{|F_n|}\sum_{g\in F_n}|f(g)|=0.
\end{equation}

\begin{cor}\label{coc} 
Every negligible function is strongly deterministic.
\end{cor}
\begin{proof}
The condition \eqref{null} implies that for any F\o lner \sq\ $\F$, the element $f\in I^G$ is $\F$-generic for the atomic measure concentrated at the constant zero function. Clearly any atomic measure concentrated at a fixpoint has Kolmogrov--Sinai entropy zero, so $f$ is $\F$-deterministic and Theorem \ref{fsq} applies. 
\end{proof}

We end this section with two examples (of classes) of strongly deterministic functions. The first example concerns the group $\Z^d$ and appears naturally in the theory of multiple recurrence.

\begin{exam}
Let \xmt\ be an invertible measure-preserving system with a probability measure. Let $k\in\N$ and let $f_0,f_1\dots,f_k$ be in $L^\infty(\mu)$. Then, for any $d\in\N$ and any polynomials $p_1,p_2,\dots,p_k:\Z^d\to\Z$, the function
$$
\varphi(n) = \int f_0\cdot T^{p_1(n)}f_1\cdot\ldots\cdot T^{p_k(n)}f_k \,d\mu,\ \ n\in\Z^d,
$$
is strongly deterministic on $\Z^d$. Indeed, it was shown in \cite[Theorem 6.2]{Le2} that $\varphi$ can be represented as $\varphi=\varphi_1+\varphi_2$, where $\varphi_1$ is a uniform limit of basic nilfunctions (i.e.\ functions of the form $\varphi(n)=\psi(T^n(x))$, where $T^n$ is a niltranslation on a compact nilmanifold $M$, $x\in M$ and $\psi\in C(M)$), and $\varphi_2$ is a negligible function. It is well known that niltranslations are distal and hence, by a general form of a theorem due to Parry, they have \tl\ entropy zero (see Proposition \ref{distal} in section \ref{sdistal}). Theorem \ref{banal} thus implies that basic nilfunctions are strongly deterministic. Now, Corollary \ref{uf} implies that the function $\varphi_1$ is strongly deterministic. By Corollary \ref{coc}, $\varphi_2$ is also strongly deterministic, so $\varphi$, being a sum of two strongly deterministic functions, is strongly deterministic. 
\end{exam}

The following family of deterministic functions has its roots in the classical topological dynamics. 

\begin{exam} Let $\tau:G\to C(X,X)$ be an action of a semigroup $G$ on a compact metric space $X$. The \emph{Ellis semigroup} $E(X,G)$ of \xg\ is the closure of the family $\{\tau(g):g\in G\}$ in the product topology of $X^X$. A \ds\ \xg\ is called \emph{weakly almost periodic} if all elements of $E(X,T)$ are continuous maps. A function $f:G\to\R$ is called \emph{weakly almost periodic} if there exists a weakly almost periodic system \xg, a continuous function $\varphi:X\to\R$ and a point $x\in X$ such that $f(g)=\varphi(g(x))$, $g\in G$. Weakly almost periodic functions are characterized by the following property: $f$ is bounded and
$$
\lim_{n\to\infty} \lim_{k\to\infty} f(g_nh_k) = \lim_{k\to\infty} \lim_{n\to\infty} f(g_nh_k)
$$
whenever $(g_n)_{n\in\N}$ and $(h_k)_{k\in\N}$ are sequences in $G$ and all limits involved exist (see \cite{BH}). A function $f:G\to\R$ is \emph{almost periodic} if the family $\{h(f):h\in G\}$ (where $h(f)$ is given by $h(f)(g)=f(gh)$, $g\in G$) is precompact in the topology of uniform convergence on $\R^G$. It follows from the definition that almost periodic functions on $G$ have the form $f(g)=\varphi(R_g(x))$, where $g\mapsto R_g$ is an equicontinous action on a compact metric space, and hence they are strongly deterministic.
\begin{prop}
Let $G$ be a countable cancellative amenable semigroup. Any weakly almost periodic function $F:G\to\R$ is strongly deterministic.
\end{prop}
\begin{proof}
By the deLeeuw--Glicksberg decomposition \cite[Thm.\ 5.7 and Cor.\ 5.9]{DG}\footnote{The existence of an invariant mean required in \cite[Thm.\ 5.7]{DG} follows from the amenability of $G$. Two years after \cite{DG}, C.\ Ryll-Nardzewski \cite{RN} proved that an invariant mean on the space of weakly almost periodic functions exists for any group.}, $f=f_1+f_2$, where $f_1$ is almost periodic and $f_2$ is a negligible function (i.e.\ $f_2$ satisfies \eqref{null}).   As mentioned above, $f_1$ is strongly deterministic. The function $f_2$ is strongly deterministic by Corollary \ref{null}, and so $f$ is strongly deterministic as a sum of two strongly deterministic functions.
\end{proof}
\end{exam}
We note in passing that for some groups (the so-called minimally almost periodic groups) the space of weakly almost periodic functions reduces (modulo negligible functions) to constants. For example, the group of finite even permutations of $\N$ is such a group.

\section{First main result: determinism = normality preservation}

We will now state the first main result of this paper, which is analogous to the Kamae--Weiss characterization of normality-preserving sets in $\N$ (Theorem \ref{KW}). Recall that a set $A\subset G$ is \emph{$\F$-deterministic} if its indicator function $y=\mathbbm 1_A\in\{0,1\}^G$ is $\F$-deterministic. Recall also that the definitions of simple, orbit- and block-$\F$-normality preservation apply only to sets $A\subset G$ such that $|F_n\cap A|\to\infty$ (condition~\eqref{infi}).

\begin{thm}\label{main}Let $G$ be a countable cancellative amenable semigroup in which we fix an arbitrary F\o lner \sq\ $\F=(F_n)_{n\in\N}$. Let $A\subset G$ satisfy $|F_n\cap A|\to\infty$. Consider the following conditions:
\begin{enumerate}
	\item $A$ has positive lower $\F$-density and is $\F$-deterministic,
	\item $A$ preserves orbit-$\F$-normality, 
	\item $A$ preserves block-$\F$-normality,
	\item $A$ preserves simple $\F$-normality,\footnote{According to our Definition 
	\ref{np}, this means that every $\F$-normal (not just simply $\F$-normal) element is 
	simply $\F$-normal along $A$.}
	\item $A$ has positive upper $\F$-density and is $\F$-deterministic.
\end{enumerate}

Then 
$$
(1)^{\longnearrow} \!\!\!\!\!\!\!\!\!\!_{^{\longsearrow}} \begin{matrix}(2)\\(3) \end{matrix}
^{\longsearrow} \!\!\!\!\!\!\!\!\!\!_{^{\longnearrow}}(4)\!\!\implies\!\!(5).
$$
\end{thm}

The proof of the above theorem consists of three main steps presented in subsections \ref{kml}, \ref{spg}, \ref{msgg}, and is wrapped up in subsection \ref{pmr} (subsections \ref{til} and \ref{dbbm} are of auxiliary character and contain the description of some tools used later in subsections \ref{spg} and \ref{msgg}).

We would like to mention that in Kamae's paper \cite{Kamae} (which served as inspiration for Theorem \ref{main}) three similar  parts are also present, but, due to the much more general setup, the details of our proofs are far more intricate.

In general, the implications (1)$\implies$(2), (1)$\implies$(3) and (4)$\implies$(5) cannot be reversed; appropriate examples are provided in Propositions \ref{5n2} and \ref{234n1} below. We leave the question about the validity of the implications (4)$\implies$(2) and (4)$\implies$(3) open. Clearly, these two implications may fail only for sets $A$ such that $0=\underline d_\F(A)<\overline d_\F(A)$. Since such sets are, from the point of view of this paper, of lesser interest, we refrain from attempting to solve this (most likely difficult) problem.

On the other hand, conditions (1)--(5) become equivalent under additional assumptions, and we will now discuss two such cases, (A) and (B).

(A) It is clear that in the class of sets of positive lower $\F$-density the conditions (1)--(5) are equivalent (with (1) and (5) reduced to just ``$A$ is $\F$-deterministic''). In this manner we obtain a direct generalization of the Kamae--Weiss characterization of normality-preserving sets which have positive lower density (see Theorem \ref{KW}):

\begin{cor}\label{cmain}
\begin{enumerate}[(i)]
\item If $A\subset G$ has positive lower $\F$-density then the following conditions are equivalent:
\begin{enumerate}
	\item $A$ is $\F$-deterministic, 
	\item $A$ preserves simple $\F$-normality, 
	\item $A$ preserves orbit-$\F$-normality, 
	\item $A$ preserves block-$\F$-normality. 
\end{enumerate}
\item In the classical case ($G=\N$ and $F_n=\{1,2,\dots,n\}$), the conditions (a)--(d) are also equivalent to \hfill\break(e) $A$ preserves normality (in the sense of Definition \ref{np} (i)).
\end{enumerate}
\end{cor}
\smallskip

(B) The assumption that $A$ has positive lower $\F$-density can be replaced in the Theorem \ref{main} with some mild assumptions about the F\o lner sequence. These assumptions are easily seen to be satisfied by the standard F\o lner \sq\ $F_n=\{1,2,\dots,n\}$ in $\N$. In particular, we obtain a strengthening of the original Kamae--Weiss Theorem \ref{KW} \notet{(the same strengthening can also be found in \cite{WK})}.

\begin{thm}\label{1}
Let $\F=(F_n)_{n\in\N}$ be a F\o lner \sq\ in $G$ which is nested (i.e.\  $F_n\subset F_{n+1}$, $n\in\N$) and satisfies the condition $\lim_{n\to\infty}\frac{|F_{n+1}|}{|F_n|}=1$. If $A\subset G$ preserves simple $\F$-normality then $A$ has positive lower $\F$-density, and consequently the conditions (1)--(5) of Theorem \ref{main} are equivalent.
\end{thm}

The proof relies on the following lemma \notet{(see also \cite[Appendix A]{WK} for the classical case)}:
\begin{lem}
Let $\F$ be as in Theorem \ref{1} and let $A\subset G$ be a set of zero lower and positive upper $\F$-density. Then there exists a set $A'\subset A$ of $\F$-density zero and such that $\limsup_{n\to\infty}\frac{|F_n\cap A'|}{|F_n\cap A|}\ge\frac23$.
\end{lem}

\begin{proof} Let $d>0$ denote the upper $\F$-density of $A$. For each $0<\eps<\frac d5$ there exist arbitrarily large integers $m_\eps<M_\eps\in\N$ such that 
$$
\frac{|A\cap F_{m_\eps}|}{|F_{m_\eps}|}<\eps\text{\ \ and \ \ }\frac{|A\cap F_{M_\eps}|}{|F_{M_\eps}|}>d-\eps.
$$
We can assume that $m_\eps$ is so large that $1-\frac{|F_n|}{|F_{n+1}|}<\eps$ for any $n\ge m_\eps$. For each $n\ge m_\eps$ the ratio $\frac{|A\cap F_n|}{|F_n|}$ can grow between the indices $n$ and $n+1$ by at most
$$
\frac{|A\cap F_{n+1}|}{|F_{n+1}|}-\frac{|A\cap F_n|}{|F_n|}\le\frac{|A\cap F_{n+1}|-|A\cap F_n|}{|F_{n+1}|}\le\frac{|F_{n+1}|-|F_n|}{|F_{n+1}|}<\eps.
$$
(The middle inequality uses nestedness of $\F$.)
This implies that for some $n$ lying between $m_\eps$ and $M_\eps$ we have
$$
3\eps\le\frac{|A\cap F_n|}{|F_n|}<4\eps.
$$
We let $n_\eps$ be the smallest such $n$ and define $A'_\eps=(A\cap F_{n_\eps})\setminus F_{m_\eps}$. Observe that, since $F_{m_\eps}\subset F_{n_\eps}$, we have
\begin{multline}\label{eqn}
\frac{|A'_\eps\cap F_{n_\eps}|}{|A\cap F_{n_\eps}|}=\frac{|A\cap F_{n_\eps}|-|A\cap F_{m_\eps}|}{|A\cap F_{n_\eps}|}=
1-\frac{|A\cap F_{m_\eps}|}{|F_{m_\eps}|}\frac{|F_{m_\eps}|}{|F_{n_\eps}|}\frac{|F_{n_\eps}|}{|A\cap F_{n_\eps}|}>\\
1-\frac\eps{3\eps}=\frac23.
\end{multline}
We choose a \sq\ $(\eps_k)_{k\in\N}$ decreasing to zero, and select integers $m_{\eps_k}$ and $M_{\eps_k}$ so that $M_{\eps_k}<m_{\eps_{k+1}}$ for each $k\ge1$. Define 
$$
A'=\bigcup_{k\in\N}A'_{\eps_k}.
$$
The inequality \eqref{eqn} implies that $\limsup_{n\to\infty}\frac{|F_n\cap A'|}{|F_n\cap A|}\ge\frac23$. In order to show that $A'$ has $\F$-density zero, fix any $n\in\N$ and let $k$ be the largest integer such that $m_{\eps_k}\le n$ (notice that $k$ tends to infinity as $n$ increases). If $n\le n_{\eps_k}$ then 
$$
\frac{|A'\cap F_n|}{|F_n|}\le\frac{|A\cap F_n|}{|F_n|}<3\eps_k.
$$
For $n>n_{\eps_k}$ we have $n<M_{\eps_k}$, so $A'\cap F_n \subset (A\cap F_{m_{\eps_k}})\cup A'_{\eps_k}$, hence
$$
\frac{|A'\cap F_n|}{|F_n|}\le\frac{|A\cap F_{m_{\eps_k}}|}{|F_{m_{\eps_k}}|}+\frac{|A\cap F_{n_{\eps_k}}|}{|F_{n_{\eps_k}}|}\le5\eps_k.
$$
The above estimates for $n\le n_{\eps_k}$ and $n>n_{\eps_k}$ end the proof of the lemma.
\end{proof}

\begin{proof}[Proof of Theorem \ref{1}]
Let $A\subset G$ preserve simple $\F$-normality. By Theorem \ref{main}, $A$ has positive upper $\F$-density. Suppose the lower $\F$-density of $A$ equals zero. We will arrive at a contradiction by showing that $A$ does not preserve simple $\F$-normality. Let $B$ be any $\F$-normal set and let $A'\subset A$ be the set of $\F$-density zero constructed in the preceding lemma. Clearly $B'=B\setminus A'$ is also $\F$-normal. But $\liminf_{n\to\infty}\frac{|F_n\cap B'|}{|F_n\cap A|}\le\frac13$, which implies that $B'$ is not simply $\F$-normal along $A$.
\end{proof}

\subsection{Special joinings between positive entropy and Bernoulli measures}\label{kml}
This subsection is devoted to the first out of three key steps towards proving Theorem~\ref{main}.

In \cite[Lemma~3.1]{Kamae}, Kamae proves a fact (for the semigroup $\N_0$) which can be stated as follows: if $\mu$ is an \im\ on $\Lambda_1^{\N_0}$ with positive entropy and $\lambda$ is the uniform Bernoulli measure on $\Lambda_2^{\N_0}$, then there exists a joining $\xi=\mu\vee\lambda$ (supported on $\Lambda_1^{\N_0}\times\Lambda_2^{\N_0}$) such that the zero-coordinate partitions on both symbolic spaces are stochastically dependent. In the proof, Kamae essentially uses the conditional entropy formula $h(\mu,\P)=H(\mu,\P|\P^+)$ (where $\P^+=\bigvee_{n=1}^\infty T^{-n}(\P)$ is the \emph{future} of the process generated by a finite partition $\P$). Clearly, the notion of the future requires the semigroup to be linearly ordered, so Kamae's proof does not generalize to semigroups considered in this paper. Nevertheless, we are able to generalize his lemma using a proof which does not depend on orderability. 
\smallskip

Recall that a \emph{joining} of two measure-preserving systems $(X,\mu,G)$ and $(Y,\nu,G)$ is any system of the form $(X\!\times\!Y,\mu\vee\nu,G)$, where the measure $\mu\vee\nu$ has marginals $\mu$ and $\nu$ and is \inv\ under the product action of $G$ given by $g(x,y)=(g(x),g(y))$, $x\in X,\ y\in Y,\ g\in G$. Any measure on $X\times Y$ (not necessarily \inv\ under the product action) with marginals $\mu$ and $\nu$ will be referred to as a \emph{coupling} of $\mu$~and~$\nu$.

\begin{thm}\label{depzc}
Let $G$ be a countable cancellative amenable semigroup with a unit. Let $\mu$ be a shift \im\ on $\Lambda_1^G$ with positive entropy $h(\mu)=h(\mu,\Lambda_1)$. Let $\lambda$ denote the uniform Bernoulli measure on $\Lambda_2^G$. Then there exists a joining $\xi=\mu\vee\lambda$ which makes the zero-coordinate partitions $\Lambda_1, \Lambda_2$ not independent.
\end{thm}

\begin{proof} Fix a centered, nested and exhaustive F\o lner \sq\ $\F=(F_n)_{n\in\N}$ in $G$ (such a F\o lner \sq\ always exists) and let us list the elements of the semigroup in some order: $G=\{e=g_0,g_1,g_2,\dots\}$, so that for each $n$ there exists $l_n$ such that $F_n=\{g_0,g_1,g_2,\dots,g_{l_n-1}\}$. For each $l\in\N$, we have
$$
\tfrac1l H(\mu,\Lambda_1^{\{g_0,g_1,g_2,\dots,g_{l-1}\}})=\tfrac1l H(\mu,\Lambda_1)+\tfrac1l \sum_{i=1}^{l-1}H(\mu,\Lambda_1^{\{g_i\}}|\Lambda_1^{\{g_0,g_1,\dots,g_{i-1}\}}).
$$
Clearly, for $l=l_n$, $\tfrac1{l_n}H(\mu,\Lambda_1^{\{g_0,g_1,g_2,\dots,g_{l_n-1}\}})=\frac1{|F_n|}H(\mu,\Lambda_1^{F_n})$ and one has 
$$
\lim_{n\to\infty}\frac1{|F_n|}H(\mu,\Lambda_1^{F_n})=h(\mu).
$$ 

Because the numbers $H(\mu,\Lambda_1^{\{g_i\}}|\Lambda_1^{\{g_0,g_1,\dots,g_{i-1}\}})$ are bounded (by $\log|\Lambda_1|$), and the averages $\tfrac1{l_n} H(\mu,\Lambda_1)+\tfrac1{l_n} \sum_{i=1}^{l_n-1}H(\mu,\Lambda_1^{\{g_i\}}|\Lambda_1^{\{g_0,g_1,\dots,g_{i-1}\}})$ converge to $h(\mu)$, there exists a positive number $\alpha$ such that for $n$ large enough we have
\begin{equation}\label{ojo}
\frac1{l_n}{|\{i=0,1,\dots,l_n-1}:H(\mu,\Lambda_1^{\{g_{i}\}}|\Lambda_1^{\{g_0,g_1,\dots,g_{i-1}\}})\ge\tfrac12h(\mu)\}|\ge\alpha.
\end{equation}
Let 
$$
M=\{g_l\in G:H(\mu,\Lambda_1^{\{g_{l}\}}|\Lambda_1^{\{g_0,g_1,\dots,g_{l-1}\}})\ge\tfrac12h(\mu)\}.
$$
Inequality \eqref{ojo} shows that $\underline d_{\F}(M)>0$. For a symbol $a\in\Lambda_1$ and $g\in G$ let $[a]_g=\{x\in\Lambda_1^G: x(g)=a\}$. The fact that $\underline d_{\F}(M)>0$ implies that there exists an $\eps>0$ such that for each $g_l\in M$, there is a collection $\B_l$ of blocks $B\in\Lambda_1^{\{g_0,g_1,\dots,g_{l-1}\}}$ satisfying 
$\mu(\bigcup_{B\in\B_l}[B])\ge\eps$, and such that for any $B\in\B_l$ the conditional distribution on the symbols of $\Lambda_1$, $a\mapsto \mu([a]_{g_l}|B)$, is ``far enough'' from being concentrated at one symbol; more precisely, there exists a symbol $a\in\Lambda_1$ (depending on $g_l\in M$ and $B\in\B_l$) with 
$$
\eps\le\mu([a]_{g_l}|B)\le1-\eps.
$$
Then either $\eps\le\mu([a]_{g_l}|B)\le\frac12$ or $\frac12<\mu([a]_{g_l}|B)\le 1-\eps$ and in the latter case there exists another symbol, $a'\in\Lambda_1$ with $\frac\varepsilon{|\Lambda_1|}\le\mu([a']_{g_l}|B)<\frac12$.
So, replacing, if needed, $a$ by $a'$ we have proved that for each $g_l\in M$ and $B\in\B_l$ there exists a symbol $a\in\Lambda_1$ such that
\begin{equation}\label{pop}
\frac\varepsilon{|\Lambda_1|}\le\mu([a]_{g_l}|B)\le\frac12.
\end{equation}
Now, for each $g_l\in M$, there exists an element $a_l\in\Lambda_1$ and a subset  $\B'_l\subset\B_l$ such that, on the one hand, 
\begin{equation}\label{tp}
\mu(\bigcup_{B\in\B'_l}[B])\ge\frac\varepsilon{|\Lambda_1|},
\end{equation} 
and, on the other hand, for all blocks $B\in\B'_l$ the inequalities \eqref{pop} hold for $a=a_l$. 
Finally, observe that $a_l$ will assume a common value $a_0\in\Lambda_1$ for $g_l$ in a subset $M'\subset M$, such that $\overline d_\F(M')\ge\frac1{|\Lambda_1|}{\underline d_\F(M)}$. We have shown that there exists a symbol $a_0\in\Lambda_1$ verifying
\begin{equation}\label{tpp}
\frac\varepsilon{|\Lambda_1|}\le\mu([a_0]_{g_l}|B)\le\frac12,
\end{equation}
for all $g_l\in M'$ and $B\in\B'_l$. By the law of total probability and invariance of $\mu$, for $g_l\in M'$ we have 
\begin{multline*}
\mu([a_0])=\mu([a_0]_{g_l})=\sum_{B\in\Lambda_1^{\{g_0,g_1,\dots,g_{l-1}\}}}\mu([B])\mu([a_0]_{g_l}|B)\ge\\ \sum_{B\in\B'_l}\mu([B])\mu([a_0]_{g_l}|B)\ge\Bigl(\frac\eps{|\Lambda_1|}\Bigr)^2>0.
\end{multline*}

We will now define a special coupling $\zeta$ of $\mu$ and $\lambda$. Before we proceed, we fix a symbol $b_0\in\Lambda_2$. We will define $\zeta$ inductively. First, we set
$$
\zeta([a]_{g_0}\times[b]_{g_0})=\mu([a])\lambda([b])=\mu([a])\frac1{|\Lambda_2|}, \ a\in\Lambda_1,\ b\in\Lambda_2.
$$ 
Let $l\in\N$ and assume that the coupling $\zeta$ is defined on 
$\Lambda_1^{\{g_0,g_1,\dots,g_{l-1}\}}\times \Lambda_2^{\{g_0,g_1,\dots,g_{l-1}\}}$. We now extend it to $\Lambda_1^{\{g_0,g_1,\dots,g_l\}}\times \Lambda_2^{\{g_0,g_1,\dots,g_l\}}$. To this end, it suffices, for every pair of blocks $(B,C)\in\Lambda_1^{\{g_0,g_1,\dots,g_{l-1}\}}\times\Lambda_2^{\{g_0,g_1,\dots,g_{l-1}\}}$, to appropriately assign the conditional values
$$
\zeta([a]_{g_l}\times[b]_{g_l}\bigl|\bigr.B\times C)
$$
for $a\in\Lambda_1$, $b\in\Lambda_2$. Here is how we do it. If $g_l\notin M'$, we let  $\zeta([a]_{g_l}\times[b]_{g_l}\bigl|\bigr.B\times C)=\mu([a]_{g_l}|B)\frac1{|\Lambda_2|}$ (regardless of $C$ and $b$), i.e.\ we declare the cylinders $[a]_{g_l}$ and $[b]_{g_l}$ to be conditionally independent given $B$. We apply the same formula if $g_l\in M'$ and $B\notin\B'_l$. But if $g_l\in M'$ and $B\in\B'_l$, we distribute the masses differently: namely we let
\begin{align}
\zeta([a_0]_{g_l}\times[b_0]_{g_l}|B\times C)&=0, &  &\label{f1}\\
\zeta([a_0]_{g_l}\times[b]_{g_l}|B\times C)&= \mu([a_0]_{g_l}|B)\tfrac1{|\Lambda_2|-1} &  &(b\neq b_0),\label{f2}\\
\zeta([a]_{g_l}\times[b_0]_{g_l}|B\times C)&=\tfrac{\mu([a]_{g_l}|B)}{1-\mu([a_0]_{g_l}|B)}\tfrac1{|\Lambda_2|} &  &(a\neq a_0),\label{f3}\\
\zeta([a]_{g_l}\times[b]_{g_l}|B\times C)&=\mu([a]_{g_l}|B)\tfrac1{|\Lambda_2|}\cdot \mathsf R & &(a\neq a_0,\,b\neq b_0),\label{f4}
\end{align} 
where 
$$
\mathsf R = \frac{|\Lambda_2|-\frac1{1-\mu([a_0]_{g_l}|B)}}{|\Lambda_2|-1}.
$$
Clearly, $\mathsf R\le 1$. Because we have arranged that $\mu([a_0]_{g_l}|B)\le\frac12$, the term $\frac1{1-\mu([a_0]_{g_l})}$ does not exceed $2$, and hence does not exceed $|\Lambda_2|$ and $\mathsf R$ is nonnegative. Obviously, the numbers appearing on the right hand sides in \eqref{f1}--\eqref{f4} range between $0$~and~$1$. It is straightforward to verify that so defined measure $\zeta$ satisfies
$$
\zeta\Bigl(\Bigl.[a]_{g_l}\times\bigcup_{b\in\Lambda_2}[b]_{g_l}\Bigr|B\times C\Bigr)=\mu([a]_{g_l}|B) \text{ \ \ and \ \ }
\zeta\Bigl(\Bigl.\bigcup_{a\in\Lambda_1}[a]_{g_l}\times[b]_{g_l}\Bigr|B\times C\Bigr)=\frac1{|\Lambda_2|}.
$$
This way the inductive procedure described above  leads to a coupling $\zeta$ whose marginals are $\mu$ and $\lambda$. We have 
\begin{equation*}\label{cas}
\zeta([a_0]_{g_l}\times[b_0]_{g_l}|B\times C)=\begin{cases}
0& \text{ if }g_l\in M'$ and $B\in\B'_l,\\
\mu([a_0]_{g_l}|B)\frac1{|\Lambda_2|}&\text{ otherwise}.
\end{cases} 
\end{equation*}
For $g_l\notin M'$ we invoke the law of total probability and the invariance of $\mu$ to obtain the following:
\begin{multline}\label{osz1}
\zeta([a_0]_{g_l}\times[b_0]_{g_l})=\sum_B\sum_C\zeta(B\times C)
\zeta([a_0]_{g_l}\times[b_0]_{g_l}|B\times C)=\\
\sum_B\mu(B)\mu([a_0]_{g_l}|B)\frac1{|\Lambda_2|}=
\mu([a_0]_{g_l})\frac1{|\Lambda_2|}=\mu([a_0])\frac1{|\Lambda_2|},
\end{multline}
(where $B$ and $C$ range over $\Lambda_1^{\{g_0,g_1,\dots,g_{l-1}\}}$ and $\Lambda_2^{\{g_0,g_1,\dots,g_{l-1}\}}$, respectively).

\noindent For $g_l\in M'$, repeating the above calculation (and using \eqref{tp} and \eqref{tpp}), we get
\begin{multline}\label{osz2}
\zeta([a_0]_{g_l}\times[b_0]_{g_l})= \mu([a_0]_{g_l})\frac1{|\Lambda_2|} - \sum_{B\in\B'_l}\mu(B)\mu([a_0]_{g_l}|B)\frac1{|\Lambda_2|}\le\\ \Bigl(\mu([a_0]) - \Bigl(\frac\eps{|\Lambda_1|}\Bigr)^2\Bigr)\frac1{|\Lambda_2|}.
\end{multline} 

Recall that $M'$ has positive upper $\F$-density. Replacing, if necessary, $\F$ by a sub\sq\ we can assume that $d_\F(M')$ exists and is positive. 

To create a joining from the coupling $\zeta$, we apply the standard averaging procedure, letting
$$
\xi = \lim_{k\to\infty} \frac1{l_{n_k}}\sum_{l=0}^{l_{n_k}-1}g_l(\zeta)=
\lim_{k\to\infty} \frac1{|F_{n_k}|}\sum_{g\in F_{n_k}}g(\zeta),
$$
where $(n_k)_{k\in\N}$ is any sub\sq\ for which the above weak* limit exists. Then $\xi$ is an \im\ on the product space $\Lambda_1^G\times\Lambda_2^G$ (see \eqref{domi}), and, by invariance of $\mu$ and $\lambda$, the marginals of $\xi$ are $\mu$ and $\lambda$. So, $\xi$ is indeed a joining of $\mu$ and $\lambda$. Finally, by \eqref{osz1} and \eqref{osz2} (and since $d_\F(M')>0$), we obtain
\begin{multline*}
\xi([a_0,b_0])= \lim_{k\to\infty}\frac1{|F_{n_k}|}\sum_{g\in F_{n_k}}\zeta([a_0,b_0]_g)\le\\
\lim_{k\to\infty}\left(\frac{|F_{n_k}\setminus M'|}{|F_{n_k}|} \mu([a_0])\frac1{|\Lambda_2|}+\frac{|F_{n_k}\cap M'|}{|F_{n_k}|} \Bigl(\mu([a_0]) - \Bigl(\frac\eps{|\Lambda_1|}\Bigr)^2\Bigr)\frac1{|\Lambda_2|}\right)=\\
\mu([a_0])\frac1{|\Lambda_2|}-\Bigl(d_\F(M')\Bigl(\frac\eps{|\Lambda_1|}\Bigr)^2\frac1{|\Lambda_2|}\Bigr) <\mu([a_0])\lambda([b_0]).
\end{multline*}
This proves that for the joining $\xi=\mu\vee\lambda$, the zero coordinate partitions $\Lambda_1$ and $\Lambda_2$ are not independent. We are done.
\end{proof}

\subsection{Tilings}\label{til} In this subsection we summarize some facts concerning tilings and systems of tilings of amenable groups introduced and studied in \cite{DHZ}, which will be used in the sequel (subsections \ref{spg} and \ref{msgg}). While the technology described in this section pertains to groups only, our main results (Theorem \ref{main} and \ref{chara}) will be proved with the help of tilings of groups for actions of general countable cancellative amenable semigroups.

\subsubsection{General tilings}
Let $G$ be a countable amenable group. 
\begin{defn}A \emph{tiling} of $G$ is a partition of $G$ into (countably many) finite sets. 
\end{defn}
%The tiling will be denoted by $\CT$ or $\Theta$ and its elements are either $T\in\CT$ or $\theta\in\Theta$. The first type of notation is reserved for proper tilings, as defined below. 

\begin{defn}\label{prop} A tiling $\CT$ is \emph{proper} if there exists a finite collection $\CS$ of finite sets $S\in\S$ each containing the unit $e$, called the \emph{shapes} of $\CT$, such that for every $T\in\CT$ there exists a shape $S\in\S$ satisfying $T=Sc$ for some $c\in G$ (in fact, we then have $c\in T$). If $|\CS|=1$, $\CT$ is called a \emph{monotiling}.
\end{defn}
When dealing with a proper tiling $\CT$, we will always fix one collection of shapes $\CS$ and one representation $T\mapsto(S,c)$, where $S\in\S, c\in G$ are such that $T=Sc$. (We remark that, in general, there may be more than one such representation, even when $\S$ is fixed; for instance this is the case when some $S\in\CS$ is a finite subgroup.) Once such a representation is fixed, we will call $S$ and $c$ the \emph{shape} and \emph{center} of $T$, respectively. Given $S\in\CS$, we will denote by $C_S(\CT)$ the set of centers of the tiles having the shape $S$, while $C(\CT)=\bigcup_{S\in\CS}C_S(\CT)$ will be used to denote the set of centers of all the tiles.

\smallskip
We now present another type of a (not proper) tiling which depends on an a priori given F\o lner \sq\ $\F$. We will introduce it via an appropriate existence theorem, which follows from \cite[Theorem 5.1]{BDM}:

\begin{thm}\label{tiln} Let $\F=(F_n)_{n\in\N}$ be a F\o lner \sq\ in a countable amenable group $G$. Then there exists a tiling $\Theta$ of $G$ with the following properties:
\begin{itemize}
		\item The tiles $\theta$ comprising $\Theta$ (taken in any order) form a disjoint and exhaustive F\o lner \sq\ $(\theta_n)_{n\in\N}$.
	\item For each $n\in\N$, denote by $F_n^\Theta$ the \emph{$\Theta$-saturation of $F_n$}, 
	i.e.\ the union of all tiles $\theta\in\Theta$ such that $\theta\cap F_n\neq\emptyset$. 
	Then 
	$$
	\lim_{n\to\infty}\frac{|F_n^\Theta\setminus F_n|}{|F_n|}=0.
	$$ 
\end{itemize}
\end{thm}
The second condition simply means that the \sq s $\F$ and $\F^\Theta=(F_n^\Theta)_{n\in\N}$ are equivalent (see Definition \ref{equi}), in particular $\F^\Theta$ is a F\o lner \sq. Since $(\theta_n)_{n\in\N}$ is a F\o lner \sq, we have $|\theta_n|\to\infty$, the tiling $\Theta$ cannot be proper. In fact, it may have only finitely many tiles of any given shape. 
\subsubsection{Dynamical tilings}

Let $\CT$ be a proper tiling with the collection of shapes $\S$. Denote by $\rm V$ the finite alphabet consisting of symbols assigned bijectively to the shapes of $\CT$ plus one additional symbol:
\begin{equation}\label{alf}
\rm V = \{``S": S\in\S\}\cup\{``0"\}.
\end{equation} 
Then $\CT$ can be identified with the symbolic element, denoted by the same letter $\CT\in{\rm V}^G$, defined as follows: 
\begin{equation}\label{tsy}
\CT(g) = \begin{cases}
``S" \text{ for some }S\in\S, &\text{if } g\in C_S(\CT) \text{ \tiny($g$ is the center of a tile of shape $S$)},\\ 
``0", & \text{otherwise \tiny($g$ is not a center of any tile)}.
\end{cases}
\end{equation}

\begin{defn}
Let $\rm V$ be an alphabet of the form \eqref{alf} for some finite collection $\CS$ of finite sets $S$. Let $\T\subset V^G$ be a subshift such that each element $\CT\in\T$ represents a proper tiling with the collection of shapes $\S$. Then we call $\T$ a \emph{dynamical tiling} and $\S$ \emph{the collection of shapes} of $\T$. 
\end{defn}
It is elementary to see that the orbit-closure (under the shift-action of $G$) of any proper tiling $\CT$ is s dynamical tiling.

\smallskip
In the sequel we will be using a very special \tl\ joining of dynamical tilings. By
a \emph{\tl\ joining} of a \sq\ of \ds s $(X_k,G)$, $k\in\N$, (denoted by $\bigvee_{k\in\N}X_k$) we mean any closed subset of the Cartesian product $\prod_{k\in\N}X_k$ which has full projections onto the coordinates $X_k$, $k\in\N$, and is \inv\ under the product action given by $g(x_1,x_2,\dots)=(g(x_1),g(x_2),\dots)$.

\begin{defn}\label{tili}
Consider a \sq\ of dynamical tilings $(\T_k)_{k\in\N}$. By a \emph{system of tilings} (generated by the dynamical tilings $\T_k$) we will mean any \tl\ joining $\TTT=\bigvee_{k\in\N}\T_k$. 
\end{defn}

The elements of $\TTT$ have the form of \sq s of tilings $\boldsymbol\CT=(\CT_k)_{k\in\N}$, where $\CT_k\in\T_k$ for each $k$.

\begin{defn}\label{tili1}
Let $\TTT=\bigvee_{k\in\N}\T_k$ be a system of tilings and let $\S_k$ denote the collection of shapes of $\T_k$. The system of tilings is:
\begin{itemize}
	\item \emph{F\o lner}, if the union of the collections of shapes $\bigcup_{k\in\N}\S_k$ 
	(taken in any order) is a F\o lner \sq;
	\item \emph{congruent}, if for each $\boldsymbol{\CT}=(\CT_k)_{k\in\N}\in\TTT$ and each 	
	$k\in\N$, every tile of $\CT_{k+1}$ is a union of some tiles of $\CT_k$.
	\item \emph{uniquely congruent}, if it is congruent and for any $k\ge 1$ 
	all tiles of $\T_{k+1}$ having the same shape are partitioned by the tiles of $\T_k$ ``the 
	same way, up to shifting''. More precisely, whenever $\boldsymbol{\CT}=(\CT_k)_{k\in\N}$ 
	and $\boldsymbol{\CT'}=(\CT'_k)_{k\in\N}$ are elements of $\TTT$, $T=S'c$ and $T'=S'c'$ 
	are tiles of $\CT_{k+1}$ and $\CT'_{k+1}$, respectively (with the same shape 
	$S'\in\S_{k+1}$), and $T=\bigcup_{i=1}^l T_i$ is the partition of $T$ into the tiles of 
	$\mathcal T_k$, then the sets $T'_i=T_ic^{-1}c'$ (for
	$i=1,2,\dots,l$) are tiles of $\mathcal T'_k$ (and then they form the partition of 
	$T'$ into the tiles of $\T'_k$). 
\end{itemize}
\end{defn}	
	
One can see that if $\TTT$ is a uniquely congruent system of tilings	then for any $k'>k$ and any shape $S'\in\CS_{k'}$, there exist sets $C_{S}(S')\subset S'$ indexed by $S\in\CS_k$, such that
	$$
	S'=\bigcup_{S\in\CS_k}\ \bigcup_{c\in C_S(S')}Sc \text{\ \ (disjoint union),}
	$$
	and for each $\boldsymbol{\CT}=(\CT_i)_{i\in\N}\in\TTT$, whenever $S'c'$ is
	a tile of $\CT_{k'}$ then 
	$$
	S'c'=\bigcup_{S\in\CS_k}\ \bigcup_{c\in C_{S}(S')}Scc'
	$$
	is the partition of $S'c'$ by the tiles of $\CT_k$. 

\begin{rem}\label{remd} In the papers \cite{DHZ} and \cite{BDM}, uniquely congruent systems of tilings are called ``deterministic'' because in such a system, for each $\boldsymbol{\CT}=(\CT_k)_{k\in\N}\in\TTT$, each tiling $\CT_{k'}$ \emph{determines} all the tilings $\mathcal T_k$ with $k<k'$, and the assignment $\CT_{k'}\mapsto\CT_k$ is a \tl\ factor map from $\T_{k'}$ onto $\T_k$. In such a case, the joining $\TTT$ is in fact an inverse limit
$$
\TTT=\overset\longleftarrow{\lim_{k\to\infty}}\T_k.
$$
Since the central subject of this paper are ``deterministic functions'' and ``deterministic sets'', where the term ``deterministic'' has a different meaning, in order to avoid terminological collision, we have renamed ``deterministic systems of tilings'' to ``uniquely congruent systems of tilings''.
\end{rem}
%\begin{rem}\label{remark}
%Any congruent system of tilings $\TTT=(\T_k)_{k\in\N}$ can be easily made deterministic in an inductive process of \emph{duplicating the shapes}, as follows: For each $S\in\CS_k$ ($k\ge 2$) there are only finitely many, say $m(S)$, possible partitions of $S$ into tiles from $\CS_{k-1}$. We duplicate the tile $S$ into $m(S)$ copies (identical as subsets of $G$, but in the symbolic representation of the tiling we will now associate to them different symbols, say $``S_1",``S_2",\dots,``S_{m(S)}"$) and use them for tiles of $\CT_k$ with the original shape $S$ according to how the tile is subdivided by the tiles of $\CT_{k-1}$. Clearly, this process increases the cardinality of $\CS_k$, but it is a \tl\ conjugacy, so the dynamical properties (for example the \tl\ entropy) of $\TTT$ remain unchanged.
%\end{rem}

The following theorem will play a crucial role in our considerations:	

\begin{thm}(\cite[Theorem 5.2]{DHZ})\label{dhz} Every countable amenable group admits a uniquely congruent F\o lner system of tilings $\TTT$ with \tl\ entropy~zero.
\end{thm}

\subsection{Empirical measures, distance between blocks and measures}\label{dbbm}
Throughout this section $G$ stands for a countable cancellative amenable semigroup. 

\smallskip
Given a finite alphabet $\Lambda$, by $\Lambda^*$ we will denote the family of all finite blocks over $\Lambda$:
$$
\Lambda^*=\bigcup_K\Lambda^K,
$$
where $K$ ranges over all finite subsets of $G$. Let $K, K_0$ be finite subsets of $G$ and let $B\in\Lambda^K$, $B_0\in\Lambda^{K_0}$. We define the \emph{frequency of $B$ in $B_0$} as
$$
\Fr_{B_0}(B)=\frac{|\{g\in (K_0)_K:\ B_0|_{Kg}\approx B\}|}{|K_0|},
$$
where $(K_0)_K$ is the $K$-core of $K_0$. Of course, if $K$ does not ``fit'' inside $K_0$ (i.e.\ if $(K_0)_K=\emptyset$, for example if $|K|>|K_0|$) then the above frequency equals zero.

An alternative formula for the frequency is:
\begin{equation}\label{okr}
\Fr_{B_0}(B)=\frac1{|K_0|}\sum_{g\in(K_0)_K}\delta_{g(x)}([B]),
\end{equation}
where $x$ is any point from $[B_0]$.

Later, we will consider finite blocks over product alphabets, such as $\Lambda_1\times\Lambda_2$ or $\Lambda_1\times\Lambda_2\times\rm V$. Blocks over product alphabets can be viewed as \emph{double} or \emph{triple} blocks, for example a block in $(\Lambda_1\times\Lambda_2)^K$ (where $K\subset G$ is finite) can be viewed as a pair of blocks $(B,C)\in \Lambda_1^K\times\Lambda_2^K$. Likewise, every block in $(\Lambda_1\times\Lambda_2\times\rm V)^K$ is a triple of blocks $(B,C,D)\in\Lambda_1^K\times\Lambda_2^K\times\rm V^K$. Thus we will be dealing, among other things, with terms of the form
$$
\Fr_{(B_0,C_0,D_0)}(B,C,D)=\frac{|\{g\in (K_0)_K:\ (B_0,C_0,D_0)|_{Kg}\approx(B,C,D)\}|}{|K_0|}.
$$
Note that we write $\Fr_{(B_0,C_0,D_0)}(B,C,D)$ instead of $\Fr_{(B_0,C_0,D_0)}((B,C,D))$.

\smallskip
When the domain $K_0$ of $B_0$ has good invariance properties (see Remark \ref{sgip}), the function associating to blocks $B\in\Lambda^*$ their frequencies in $B_0$ has some resemblance to an \im, and by slight abuse of terminology we will call it the \emph{empirical measure} generated by $B_0$ and denote it by $\hat B_0$. A similar terminology will be used when dealing with double and triple blocks.
\smallskip

We will be using the following distance between shift-\im s on $\Lambda^G$:
$$
\dist(\mu,\nu)=\sup_K \bigl(u(K)\sup_{B\in\Lambda^K}|\mu(B)-\nu(B)|\bigr),
$$
where the first supremum is taken over all finite sets $K\subset G$ and $u(\cdot)$ is a fixed positive function defined on finite subsets of $K$ and converging to zero (which means that for any $\eps>0$, $u(K)\ge\eps$ for at most finitely many sets $K$). Note that $\dist(\cdot,\cdot)$ is compatible with the weak* topology.
%The factor $\frac1{|K|}$ is added only to ensure correct measurement of the distance between blocks in the vicinity of atomic measures.

By extension, we can also define the distance between any finite blocks $B_0,B_0'\in\Lambda^*$ by
$$
\dist(B_0,B_0')=\dist(\hat B_0,\hat B_0')=\sup_K \bigl(u(K)\sup_{B\in\Lambda^K}|\Fr_{B_0}(B)-\Fr_{B_0'}(B)|\bigr),
$$
and the distance between a block $B_0$ and an \im\ $\mu\in\mgl$ by 
$$
\dist(B_0,\mu)=\dist(\hat B_0,\mu)=
\sup_K\bigl(u(K)\sup_{B\in\Lambda^K}|\Fr_{B_0}(B)-\mu(B)|\bigr).
$$
So defined function $\dist$ is a metric on the union $\mgl\cup\Lambda^*$. In the following proposition we provide a list of several useful facts involving the metric $\dist(\cdot,\cdot)$.

\begin{fact}\label{fakty}{\color{white}.}
\begin{enumerate}
	\item For every $\eps>0$ there exists a finite family $\Pi_\eps$ of finite subsets of $G$ 
	and a $\delta>0$, such that if $\mu,\nu\in\mgl\cup\Lambda^*$ and $|\mu(B)-\nu(B)|<\delta$ 
	for each block $B\in\Lambda^K$ with $K\in\Pi_\eps$, then $\dist(\mu,\nu)<\eps$. (Such 
	blocks $B$ will be referred to as \emph{pivotal blocks}.) 
	\item For every $\eps>0$ there exists a $\delta>0$ such that if two finite sets 
	$K_1,K_2\subset G$ satisfy
	$$
	\frac{|K_1\cap K_2|}{|K_1\cup K_2|}>1-\delta,
	$$
	and blocks $B_1\in\Lambda^{K_1},\, B_2\in\Lambda^{K_2}$ agree on $K_1\cap K_2$, then
	$\dist(B_1,B_2)<\eps$.
	\item Fix an $\eps>0$. If a finite set $F\subset G$ has sufficiently good invariance 
	properties (see Remark \ref{sgip}) then any block $B\in\Lambda^F$ is $\eps$-close to an   
	\im\ supported by 
	$\Lambda^G$. 
	\item Given a F\o lner \sq\ $\F=(F_n)_{n\in\N}$ and a measure $\mu\in\M_G(\Lambda^G)$, a point $y\in\Lambda^G$ is $\F$-generic 
	for $\mu$ if and only if the blocks $y|_{F_n}$ converge in $\dist$ to $\mu$.
	\item The metric $\dist$ is convex on $\mgl$ and can be naturally extended to a convex 
	metric on the the convex hull of $\mgl\cup\Lambda^*$.
	\item Fix an $\eps>0$ and let $\delta$ and the family of pivotal blocks be as in (1).
	Whenever $K_1,K_2,\dots,K_q$ are pairwise disjoint finite subsets of $G$ with sufficiently 
	good invariance properties, and $B_i\in\Lambda^{K_i}$ ($i=1,2,\dots,q$), then, denoting by 
	$B_0$ the \emph{concatenation} of the blocks $B_1,B_2,\dots,B_q$ (which is 
	a block defined on $K_0=\bigcup_{i=1}^q K_i$), we have 
	for any pivotal block $B$:
	$$
	\Bigl|\Fr_{B_0}(B)-\sum_{i=1}^q \frac{|K_i|}{|K_0|}\ 
	\Fr_{B_i}(B)\Bigr|<\delta.
	$$
	In particular, by (1), we have
	$$
	\dist\Bigl(\hat B_0,\sum_{i=1}^q \frac{|K_i|}{|K_0|}\hat B_i\Bigr)<\eps.
	$$ 
\end{enumerate}
\end{fact}

\begin{proof} We skip the (fairly standard) proofs except that of (3). The formula \eqref{okr} implies that, given $\eps>0$, whenever $F$ has sufficiently good invariance properties then, for any block $B\in\Lambda^F$, we have
$$
\dist\Bigl(B,\frac1{|F|}\sum_{g\in F}\delta_{g(x)}\Bigr)<\eps,
$$
where $x$ is any point from $[B]$. In view of \eqref{domi}, this implies that any \sq\ of blocks, whose domains form a F\o lner \sq, has a sub\sq\ convergent to an \im\ supported by $\Lambda^G$.  Assume that (3) is false. Then for some $\eps>0$ there exist finite sets $F$ with arbitrarily good invariance properties and blocks $B\in\Lambda^F$ such that $\dist(B,\M_G(\Lambda^G))\ge\eps$. Out of these sets $F$ we can select a F\o lner \sq\ $\F=(F_n)_{n\in\N}$ and choose blocks $B_n\in\Lambda^{F_n}$ which are $\eps$-apart from $\M_G(\Lambda^G)$. This is a contradiction as the \sq\ of blocks $(B_n)_{n\in\N}$ must contain a sub\sq\ converging to an element of $\M_G(\Lambda^G)$.
\end{proof}

In the sequel we will need the following lemma, which complements the condition (3). We remark that it is actually valid without the ergodicity assumption, but then proof is much longer.

\begin{lem}\label{bloki}
Fix an $\eps>0$ and let $\nu\in\mgl$ be ergodic. If a finite set $F\subset G$ has sufficiently good invariance properties then there exists a block $B\in\Lambda^F$ such that
$\dist(B,\nu)<\eps$. 
\end{lem}

\begin{proof}
Suppose this is not true. Then there exists a F\o lner \sq\ $\F=(F_n)_{n\in\N}$ such that, for any $n$, any block $B\in\Lambda^{F_n}$ satisfies $\dist(B,\nu)\ge\eps$. By Proposition~\ref{miner},
there exists a sub\sq\ $\F_\circ=(F_{n_k})_{k\in\N}$ of $\F$ and a point $y\in\Lambda^G$ which is $\F_\circ$-generic for $\nu$. The blocks $y|_{F_{n_k}}$ converge in $\dist$ to $\nu$, contradicting the choice~of~$\F$.
\end{proof}

\subsection{Quasi-generic points for joinings}\label{spg} 
The next theorem is the second out of the three main steps leading to the proof of Theorem \ref{main}.

\begin{thm}\label{generic}{\rm(Cf. \cite[Theorem 2]{Kamae})} Let $G$ be a countable cancellative amenable semigroup, $\F$ an arbitrary F\o lner \sq\ in $G$ and $\Lambda_1, \Lambda_2$ some finite alphabets. Let $\mu$ and $\nu$ be \im s on $\Lambda_1^G$ and $\Lambda_2^G$, respectively. Let $x\in\Lambda_1^G$ be $\F$-generic for $\mu$. For any joining $\xi=\mu\vee\nu$ there exists an element $y\in\Lambda_2^G$ such that the pair $(x,y)$ is $\F$-quasi-generic for $\xi$ (in particular, $y$ is quasi-generic for $\nu$).
\end{thm}

\begin{que} We failed to prove (although we believe it is true) that there exists an element $y\in\Lambda_2^G$ such that $(x,y)$ is $\F$-generic for $\xi$. Such a fact is known for $\N_0$ or $\Z$, but in the general case the question seems to be open.
\end{que}

\begin{rem} We would like to stress two important features of Theorem \ref{generic}. The first one is that ergodicity of $\xi$ or even of $\mu$ is not assumed (hence the set of $\F$-generic points for $\mu$ can have measure zero). The second one is that the corresponding element $y$ exists for \emph{every} (not just almost every) $\F$-generic point $x$. 
\end{rem}

\begin{proof}[Proof of Theorem \ref{generic}] In order to be able to use the machinery of tilings we temporarily assume that $G$ is a group.
Let $(\eps_k)_{k\in\N}$ be a decreasing to zero \sq\ of positive numbers. Each $\eps_k$ determines a $\delta_k>0$ and the family of $k$-pivotal blocks, as in Proposition \ref{fakty} (1). Let $(\CT_k)_{k\in\N}$ be a congruent F\o lner \sq\ of proper tilings of $G$. By passing, if needed, to a sub\sq\ we can assume that the shapes $S\in\S_k$
(where $\S_k$ is the collection of shapes of $\CT_k$) have good enough invariance properties, so that for each $k$-pivotal set $K$ the $K$-core of each $S$ is a $(1-\frac{\delta_k}4)$-subset of $S$ (see Lemma \ref{estim}). That is to say,
\begin{equation}\label{jeden}
\text{for each }S\in\S_k,\ |\{g\in S_k:Kg\not\subset S_k\}|\le\tfrac{\delta_k}4|S_k|.
\end{equation}
Let $M_k$ denote the cardinality of the set of all possible double blocks over $\Lambda_1\times\Lambda_2$ whose domains are the shapes of $\CT_k$. Clearly, $M_k$ is a finite number, not exceeding $|\S_k|(|\Lambda_1||\Lambda_2|)^{|S_{k,\max}|}$, where $S_{k,\max}$ is the largest, in terms of cardinality, element of $\S_k$. Recall that $\CT_k$ can be viewed as a symbolic element over the finite alphabet ${\rm V}_k=\{``S":S\in\S_k\}\cup\{``0"\}$. Recall also that we are given a F\o lner \sq\ $\F=(F_n)_{n\in\N}$ and a point $x\in\Lambda_1^G$ which is $\F$-generic for an \im\ $\mu$. We can choose a sub\sq\ $\F_\circ=(F_{n_k})_{k\in\N}$ of $\F$ so that the following conditions are satisfied: 
\begin{itemize}
	\item if $\bar F_{n_k}$ denotes the $\CT_k$-saturation of $F_{n_k}$ then 
	$\underset{k\to\infty}{\lim}\frac{|\bar F_{n_k}\setminus F_{n_k}|}{|F_{n_k}|}=0$,
 	\item $\underset{k\to\infty}{\lim}\frac{|\bar F_{n_1}\cup \bar F_{n_2}\cup\cdots\cup 
	\bar F_{n_{k-1}}|}{|F_{n_k}|}=0$.
\end{itemize}
The above conditions imply that the \sq\ of sets $\F'_\circ=(F'_{n_k})_{k\in\N}$, where $F'_{n_1}=\bar F_{n_1}$ and
$$
F'_{n_k} = \bar F_{n_k}\setminus(\bar F_{n_1}\cup \bar F_{n_2}\cup\cdots\cup \bar F_{n_{k-1}}), \ \ k\ge 2,
$$
is a F\o lner \sq\ equivalent to $\F_\circ$. In particular, $\F_\circ$-genericity of a point for some \im\ is equivalent to its $\F'_\circ$-genericity for that measure. So, it suffices to find an element $y\in\Lambda_2^G$ such that $(x,y)$ is $\F'_\circ$-generic for $\xi$. The advantage of $\F'_\circ$ over $\F_\circ$ is that $\F'_\circ$ consists of disjoint sets, each $F'_{n_k}$ being a union of tiles of $\CT_k$.

Passing, if necessary, to a sub\sq\ of $(n_k)_{k\in\N}$ we can also assume that
\begin{enumerate}
	\item $\frac{|S_{k,\max}|M_k}{|F'_{n_k}|}<\frac{\delta_k}4$,	
	\item the double block 
	$(B_{F'_{n_k}},D_{F'_{n_k}}):=(x,\CT_k)|_{F'_{n_k}}\in(\Lambda_1\times{\rm 
	V}_k)^{F'_{n_k}}$) satisfies
	$$
	\dist\bigl((B_{F'_{n_k}},D_{F'_{n_k}}),\zeta_k\bigr)<\tfrac{\delta_k\min\{u(S):S\in\CS_k\}}
	{4|S_{k,\max}|M_k},
	$$
	where	$\zeta_k=\mu\vee\omega_k$ is a joining of $\mu$ with some \im\ $\omega_k$ supported 
	by the orbit closure of $\CT_k$ (and $u(\cdot)$ is the function appearing the definition 
	of $\dist$).
\end{enumerate}
The condition (2) can be fulfilled due to Proposition \ref{fakty} (3), and since $x$ is $\F$-generic (hence $\F_\circ$-generic and thus $\F'_\circ$-generic) for $\mu$. 

The continuation of the proof depends on the following technical (and admittedly cumbersome) lemma.

\begin{lem}\label{pomoc}
For each $k\in\N$ let $B_{F'_{n_k}}=x|_{F'_{n_k}}$. There exists a block $C_{F'_{n_k}}\in\Lambda_2^{F'_{n_k}}$ such that the empirical measure generated by the double block $(B_{F'_{n_k}},C_{F'_{n_k}})$ is $\eps_k$-close to $\xi$: $\dist
\bigl((B_{F'_{n_k}},C_{F'_{n_k}}),\xi\bigr)\le\eps_k$.
\end{lem}

\begin{proof}
Since we will work with a fixed index $k$, we can simplify our notation. And so, $F$ will stand for $F'_{n_k}$ and we will write $\CT$, $\CS$, $\eps$, $\delta$, $\rm V$, $M$, $\zeta$, $\omega$, etc., skipping the subscript $k$. The block $\CT|_F$ will be written as $D_F$ (hence $(x,\CT)|_F=(B_F,D_F)$).

We are given two joinings of $\mu$: $\xi=\mu\vee\lambda$ on $(\Lambda_1\times\Lambda_2)^G$ and $\zeta=\mu\vee\omega$ on $(\Lambda_1\times{\rm V})^G$. We let $\eta$ be any \im\ supported by $(\Lambda_1\times\Lambda_2\times{\rm V})^G$ whose marginal on the first two coordinates is $\xi$ and that on the first and last coordinates is $\zeta$. For example, $\eta$ can be the joining of $\xi$ and $\zeta$ relatively independent over the common factor $\mu$ on the first coordinate. 

The block $(B_F,D_F)$ is partitioned by the tiling $\CT$ into ``small'' (still much larger than ``pivotal'') blocks whose domains are the tiles of $\CT$ (which have shapes $S\in\S$). Let $(B_0,D_0)\in(\Lambda_1\times{\rm V})^S$ be a double block which occurs at least once in $(B_F,D_F)$. Notice that $D_0$ is determined by the fact that its domain is the shape $S$; $D_0$ has the symbol $``S"$ at its center and $``0"$ everywhere else. We will write $D_S$ instead of $D_0$, to emphasize that this block is determined by its shape $S$. Consider now all occurrences of $(B_0,D_S)$ inside $(B_F,D_F)$. Clearly, due to the presence of $D_S$, $(B_0,D_S)$ can occur only over the tiles of $\CT$ which have the shape $S$. The joining $\zeta$ assigns to this double block some value $\zeta(B_0,D_S)$. 

For convenience, we will use the following notational convention concerning triple blocks: for a finite set $K\subset G$, $B\in\Lambda_1^K$ and $D\in{\rm V}^K$, we will write
$$
[B,*,D]:=\bigcup_{C\in{\Lambda_2^K}}[B,C,D].
$$
Similarly, we will write $[B,C,*]:=\bigcup_{D\in{\rm V}^K}[B,C,D]$. We also remind the reader that in accordance to our convention, when applying measures to cylinders, we will skip square brackets. For example, we will write $\eta(B,*,D)$ instead of $\eta([B,*,D])$.

The value $\zeta(B_0,D_S)$ can be also written as $\eta(B_0,*,D_S)$. The joining $\eta$ assigns values $\eta(B_0,C_0,D_S)$ to all triple blocks $(B_0,C_0,D_S)$, where $C_0$ ranges over $\Lambda_2^{S}$. Consider the following probability vector of conditional probabilities:
$$
\mathsf P_{(B_0,D_S)}=\left\{\frac{\eta(B_0,C_0,D_S)}{\eta(B_0,*,D_S)}:\ C_0\in\Lambda_2^S\right\}.
$$
We will construct the block $C_F$ so that the conditional probabilities are optimally approximated by the frequencies of $(B_0,C_0,D_S)$ in $(B_F,C_F,D_F)$. To this end, we denote by $N_F(B_0,D_S)$ the (nonzero) number of all occurrences of $(B_0,D_S)$ within $(B_F,D_F)$ and then we approximate, in a best possible way, the vector $\mathsf P_{(B_0,D_S)}$ by a probability vector whose entries are rational numbers with the denominator $N_F(B_0,D_S)$. We denote the respective numerators by $N_F(B_0,C_0,D_S)$. Observe that, for each $C_0\in\Lambda_2^S$, we have
\begin{equation}\label{szesc}
\left|\frac{N_F(B_0,C_0,D_{S})}{N_F(B_0,D_{S})}-\frac{\eta(B_0,C_0,D_S)}{\eta(B_0,*,D_S)}\right|\le\frac1{N_F(B_0,D_{S})}.
\end{equation}
Finally, we place the blocks $C_0\in\Lambda_2^{S}$ in the middle layer of the future triple block $(B_F,C_F,D_F)$ exactly at the positions of the occurrences of $(B_0,D_{S})$ within $(B_F,D_F)$, subject to the condition that each block $C_0$ should be used exactly $N_F(B_0,C_0,D_{S})$ times. This is possible since
$$
\sum_{C_0\in\Lambda_2^S}N_F(B_0,C_0,D_{S})=N_F(B_0,D_{S}). 
$$
Finally, we perform the above procedure for any double block $(\tilde B_0,\tilde D_0)$ appearing at least once in $(B_F,D_F)$. Once this is done, the definition of $C_F$ is completed. 
\smallskip

We shall now check that the empirical measure associated to the double block $(B_F,C_F)$ is indeed $\eps$-close to the joining $\xi$. Consider any pivotal double block $(B,C)$. The frequency of $(B,C)$ in $(B_F,C_F)$ can be evaluated by averaging its frequencies in the blocks $(B_0,C_0,D_S)$ (equivalently, in $(B_0,C_0)$) with weights proportional to the sizes of the shapes $S$ and the frequencies of $(B_0,C_0,D_S)$ in $(B_F,C_F,D_F)$. In this manner we will miss some occurrences of $(B,C)$, namely the ones not entirely contained in some $(B_0,C_0,D_S)$ (but these occurrences constitute a small fraction). We present now these ideas in a more rigorous way. The equalities below hold up to error terms indicated by the expressions to the right of ``$\pm$'' sign. The sum sign $\sum$ stands for the triple sum $\sum_{S\in\S}\sum_{B_0\in\Lambda_1^S}\sum_{C_0\in\Lambda_2^S}$. Observe also that the number of terms in $\sum$ is equal to $M$.
\begin{multline*}
\Fr_{(B_F,C_F)}(B,C)=\\
\sum|S|\Fr_{(B_F,C_F,D_F)}(B_0,C_0,D_S)\cdot\Fr_{(B_0,C_0)}(B,C)\overset{\eqref{jeden}}\pm\frac\delta4=\\
\sum|S|\tfrac{N_F(B_0,C_0,D_S)}{|F|}\cdot\Fr_{(B_0,C_0)}(B,C)\pm\frac\delta4=\\
\sum|S|\cdot\Fr_{(B_0,C_0)}(B,C)\cdot \frac{N_F(B_0,C_0,D_S)}{N_F(B_0,D_S)}\cdot \frac{N_F(B_0,D_S)}{|F|}\pm\frac\delta4=\\
\sum|S|\cdot\Fr_{(B_0,C_0)}(B,C)\cdot \left(\frac{\eta(B_0,C_0,D_S)}{\eta(B_0,*,D_S)}\overset{\eqref{szesc}}\pm\frac1{N_F(B_0,D_S)}\right)\cdot \frac{N_F(B_0,D_S)}{|F|}\pm\frac\delta4=\\
\sum|S|\cdot\Fr_{(B_0,C_0)}(B,C)\cdot \frac{\eta(B_0,C_0,D_S)}{\eta(B_0,*,D_S)}\cdot\Fr_{(B_F,D_F)}(B_0,D_S)\pm\\
\sum\tfrac{|S|}{|F|}\cdot\Fr_{(B_0,C_0)}(B,C)\pm\frac\delta4=\dots
\end{multline*}
The last sum does not exceed $\frac{|S_{\max}|}{|F|}\cdot M$, and so, by assumption (1), it does not exceed $\frac\delta4$. Moreover, by assumption (2) and the formula defining $\dist$, the frequency $\Fr_{(B_F,D_F)}(B_0,D_S)$ differs from $\zeta(B_0,D_S)$ (i.e.\ from $\eta(B_0,*,D_S)$) by at most $\tfrac\delta{4|S_{\max}|M}$, and so we can continue as follows:
\begin{multline*}
\!\!\!\!\!\!\!\ldots=\sum|S|\cdot\Fr_{(B_0,C_0)}(B,C)\cdot \frac{\eta(B_0,C_0,D_S)}{\eta(B_0,*,D_S)}\cdot\left(\eta(B_0,*,D_S)\pm \frac\delta{4|S_{\max}|M}\right)\pm\frac{2\delta}4=\\
\sum|S|\cdot\Fr_{(B_0,C_0)}(B,C)\cdot \eta(B_0,C_0,D_S)\pm\\
\sum|S|\cdot\Fr_{(B_0,C_0)}(B,C)\cdot \frac{\eta(B_0,C_0,D_S)}{\eta(B_0,*,D_S)}\cdot\frac\delta{4|S_{\max}|M}\pm\frac{2\delta}4.
\end{multline*}
The last sum, this time, does not exceed $\frac\delta{4M}\cdot M=\frac\delta4$.
We have proved that
\begin{equation}\label{joe}
\Fr_{(B_F,C_F)}(B,C)=\sum_{S\in\S}\sum_{B_0\in\Lambda_1^S}\sum_{C_0\in\Lambda_2^S}|S|\cdot\Fr_{(B_0,C_0)}(B,C)\cdot \eta(B_0,C_0,D_S)\pm\frac{3\delta}4.
\end{equation}

Now we are going to approximate $\xi(B,C)$. Let $K\in\Pi_\eps$ be the domain of $(B,C)$
(see Proposition \ref{fakty} (1)). Consider a point $(x,y)$ belonging to the cylinder corresponding to the double block $(B,C)$. For any $z\in\rm V^G$, $(x,y,z)|_K = (B,C,D)$ for some $D\in{\rm V}^K$. We are only interested in elements $z$ contained in the orbit closure of $\CT$. Notice that every such $z$ represents a tiling with the set of shapes $\CS$. For every such $z$, the coordinate $e$ is contained in a tile $T$ of some shape $S\in\S$. Then $T=Ss^{-1}$, where the center $s^{-1}$ is the inverse of some $s\in S$ (because $e\in T$). So $z(s^{-1})=``S"$ which means that $z\in s([``S"])$. 
(Recall that $[``S"]$ denotes the cylinder associated to the one-symbol block with the symbol $``S"$.)

Thus, we have the following disjoint union representation of the cylinder corresponding to $(B,C)$:
\begin{multline*}
[B,C]=\bigcup_{S\in\S}\bigcup_{s\in S}[B,C]\cap s([``S"])= \\
\bigcup_{S\in\S}\bigcup_{s\in S}\bigcup_{B_0\in\Lambda_1^S}\bigcup_{C_0\in\Lambda_2^S}[B,C,*]\cap s([B_0,C_0,D_S])=\\
\bigcup_{S\in\S}\bigcup_{B_0\in\Lambda_1^S}\bigcup_{C_0\in\Lambda_2^S}\ \bigcup_{s\in S}[B,C,*]\cap s([B_0,C_0,D_S]),
\end{multline*}
where $(B_0,C_0)$ ranges over all double blocks with the domain $S$. 

We will break the union over $s\in S$ in two parts: $s\in S_K$ and $s\in S\setminus S_K$. If $s\in S_K$ then the intersection of cylinders $[B,C,*]\cap s([B_0,C_0,D_S])$ either equals $s([B_0,C_0,D_S])$ or it is empty, depending on whether $(B,C)$ occurs in $(B_0,C_0)$ with an anchor at $s$ or not. The number of elements $s\in S$ for which $(B,C)$ occurs in $(B_0,C_0)$ with an anchor at $s$ equals $|S|\cdot\Fr_{(B_0,C_0)}(B,C)$. Finally, recall that $|S\setminus S_K|\le|S|\frac\delta4$.

Thus, using the invariance of $\eta$, we can write
\begin{multline*}
\xi(B,C)=\eta(B,C,*)=\\
\sum_{S\in\S}\sum_{B_0\in\Lambda_1^S}\sum_{C_0\in\Lambda_2^S}|S|\cdot(\Fr_{(B_0,C_0)}(B,C)\pm\tfrac\delta4)\cdot\eta(B_0,C_0,D_S)=\\
\sum_{S\in\S}\sum_{B_0\in\Lambda_1^S}\sum_{C_0\in\Lambda_2^S}|S|\cdot(\Fr_{(B_0,C_0)}(B,C))\cdot\eta(B_0,C_0,D_S)\pm\frac\delta4
\end{multline*}
(we have used the equality $\sum_{S\in\S}\sum_{B_0\in\Lambda_1^S}\sum_{C_0\in\Lambda_2^S}|S|\eta(B_0,C_0,D_S)=1$). 
Combining this approximation with \eqref{joe} we conclude that
$$
\Fr_{(B_F,C_F)}(B,C)=\xi(B,C)\pm\delta.
$$
By Proposition \ref{fakty} (1), the empirical measure generated by $(B_F,C_F)$ is $\eps$-close to~$\xi$.
\end{proof}

We point out that the method of proof of the above lemma bears resemblance to the combinatorial method of constructing normal numbers pioneered by Copeland and Erd\H{o}s \cite{CE}.
\smallskip

We are just one easy step away from completing the proof of Theorem \ref{generic}  for groups. In Lemma~\ref{pomoc}, for each $k\in\N$, we have created a block $C_{F'_{n_k}}$ such that the block $(B_{F'_{n_k}},C_{F'_{n_k}})$ is $\eps_k$-close to $\xi$. We now define $y\in\Lambda_2^G$ by first defining its restriction to the (disjoint) union of the F\o lner sets $F'_{n_k}$ by the formula $y(g) = C_{F'_{n_k}}(g)$, where $k$ is the unique index such that $g\in F'_{n_k}$. We then define $y$ on the remaining part of $G$ completely arbitrarily. Since $(x,y)|_{F'_{n_k}}=(B_{F'_{n_k}},C_{F'_{n_k}})$, and the empirical measure generated by this double block is $\eps_k$-close to $\xi$, $(x,y)$ is $\F'$-generic, hence $\F_\circ$-generic and thus $\F$-quasi-generic, for $\xi$.
\smallskip

Lastly, let $G$ be a general countable cancellative amenable semigroup. It is known that any such semigroup is embeddable in a group (see \cite[Volume 1, Chapter 1]{CP} and \cite[Proposition 1.23]{P}). As shown in \cite[Theorem 2.12]{BDM}, for any such semigroup there exists a countable amenable group $\bar G$ containing $G$ as a subsemigroup, such that any F\o lner sequence $\F$ in $G$ is a F\o lner sequence in $\bar G$. We extend the symbolic element $x\in\Lambda^G$ (which is $\F$-generic for $\mu$) to an element $\bar x\in\Lambda^{\bar G}$ by assigning the values $\bar x(g)$ for $g\in\bar G\setminus G$ completely arbitrarily. Such an element $\bar x$ is $\F$-generic for the unique \inv\ (under the shift action of $\bar G$) extension $\bar\mu$ of $\mu$ onto $\Lambda^{\bar G}$. By the proved above version of the theorem for groups, we have an element $\bar y\in\Lambda^{\bar G}$ such that $(\bar x,\bar y)$ is $\F$-generic for the unique extension $\bar\xi$ of $\xi$ ($\bar\xi$ is a joining of $\bar\mu$ and the unique extension $\bar\nu$ of $\nu$). Obviously, if $y$ denotes the restriction of $\bar y$ to $G$, the pair $(x,y)$ is $\F$-generic for $\xi$ (under the shift action of $G$). We are done.
\end{proof}

\subsection{Making a quasi-generic point generic}\label{msgg}
In this subsection we deal with the last key step leading to the proof of Theorem \ref{main}. This step concerns a passage from an $\F$-quasi-generic point to an $\F$-generic point. 

We will need the following general (and standard) fact about extreme points in compact convex sets. We supply a proof for reader's convenience.

\begin{fact}\label{extreme}
Let $\M$ be a compact subset of a locally convex linear space $(V,\rho)$, where $\rho$ is a convex metric. Let $\nu$ be an extreme point of $\M$. For any $\eps>0$ there exists a $\delta=\delta(\eps)>0$ such that any element in $B_\delta(\nu)$ (the open $\delta$-ball around $\nu$) \emph{cannot} be decomposed as a convex combination of elements of $\M$ with larger than or equal to $\eps$ contribution of points lying outside $B_\eps(\nu)$. Formally, if 
\begin{equation}\label{cc}
\mu=\sum_{i=1}^n\alpha_i\mu_i+\sum_{j=1}^m\beta_j\nu_j,
\end{equation}
where $\{\alpha_1,\alpha_2,\dots,\alpha_n,\beta_1,\beta_2,\dots,\beta_m\}$ is 
a probability vector, $\sum_{i=1}^n\alpha_i\ge\eps$, $\mu_i, \nu_j\in \M$ ($i=1,2\dots,n, \ j=1,2\dots,m$) and $\rho(\mu_i,\nu)\ge\eps$ for all $i=1,2,\dots,n$, then $\rho(\mu,\nu)\ge\delta$.
\end{fact}

\begin{proof}
The convex combination \eqref{cc} can be viewed as the barycenter of the following probability measure on $\M$:
$$
\mathsf P =\sum_{i=1}^n\alpha_i\delta_{\mu_i}+\sum_{j=1}^m\beta_j\delta_{\nu_j},
$$
with $\mathsf P(D)\ge\eps$, where $D=\M\setminus B_\eps(\nu)$. If the statement was not true, we could find a \sq\ of measures $(\mathsf P_k)_{k\in\N}$ satisfying $\mathsf P_k(D)\ge\eps$ and whose barycenters converge to $\nu$. By the weak* compactness of the space of all probability measures on $\M$, we can assume (passing to a sub\sq) that $\mathsf P_k$ converge to a probability measure $\mathsf R$ on $\M$. Because $D$ is closed, we would have $\mathsf R(D)\ge\eps$. On the other hand, by continuity of the map associating to each probability measure on $\M$ its barycenter, we would have that the barycenter of $\mathsf R$ equals $\nu$. This is impossible because $\nu$, being an extreme point of $\M$, is the barycenter of but one probability measure on $\M$, the Dirac measure $\delta_\nu$, and since $\delta_\nu(D)=0$, we have that $\mathsf R\neq\delta_\nu$, a contradiction.
\end{proof}

For a given symbolic element $y\in\Lambda^G$ and a F\o lner \sq\ $\F$ in $G$, a symbolic element $y'\in\Lambda^G$ will be called an \emph{$\F$-modification} of $y$ if the set $\{g\in G:y(g)\neq y'(g)\}$ has $\F$-density zero. We can now state the main theorem of this subsection.  

\begin{thm}\label{semi}{\rm(Cf. \cite[Theorem 1]{Kamae})}
Let $\F=(F_n)_{n\in\N}$ be a F\o lner \sq\ in a countable cancellative amenable semigroup $G$. Suppose that $y\in\Lambda^G$ is $\F$-quasi-generic for an \emph{ergodic} measure $\nu$. Let $\F_\circ=(F_{n_k})_{k\in\N}$ be a sub\sq\ of $\F$ such that $\F_\circ$-generic for $\nu$. Then there exists an $\F_\circ$-modification $y'$ of $y$ which is $\F$-generic for $\nu$. 
\end{thm}

\begin{rem}
Easy examples for $G=\N_0$ (or $\Z$) show that in the formulation below the ergodicity assumption is essential.
\end{rem}

\begin{proof} We will prove Theorem \ref{semi} for groups. The argument behind the reduction to the case of groups is very similar to the one provided at the end of the proof of Theorem \ref{generic} and we will skip the details. We will be using the tiling $\Theta$ described in Theorem~\ref{tiln}. The \sq\ of $\Theta$-saturations, $\F^\Theta=(F_n^\Theta)_{n\in\N}$, is a F\o lner \sq\ equivalent to $\F$. In particular, it suffices to prove the theorem for the new F\o lner \sq\ $\F^\Theta$ (and its sub\sq\ $\F_\circ^\Theta=(F_{n_k}^\Theta)_{k\in\N}$). The advantage of $(F_n^\Theta)_{n\in\N}$ over $\F$ is that for every $n$, $F_n^\Theta$ is a union of the elements of an auxiliary F\o lner \sq\ $\Theta$ (which is disjoint and exhaustive). For the rest of this proof, for any $n\in\N$, $F_n$ will denote $F_n^\Theta$. 

Fix an $\eps>0$ and divide the tiles $\theta$ of $\Theta$ in two classes: ``bad'', i.e.\ such that $\dist(y|_\theta,\nu)>2\eps$, and the remaining ones (``good''). We will now prove the following claim: 
\begin{equation}\label{siudym}
\text{\parbox{.85\textwidth}{The union $U_{\mathsf{bad}}$ of the ``bad'' tiles has $\F_\circ$-density zero. 
}}
\end{equation}
Suppose that $\bar d_{\F_\circ}(U_{\mathsf{bad}})=\gamma>0$. Note that $\nu$ is an extreme point of $\mgl$. Proposition \ref{extreme} applied to $\nu$ and $\eps=\frac\gamma4$ provides a number $\delta=\delta(\eps)$. Let $\Theta'\subset\Theta$ be a family of tiles $\theta$ which satisfies the following two conditions: 
\begin{enumerate}
\item For any $\theta\in\Theta'$, for any block $B\in\Lambda^\theta$, we have $\dist\bigl(B,\mgl\bigr)<\tfrac\delta4$.
\item For any finite collection of distinct tiles $\theta_1,\theta_2,\cdots,\theta_q\in\Theta'$ ($q\in\N$) and any blocks $B_i\in \Lambda^{\theta_i}$ ($i=1,2\dots,q$), the concatenation $B_0=B_1B_2\dots B_q$ satisfies
$$
\dist\Bigl(B_0,\sum_{i=1}^q\tfrac{|\theta_i|}{\sum_{j=1}^q|\theta_j|}\hat B_i\Bigr)<\frac\delta4.
$$
(Recall that $\hat B$ denotes the empirical measure associated with a block $B$ and that when applying $\dist$ we identify blocks with their empirical measures.) 
\end{enumerate}
By Proposition \ref{fakty} (3) and (6), and since $\Theta$ is in fact a F\o lner \sq, such a family $\Theta'$ exists and contains all but finitely many tiles of $\Theta$. 
Thus, for large enough $k$, if $\bar F_{n_k}$ denotes the union of the tiles contained in $F_{n_k}$ which satisfy conditions (1) and (2) above, then the fraction $\frac{|\bar F_{n_k}|}{|F_{n_k}|}$ is as close to 1 as we wish. We choose $k$ so that this fraction is larger than $1-\eps$, and we choose $k$ also large enough to ensure that 
\begin{equation}\label{cuatro}
\dist(y|_{\bar F_{n_k}},y|_{F_{n_k}})<\tfrac\delta4
\end{equation}
(we are using Proposition~\ref{fakty}~(2)). Since $y$ is $\F_\circ$-generic for $\nu$, we can also assume that 
\begin{equation}\label{cinco}
\dist(y|_{F_{n_k}},\nu)<\tfrac\delta4.
\end{equation}
Now, by (1), for every $\theta\subset\bar F_{n_k}$ there exists an \im\ $\mu_\theta$ such that $\dist(y|_\theta,\mu_\theta)<\frac\delta4$. Note that if $\theta$ is a ``bad'' tile then
\begin{equation}\label{dos}
\dist(\mu_\theta,\nu) \ge \dist(y|_\theta,\nu)-\dist(\mu_\theta,y|_\theta)\ge 2\eps-\tfrac\delta4=\eps.
\end{equation}
By (2), we obtain
\begin{equation}\label{seis}
\dist\Bigl(y|_{\bar F_{n_k}}, \sum_{\theta\subset\bar F_{n_k}}\frac{|\theta|}{|\bar F_{n_k}|}\,y|_\theta\Bigr)<\tfrac\delta4,
\end{equation}
and by convexity of the metric $\dist$, we also have
\begin{equation}\label{siete}
\dist\Bigl(\sum_{\theta\subset\bar F_{n_k}}\frac{|\theta|}{|\bar F_{n_k}|}\,y|_\theta,\mu\Bigr)<\tfrac\delta4,
\end{equation}
where
\begin{equation}\label{mu}
\mu=\sum_{\theta\subset\bar F_{n_k}}\frac{|\theta|}{|\bar F_{n_k}|}\,\mu_\theta.
\end{equation}
By \eqref{cuatro}, \eqref{cinco}, \eqref{seis}, \eqref{siete} and the triangle inequality, we get 
\begin{equation}\label{uno}
\dist(\mu,\nu)<\delta.
\end{equation}
Since $\bar d_{\F_\circ}(U_{\mathsf{bad}})=\gamma$, we can choose $k$ large enough so that (in addition to the previous conditions), we have
$$
\frac{|F_{n_k}\cap U_{\mathsf{bad}|}}{|F_{n_k}|}\ge \frac\gamma2=2\eps.
$$
With this choice of $k$, since $\frac{|\bar F_{n_k}|}{|F_{n_k}|}>1-\eps$, we easily see that 
$$
\frac{|\bar F_{n_k}\cap U_{\mathsf{bad}|}}{|\bar F_{n_k}|}\ge \frac\gamma2-\eps=\eps.
$$
 This implies that in the convex combination \eqref{mu} defining $\mu$, the contribution of measures $\mu_\theta$ associated to ``bad'' tiles $\theta$ is at least $\eps$. In view of \eqref{uno} and \eqref{dos}, this stands in contradiction with Fact \ref{extreme}. This ends the proof of \eqref{siudym}.
\smallskip

For each tile $\theta$ of $\Theta$ let $B_\theta$ be the block such that 
$$
\dist(B_\theta,\nu)=\min\{\dist(B,\nu):B\in\Lambda^\theta\}.
$$
Because $\Theta$ (taken in some order) is a F\o lner \sq, Lemma \ref{bloki} implies that $\dist(B_\theta,\nu)\to 0$ as the index of $\theta$ goes to infinity. So, the point $\tilde y$ defined by the relation $\tilde y|_\theta=B_\theta$ for every $\theta\in\Theta$, is $\Theta$-generic for $\nu$ and thus also $\F$-generic for $\nu$ (recall that each $F_n$ is a union of the tiles of $\Theta$). However, $\tilde y$ need not be an $\F_\circ$-modification of $y$. To create such a modification we will define $y'$ by the same rule as $\tilde y$ except that it will coincide with $y$ on the tiles on which $y|_\theta$ is already sufficiently close to $\nu$. To this end, we fix a \sq\ of positive numbers $(\eps_i)_{i\in\N}$ that is decreasing to zero. We let $A_i$ denote the union of all tiles that are ``good'' for $\eps_i$, i.e.\ such that $\dist(y|_\theta,\nu)\le2\eps_i$. The sets $A_i$ are nested (that is, $A_i \supset A_{i+1}$) and have empty intersection. On the other hand, by \eqref{siudym}, each of the sets $A_i$ has the $\F_\circ$-density $1$. Thus, for each $i$ and all but finitely many $k$, we have
\begin{equation}\label{frp}
\frac{|F_{n_k}\cap A_i|}{|F_{n_k}|}>1-\eps_i.
\end{equation}
It follows that, for all but finitely many $k$ \eqref{frp} holds for a nonempty (yet finite) set of indices $i$. For each $k\in\N$, let $i_k$ be the maximal index $i$ for which \eqref{frp} holds. Then $i_k$ grows to infinity as $k$ increases. Observe that the set $A=\bigcup_{k\in\N} (F_{n_k}\cap A_{i_k})$ is a union of tiles of $\Theta$. We now define $y'$ by setting 
$$
y'|_\theta=\begin{cases} y|_\theta,&\text{ if }\theta\subset A,\\
B_\theta,&\text{ otherwise}.
\end{cases}
$$
It is immediately seen that $y'$ is an $\F_\circ$-modification of $y$: within each set $F_{n_k}$, we have $y'=y$ at least on $F_{n_k}\cap A_{i_k}$ of cardinality at least $(1\!-\!\eps_{i_k})|F_{n_k}|$. It remains to check that $y'$ is $\F$-generic for $\nu$. This will be accomplished once we show that, given $\eps>0$, only finitely many tiles $\theta$ satisfy $\dist(y'|_\theta,\nu)>\eps$. Actually, in this manner we will show that $y'$ is $\Theta$-generic for $\nu$. However, since each $F_n$ is a disjoint union of tiles of $\Theta$, this will immediately imply (with the help of Proposition \ref{fakty} (6)) that $y'$ is $\F$-generic for $\nu$. By Lemma \ref{bloki} and since $\Theta$ is a F\o lner \sq, it is clear that among the tiles $\theta$ such that $y'|_\theta=B_\theta$, only finitely many satisfy $\dist(y'|_\theta,\nu)>\eps$. For each of the remaining tiles $\theta$ we have $y'|_\theta=y|_\theta$ and, by the construction, there exists an index $k$ such that $\theta$ is contained in $F_{n_k}$ and is ``good'' for $\eps_{i_k}$, meaning $\dist(y|_\theta,\nu)<2\eps_{i_k}$. If $2\eps_{i_k}\le\eps$, then $\dist(y'|_\theta,\nu)\le\eps$, and we do not need to count such tiles. The remaining tiles are contained in the sets $F_{n_k}$ with $k$ such that $2\eps_{i_k}>\eps$. There are finitely many such integers $k$ and hence only finitely many tiles fall in this category.
\end{proof}

\subsection{Proof of the first main result}\label{pmr}

\begin{proof}[Proof of Theorem \ref{main}]
(4)$\implies$(5): If $A$ has zero $\F$-density, then any $\F$-normal element $y\in\Lambda^G$ can be altered along $A$ in such a way that the resulting element $y'$ is not simply $\F$-normal along $A$. On the other hand the modified element $y'$ will remain $\F$-normal.
If $A$ is not $\F$-deterministic, then, for some sub\sq\ $\F_\circ=(F_{n_k})_{k\in\N}$ of $\F$, the element $x=\mathbbm 1_A$ is $\F_\circ$-generic for an \im\ $\mu$ of positive entropy (and clearly $A$ has positive $\F_\circ$-density). By Theorem~\ref{depzc}, there exists a joining $\xi=\mu\times\lambda$ which makes the zero-coordinate partitions $\{[0],[1]\}$ of $\{0,1\}^G$ and $\Lambda$ of $\Lambda^G$ not stochastically independent. In particular, for some $a\in\Lambda$, 
$$
\xi\bigl([*,a]\,\bigl|\bigr.\,[1,*]\bigr)\neq\tfrac1{|\lambda|},
$$
where $\xi(\cdot|\cdot)$ denotes the conditional measure of one set given another, and
$$
[*,a]:=\{0,1\}^G\times [a], \ \ \ [1,*]:=[1]\times\Lambda^G.
$$
By Theorem~\ref{generic}, there exists an element $y\in\Lambda^G$ such that the pair $(x,y)$ is $(F_{n_k})$-quasi-generic. Hence there exists a \sq\ $(n_{k_i})_{i\in\N}$ such that $y$ $(F_{n_{k_i}})$-generic for $\xi$. In particular, $y$ is $(F_{n_{k_i}})$-generic for $\lambda$. Because $\lambda$ is ergodic, Theorem \ref{semi} provides an $(F_{n_{k_i}})$-modification $y'$ of $y$, which is $\F$-generic for $\lambda$, i.e.\ $\F$-normal. The pair $(x,y')$ remains $(F_{n_{k_i}})$-generic for $\xi$. Thus, the limit 
$$
\lim_{i\to\infty} \frac{|\{g\in (F_{n_{k_i}}\cap A): y'(g)=a|}{|F_{n_{k_i}}\cap A|}
$$ 
exists and is different from $\frac1{|\Lambda|}$, and hence the limit 
$$
\lim_{n\to\infty} \frac{|\{g\in (F_n\cap A): y'(g)=a|}{|F_n\cap A|}
$$ 
either does not exist or is different from $\frac1{|\Lambda|}$. In either case, $y'$ is not simply $\F$-normal along $A$ and thus $A$ does not preserve simple $\F$-normality.

\smallskip\noindent
(2)$\implies$(4): Orbit-$\F$-normality along $A$ of some $y\in\Lambda^G$ implies in particular that the ``thin cylinders'', i.e. cylinders associated with single symbols, are visited by the orbit of $y$ with appropriate relative $\F$-densities along $A$, which is precisely the simple $\F$-normality along $A$ of $y$. 

\smallskip\noindent
(3)$\implies$(4): 
The set $\{e\}$ is always visible in $A$ (for some sets $A$ this is the only visible set). 
So, the definition of block-$\F$-normality along $A$ applies to thin cylinders, which is precisely the simple $\F$-normality~along~$A$.

\smallskip\noindent
(1)$\implies$(2): Assume $A$ has positive lower $\F$-density and fails to preserve orbit-$\F$-normality. Then there exists some $\F$-normal element $y\in\Lambda^G$, a finite set $K\subset G$ and block $B\in\Lambda^K$, such that the following limit either does not exist or is different from $|\Lambda|^{-|K|}$:
\begin{equation}\label{lim}
\lim_{n\to\infty}\frac{|\{g\in F_n\cap A: y|_{Kg}\approx B\}|}{|F_n\cap A|}.
\end{equation}
We can thus choose a sub\sq\ $\F_\circ=(F_{n_k})_{k\in\N}$  of $\F$ such that
\begin{itemize}
	\item $A$ has positive $\F_\circ$-density,
	\item $(x,y)$ is $\F_\circ$-generic for some joining $\xi=\mu\vee\lambda$, where $\mu$ is 
	an \im\ for which $x=\mathbbm 1_A$ is $\F$-quasi-generic,  
	\item the limit \eqref{lim} taken along $n_k$ exists and is different from 
	$|\Lambda|^{-|K|}$.
\end{itemize}
Since $A$ has positive $\F_\circ$-density, we have that $\mu([1])>0$.
The (different from $|\Lambda|^{-|K|}$) limit equals the conditional probability $\xi([*,B]\,\bigl|\bigr.\,[1,*])$. More formally, we have:
\begin{equation*}
|\Lambda|^{-|K|}\neq\lim_{k\to\infty}\frac{|\{g\in F_{n_k}\cap A: y|_{Kg}\approx B\}|}{|F_{n_k}\cap A|}= \xi([*,B]\,\bigl|\bigr.\,[1,*]).
\end{equation*}
Since $\lambda(B)=|\Lambda|^{-|K|}$, the joining $\xi$ is not the independent joining. This implies that $\mu$ has positive entropy 
(cf. \cite[page 14]{Fu2}), and we conclude that $A$ is not $\F$-deterministic.

\smallskip\noindent
(1)$\implies$(3): Now suppose that a set $A$ of positive lower $\F$-density does not preserve block-$\F$-normality. Then there exists an $\F$-normal element $y$, an $\F$-visible in $A$ finite set $K\subset G$ and a block $B\in\{0,1\}^K$ for which the following limit either does not exist or is different from $|\Lambda|^{-|K|}$:
\begin{equation}\label{lim1}
\lim_{n\to\infty} \frac{|\{g\in F_n\cap A_K, \  y|_{Kg}\approx B\}|}{|F_n\cap A_K\}|}.
\end{equation}
We can choose a subsequence $\F_\circ=(F_{n_k})_{k\in\N}$  of $\F$ so that
\begin{itemize}
	\item $A$ has positive $\F_\circ$-density,
	\item $(x,y)$ is $\F_\circ$-generic for some joining $\xi=\mu\vee\lambda$, 
	\item the limit \eqref{lim1} taken along $n_k$ exists and is different from 
	$|\Lambda|^{-|K|}$,
	\item the core $A_K$ has $\F_\circ$-density (i.e.\ $\underline d_\F(A_K)=\overline d_\F(A_K)$). 
\end{itemize}
Because $K$ is $\F$-visible (and thus $\F_\circ$-visible) in $A$ and $A$ has positive $\F_\circ$-density, the $\F_\circ$-density of $A_K$ is positive. This implies that $\mu([C])>0$, where $C\in\{0,1\}^K$ is the block of just 1's. We have
\begin{equation*}
|\Lambda|^{-|K|}\neq\lim_{k\to\infty}\frac{|\{g\in F_{n_k}\cap A_K: y|_{Kg}\approx B\}|}{|F_{n_k}\cap A_K|}= \xi([*,B]\,\bigl|\bigr.\,[C,*]).
\end{equation*}
Thus the joining $\xi$ is not the independent joining. As before, $\mu$ must have positive entropy and hence $A$ is not $\F$-deterministic.
\end{proof}

By the examples provided below, (5) (positive upper $\F$-density and $\F$-determinism) does not imply (4) (simple $\F$-normality preservation), and neither (2) (orbit-$\F$-normality preservation) nor (3) (block-$\F$-normality preservation) implies (1) (positive lower $\F$-density and $\F$-determinism). So, without additional assumptions on $A$ (for example, that it has positive lower $\F$-density) or on $\F$ (for example, that it is nested and grows subexponentially) we cannot formulate a theorem with full equivalence. 

\begin{prop}\label{5n2} In any countable cancellative amenable semigroup $G$ there exists a F\o lner \sq\ $\F$ and an $\F$-deterministic set $A$ of positive upper $\F$-density, which does not preserve simple $\F$-normality.
\end{prop}
\begin{proof} Let $\F=(F_n)_{n\in\N}$ consist of disjoint sets, and let $A'=\bigcup_{n\in\N}F_{2n-1}$. The set $A'$ is $\F$-deterministic, because the indicator function $\mathbbm1_{A'}\in\{0,1\}^G$ is $\F$-quasi-generic for only two measures, the pointmass concentrated at the fixpoint $\{0\}^G$ and the pointmass concentrated at the fixpoint $\{1\}^G$. Let $A''$ be the subset of $\bigcup_{n\in\N}F_{2n}$ such that,
$$
\lim_{n\to\infty}|A''\cap F_{2n}|=\infty \text{ \ \ and \ \ }\lim_{n\to\infty}\frac{|A''\cap F_{2n}|}{|F_{2n}|}=0.
$$
Define $A=A'\cup A''$. Since $A$ is an $\F$-modification of $A'$, it is $\F$-deterministic as well. Clearly, $\underline d_\F(A)=0$ and $\overline d_\F(A)=1$. Let $y\in\Lambda^G$ be $\F$-normal. Now, $y$ can be altered along $A''$ in such a way that the resulting element $y'$ is not simply $(F_{2n})$-normal along $A$. Then $y'$ is not simply $\F$-normal along~$A$. On the other hand, $y'$ is an $\F$-modification of $y$, and thus it remains $\F$-normal. This shows that $A$ does not preserve simple $\F$-normality.
\end{proof}

\begin{prop}\label{234n1}
There exists a F\o lner \sq\ $\F$ in $\N$ and a set $A\subset\N$ which preserves both orbit- and block-$\F$-normality and has lower $\F$-density zero.
\end{prop}
\begin{proof}
Let $A$ be the union over $k\in\N$ of the intervals \ $[(2k)!+1,(2k+1)!]$. The \sq\ $x=\mathbbm 1_A$ consists of alternating blocks of zeros and ones, each dominating (in terms of length) everything on its left. For each $n\in\N$, let $F_n=[1,n!]$. Let $\F,\ \F_{\mathsf{odd}}$ and $\F_{\mathsf{even}}$ be defined as $(F_n)_{n\in\N},\ (F_{2n+1})_{n\in\N}$ and $(F_{2n})_{n\in\N}$, respectively. Clearly, $A$ has lower $\F$-density zero (realized along $\F_{\mathsf{even}}$) and upper $\F$-density one (along $\F_{\mathsf{odd}}$). Now, let $y\in\Lambda^G$ be $\F$-normal (since $\F$ is a sub\sq\ of the classical F\o lner \sq, every classically normal \sq\ will do). 
Clearly, $y$ is $\F_{\mathsf{odd}}$-normal (as well as $\F_{\mathsf{even}}$-normal). 
Since $A$ has $\F_{\mathsf{odd}}$-density $1$, it preserves both orbit- and block-$\F_{\mathsf{odd}}$-normality. On the other hand, to see that $A$ is $\F_{\mathsf{even}}$-normality preserving, observe that, for any $n\in\N$, $F_{2n}\cap A=F_{2n-1}\cap A$, so that $\F_{\mathsf{odd}}$-normality of some $y\in\Lambda^G$ along $A$ implies $\F_{\mathsf{even}}$-normality of $y$ along $A$. Hence $A$ preserves both orbit- and block-$\F$-normality, even though its lower $\F$-density equals zero.
\end{proof}

\begin{rem} In view of Theorem \ref{1}, the above example is possible due to the fast growth of the cardinalities $|F_n|$.
\end{rem}

\section{Second main result: determinism = subexponential complexity}

In his book \cite{W2}, Weiss sketched a proof of a combinatorial characterization of completely deterministic sets in $\N$, which relates determinism to low subword complexity. In this section we establish a generalization of Weiss' criterion for countable cancellative amenable semigroups.
\smallskip

\subsection{The notion of complexity and its growth rate}
Let $G$ be a countable cancellative amenable semigroup and let $x\in\Lambda^G$ be a symbolic element. (As always, $\Lambda$ stands for a finite alphabet.)
If $\Lambda=\{0,1\}$, we can naturally identify $x$ with the set $\{g\in G:x(g)=1\}$. 

\begin{defn}\label{61}
Given a finite set $K\subset G$ and an arbitrary set $A\subset G$, we let $\mathbf C_x(K|A)$ denote the cardinality of the the collection of blocks over $K$ appearing in $x$ \emph{anchored at $A$}, that is 
$$
\mathbf C_x(K|A)=|\{B\in\Lambda^K: (\exists g\in A)\ x|_{Kg}\approx B\}|.
$$ 
If $A=G$, we will simply write $\mathbf C_x(K)$.
\end{defn}

\begin{defn}\label{coml}
We define the \emph{complexity} of $x\in\Lambda^G$ as the function $K\mapsto\ \mathbf C_x(K)$ on finite subsets of $G$. Given a F\o lner \sq\ $\F$ in $G$, we also define the $(\F,\eps)$-complexity of $x$ as the function
$$
K\mapsto\ \mathbf C_{x,\F,\eps}(K)= \inf_A \mathbf C_x(K|A),
$$
where $A$ ranges over all sets $A$ with $\underline d_\F(A)\ge1-\eps$.
\end{defn}

\begin{lem}\label{cinv}
Both $\mathbf C_x(K)$ and $\mathbf C_{x,\F,\eps}(K)$ are \inv, i.e.\ for each finite set $K\subset G$, every F\o lner \sq\ $\F$, every $\eps>0$ and each $g\in G$, we have 
\begin{gather*}
\mathbf C_{g(x)}(K)=\mathbf C_x(Kg) = \mathbf C_x(K), \text{ and}\\
\mathbf C_{g(x),\F,\eps}(Kg)=\mathbf C_{x,g\F,\eps}(K)=\mathbf C_{x,\F,\eps}(K).
\end{gather*}
\end{lem}
\begin{proof} Only the last equality requires explanation. It follows from the observation that the F\o lner \sq s $\F$ and $g\F = (gF_n)_{n\in\N}$ are equivalent, hence yield the same upper and lower densities of sets, so the respective infima over $A$ in Definition \ref{coml} (applied to $\F$ and $g\F$) have the same range.
\end{proof}

\begin{defn}\label{rof}
The \emph{rate of growth} of the complexity (respectively $(\F,\eps)$-complexity) of $x\in\Lambda^G$ is defined as 

\begin{align}
\mathsf{Rate}(x)&=\lim_{m\to\infty} \frac{\log \mathbf C_x(K_m)}{|K_m|}\,,\label{ratex}\\
\mathsf{Rate}_{\F,\eps}(x)&=\sup_{\mathfrak K}\limsup_{m\to\infty} \frac{\log \mathbf C_{x,\F,\eps}(K_m)}{|K_m|},\label{ratexfe}
\end{align}
where, in \eqref{ratex}, $(K_m)_{m\in\N}$ is some fixed centered F\o lner \sq\ in $G$ and, in \eqref{ratexfe}, the supremum is taken over all centered F\o lner \sq s $\mathfrak K=(K_m)_{m\in\N}$ in $G$.
\end{defn}

\begin{prop}\label{limex}
\begin{enumerate}
\item The limit in \eqref{ratex} exists, does not depend on the F\o lner \sq\ $(K_m)_{m\in\N}$, and can be replaced by the infimum over \emph{all} finite sets $K\subset G$. That is, 
$$
\mathsf{Rate}(x)= 
\inf\Bigl\{\tfrac{\log \mathbf C_x(K)}{|K|}:\ K\subset G,\ |K|<\infty\Bigr\}. 
$$

\item There exists a centered F\o lner \sq\ $\mathfrak K$ for which the supremum in \eqref{ratexfe} is attained.
\end{enumerate}
\end{prop}
\begin{proof}
A family $\K$ of finite (not necessarily different) subsets $K\subset G$ is called a $k$-cover of a finite set $F\subset G$ whenever each element of $F$ belongs to at least $k$ elements $K$ of $\K$.

We will now show that for every set $A\subset G$, the function $f(K)=\log \mathbf C_x(K|A)$ satisfies Shearer's inequality, i.e.\ whenever $\K$ is a $k$-cover of $F$, we have
\begin{equation}\label{sh}
f(F) \leq \frac1k\sum_{K\in\K} f(K).
\end{equation}

To see that this is true, denote by $\mathfrak X$ the collection of blocks 
$\{B\in\Lambda^F: (\exists g\in A)\ x|_{Fg}\approx B\}$, and for each $K\in\K$, let $\mathfrak X_K$ denote the family of restrictions
$\{B|_{F\cap K}: B\in\mathfrak X\}.$
With this notation we are precisely in the setup of \cite[Proposition 6.1]{DFR}, which says that
$$
\mathbf C_x(F|A) = |\mathfrak X|\le \prod_{K\in\K} |\mathfrak X_K|^{\frac1k}.
$$
Now, for each $K\in\K$, the collection of blocks $\mathfrak X_K$ is contained in the family $\{B\in\Lambda^{F\cap K}: (\exists g\in A)\ x|_{(F\cap K)g}\approx B\}$ whose cardinality is obviously not larger than $\mathbf C_x(K|A)=|\{B\in\Lambda^K:(\exists g\in A)\ x|_{Kg}\approx B\}|$. We have shown that 
$$
\mathbf C_x(F|A) \le \prod_{K\in\K} \mathbf C_x(K|A)^{\frac1k}.
$$
By taking logarithms on both sides we obtain \eqref{sh}. Since complexity is invariant, the proof of (1) is completed by a direct application of the infimum rule \cite[Proposition 3.3]{DFR}.

The statement (2) is nearly obvious. The maximizing F\o lner \sq\ $\mathfrak K$ can be obtained using a simple diagonal technique.
\end{proof}

We precede the formulation and proof of our second main theorem (Theorem~\ref{chara}, which is another characterization of determinism) by the presentation of some important auxiliary material: the notion of tile-entropy (introduced in \cite{DZ} under the name of ``tiled entropy'') and the associated Counting Lemma.

\subsection{Tile-entropy} In this subsection we introduce the notion of tile-entropy which is especially useful for obtaining certain combinatorial estimates. It will be instrumental in subsection \ref{secmain}, where we establish a characterization of determinism via combinatorial complexity.

As in Section \ref{til}, we will temporarily assume that $G$ is a (countable amenable) group. Throughout the rest of this subsection we let $\TTT=\bigvee_{k\in\N}\T_k$ be a F\o lner, uniquely congruent, zero entropy system of tilings of $G$ (which exists by Theorem~\ref{dhz}). We now introduce more notation. 
%Recall that given $\boldsymbol\CT=(\CT_k)_{k\in\N}\in\mathbf T$, the set of centers of $\CT_k$ of the tiles with shape $S\in\CS_k$ is denoted by $C_S(\CT_k)$. The set of all centers of $\CT_k$ is $C(\CT_k)=\bigcup_{S\in\CS_k}C_S(\CT_k)$. 
For each $k\in\N$ let $\S_k$ be the family of shapes of $\T_k$ (see Definition \ref{tili1}).
For $S\in\CS_k$ and $s\in S$, define
$$
[S,s] =\{\boldsymbol\CT=(\CT_k)_{k\in\N}\in\TTT: s^{-1}\in C_S(\CT_k)\}.
$$
If $\boldsymbol\CT\in[S,s]$, then $Ss^{-1}$ is the tile of $\CT_k$ which contains the unit $e$, i.e.\ \emph{the central tile of $\CT_k$}. The set $[S,e]$ will be abbreviated as $[S]$. Observe that  $\boldsymbol\CT\in[S]$ if and only if $\CT_k(e)=``S"$, so the notation is consistent with that of one-symbol cylinders $[``S"]$ over the alphabet ${\rm V}_k=\{``S": S\in\CS_k\}\cup\{0\}$. The family $\D_{\CS_k}=\{[S,s]:S\in\CS_k,\ s\in S\}$ is a partition of $\TTT$ which classifies its elements $\boldsymbol\CT$ according to the shape and position of the central tile of $\CT_k$. Also note that $\boldsymbol\CT\in[S,s]$ if and only if $s^{-1}(\boldsymbol\CT)\in[S]$, i.e.\ $[S,s]=s([S])$. So, if $\nu$ is a shift-\im\ on $\TTT$, then $\nu([S,s])=\nu([S])$ for all $s\in S$.

\smallskip
Let $X\subset\Lambda^G$ be a subshift and let $\bar X=X\times\TTT$. The elements of $\bar X$ have the form $\bar x=(x,\boldsymbol\CT)=(x,\CT_1,\CT_2,\dots)$.

We will use the following notational convention. For $\bar x\in\bar X$ and $S\in\CS_k$, the expression $\bar x(g)=``S"$ means that if $\bar x=(x,\CT_1,\CT_2,\dots)$, then $\CT_k(g)=``S"$.
With this convention, $[S]$ and $[S,s]$ (where $S\in\CS_k,\ s\in S$), and hence also the partition $\D_{\CS_k}$, may be ``lifted'' to $\bar X$ in the following way:
$$
[S]=\{\bar x\in \bar X: \bar x(e)=``S"\}, \ \ [S,s]=s([S]) =\{\bar x\in \bar X: \bar x(s^{-1})=``S"\}.
$$

\begin{defn}
Let $\P$ be a finite measurable partition of $\bar X$ and let $\mu$ be  a probability measure on $\bar X$. The \emph{$k$th tile-entropy of $\P$} with respect to $\mu$, $H_{\T_k}(\mu,\P)$, is defined by
$$
H_{\T_k}(\mu,\P)=\sum_{S\in\CS_k}\mu([S])H(\mu_{[S]},\P^S), \ \ 
$$
where $\mu_{[S]}$ is the normalized conditional measure $\mu$ on $[S]$.
\end{defn}

The following theorem is proved in \cite{DZ}. It is the monotonicity that makes tile-entropy especially useful.

\begin{thm}\label{tte} (\cite[Theorem 4.27]{DZ})
Let $\mu\in\mgxb$. The \sq\ of tile-entropies $(H_{\T_k}(\mu,\P))_{k\in\N}$ is nonincreasing
and 
$$
\lim_{k\to\infty}H_{\T_k}(\mu,\P)= h(\mu,\P).
$$
\end{thm}

\begin{cor}\label{unif}
Let $\M\subset\mgxb$ be a compact subset of \im s on $\bar X$, all having entropy zero. The \sq\ of functions $\mu\mapsto H_{\T_k}(\mu,\Lambda)$ converges to zero uniformly on $\M$.
\end{cor}

\begin{proof}In the formulation of the corollary, $\Lambda$ stands for the lift from $X$ to $\bar X$ of the zero-coordinate partition $\Lambda$, which is clopen. Thus the functions $\mu\mapsto H_{\T_k}(\mu,\Lambda)$ are continuous on $\M$. Now the assertions follows via Dini's theorem from the monotonicity guaranteed by Theorem~\ref{tte}.
\end{proof}

Recall that $\hat B$ stands for the empirical measure associated to a block $B$. The following definition extends the $k$th tile-entropy of the zero-coordinate partition $\Lambda$ to certain empirical measures.

\begin{defn}\label{69}
Fix two indices $k,k'\in\N$, $k<k'$. Choose a shape $S'\in\CS_{k'}$ and let $B\in\Lambda^{S'}$. We define
\begin{equation}\label{ddd}
H_{\T_k}(B,\Lambda)=\sum_{S\in\CS_k}\hat B([S])H(\hat B_{[S]},\Lambda^S).
\end{equation}
\end{defn}
The terms appearing in the above definition require some clarification. Recall that $S'$ is partitioned by the tiles of $\T_k$: 
$$
S'=\bigcup_{S\in\CS_k}\ \bigcup_{c\in C_S(S')}Sc.
$$ 
We will refer to the tiles $Sc$ in this partition as \emph{$k$-subtiles} of $S'$. For each $S\in\CS_k$, $\hat B([S])$ equals $\frac{|C_S(S')|}{|S'|}$ (i.e.\ it is the frequency of the symbol $``S"$ in the symbolic representation of $\T_k$ within the shape $S'$; note that this value does not depend on the block $B$, only on the shape $S'$). Next, given $S\in\CS_k$ and a block $R\in\Lambda^S$, we have
$$
\hat B_{[S]}(R) = \frac{|\{c\in C_S(S'): B|_{Sc}=R\}|}{|C_S(S')|},
$$
i.e.\ this is the frequency of the occurrences of $R$ among all blocks occurring within the block $B$ over the $k$-subtiles which have the shape $S$. Note that $\hat B_{[S]}$ is a probability measure on the finite collection of blocks $\Lambda^S$, so the right hand side of \eqref{ddd} is nothing but a linear combination of Shannon entropies $H(\hat B_{[S]},\Lambda^S)$.

Observe that the $k$th tile-entropy of the zero-coordinate partition $\Lambda$, viewed as the function $\nu\mapsto H_{\T_k}(\nu,\Lambda)$, $\nu\in \M_G(\Lambda^G)\cup\B^*$, is continuous with respect to the metric $\dist$ (defined in subsection \ref{dbbm}).  

\begin{lem}[The counting lemma]\label{counting}
Fix some $k\in\N$. For $S'\in\bigcup_{k'>k}\S_{k'}$ and $c>0$ define
$$
\mathbf C[k,S',c] := |\{B\in\Lambda^{S'}: H_{\T_k}(B,\Lambda)<c\}|.
$$
Then, 
$$
\limsup_{|S'|\to\infty}\frac{\log(\mathbf C[k,S',c])}{|S'|}\le c.
$$
\end{lem}

\begin{proof}  We will use \cite[Lemma 2.8.7]{D}, case of $n=1$, which concerns double words $\mathbb B=(\mathbb B_1,\mathbb B_2)\in \Lambda^{\{1,2,\dots,m\}}$, where $\Lambda=\Lambda_1\times\Lambda_2$ is a product alphabet. Following the notation of \cite{D}, we let $H_1(\mathbb B)$ stand for the entropy of the zero-coordinate partition $\Lambda$ with respect to the empirical measure associated with $\mathbb B$. $H_1(\mathbb B_1)$ is defined analogously for $\mathbb B_1$ and $\Lambda_1$ and the conditional entropy $H_1(\mathbb B|\mathbb B_1)$ is simply the difference $H_1(\mathbb B)-H_1(\mathbb B_1)$. Also, for $c>0$, we define
$$
\mathbf C[1,m,c]=|\{\mathbb B\in \Lambda^{\{1,2,\dots,m\}}: H_1(\mathbb B|\mathbb B_1)\le c\}|.
$$
Then, \cite[Lemma 2.8.7]{D} states that for any $c>0$, we have
$$
\limsup_{m\to\infty}\frac{\log(\mathbf C[1,m,c])}m\le c.
$$

Our Lemma \ref{counting} follows almost directly from the cited result, subject to an appropriate translation of terminology. 

Recall that for $S'\in\CS_{k'}$, $C_S(S')$ denotes the set of centers of the $k$-subtiles of $S'$ of shape $S$. We also denote $C_k(S')=\bigcup_{S\in\S_k}C_S(S')$ (the set of centers of all the $k$-subtiles of $S'$). Now, we enumerate $C_k(S')$ as $\{c_1,c_2,\dots, c_{m_{S'}}\}$, where $m_{S'}=|C_k(S')|$. Next, we create a word $\mathbb B_1\in({\rm V}_k\setminus\{``0"\})^{\{1,2,\dots,m_{S'}\}}$, by the rule
$$
\mathbb B_1(i) = ``S" \iff c_i\in C_S(S'), \ \ \ i=1,2,\dots,m_{S'}.
$$
Note that $\mathbb B_1$ is simply the ordered list of shapes of the $k$-subtiles of $S'$. Now, given a block $B\in\Lambda^{S'}$, we create another word, $\mathbb B_2\in\{``R": R\in\bigcup_{S\in\CS_k}\Lambda^S\}^{\{1,2,\dots,m_{S'}\}}$ (the symbols $``R"$ are associated bijectively to all possible blocks $R$ over all shapes $S\in\CS_k$). The rule is now
$$
\mathbb B_2(i) = ``R" \text{ for some }  R\in\Lambda^S, S\in\CS_k \iff c_i\in C_S(S')\text{ and }B|_{Sc_i}\approx R.
$$
Note that $\mathbb B_2$ is the ordered list of blocks obtained by restricting $B$ to the $k$-subtiles of $S'$. Finally, we create a double word $\mathbb B=(\mathbb B_1,\mathbb B_2)$. Clearly, $\mathbb B$ depends on $B$ and the map $B\mapsto\mathbb B$ is injective. Now, by a straightforward comparison of the empirical measures, we see that
$$
\hat{\mathbb B}_1([``S"])=\frac{|C_S(S')|}{|C_k(S')|}= \hat B([S])\cdot \frac{|S'|}{m_{S'}}.
$$
The conditional measures do not need normalization: for any $S\in\CS_k$ and $R\in\Lambda^S$, we have 
$$
\hat{\mathbb B}\bigl([``S",``R"]\,\bigr|\,\bigl.[``S",*\,]\bigr)=\hat B_{[S]}(R). 
$$
So, we obtain:
\begin{equation}\label{rra}
H_{\T_k}(B,\Lambda) = H_1(\mathbb B|\mathbb B_1)\cdot \frac{m_{S'}}{|S'|}.
\end{equation}  
Since the multiplier $\frac{m_{S'}}{|S'|}$ depends on $S'$, we cannot use \cite[Lemma 2.8.7]{D} (see the discussion at the beginning of the proof) just yet.

Note that the fraction $\frac{m_{S'}}{|S'|}=\frac{|C_k(S')|}{|S'|}$ represents the frequency of the centers of the $k$-subtiles within $S'$ and hence it satisfies
% which can be interpreted as $\hat D_{S'}(\bigcup_{S\in\CS_k}[``S"])$, where $D_{S'}\in {\rm V}_k^{S'}$ is defined by the obvious rule (cf. \eqref{tsy})
%$$
%D_{S'}(g)= \begin{cases}
%``S" \text{ for some }S\in\S_k, &\text{if } g\in C_S(S') \text{ \tiny($g$ is the center of a $k$-subtile of shape $S$)},\\ 
%``0", & \text{otherwise} \text{ \tiny($g$ is not the center of any $k$-subtile)}.
%\end{cases}
%$$
$$
1\ge\frac{m_{S'}}{|S'|}\ge \frac1{\max\{|S|:S\in\S_k\}}=:t_0>0.
$$ 
Thus, any accumulation point of the set of ratios $\{\frac{m_{S'}}{|S'|}:S'\in\bigcup_{k'>k}\S_{k'}\}$ is a positive number.  Clearly, it suffices to prove the assertion of the lemma for $S'$ varying along a \sq\ $(S'_n)_{n\in\N}$, selected from $\bigcup_{k'>k}\S_{k'}$, for which the fractions $\frac{m_{S'_n}}{|S'_n|}$ converge. Choose such a \sq\ and let $t\ge t_0$ be the corresponding limit. Given $\eps>0$, for any sufficiently large $n$ and $S'=S'_n$ we have  
\begin{equation}\label{obie}
t-\eps\le\frac{m_{S'}}{|S'|}\le t+\eps.
\end{equation} 
Then, for any block $B\in\Lambda^{S'}$ and the corresponding double block $\mathbb B$, by \eqref{rra} and \eqref{obie}, we have
$$
H_{\T_k}(B,\Lambda) \ge H_1(\mathbb B|\mathbb B_1)\cdot(t-\eps).
$$
Thus, for any $c>0$, we have 
$$
\mathbf C[k,S',c]\le|\{\mathbb B\in\Lambda^{\{1,2,\dots,m_{S'}\}}:H_1(\mathbb B|\mathbb B_1)\le \tfrac c{t-\eps}\}|=\mathbf C[1,m_{S'},\tfrac c{t-\eps}].
$$
Now we are in a position to use \cite[Lemma 2.8.7]{D}, which implies that if $m_{S'}$ is sufficiently large (equivalently, if $S'=S'_n$ for a sufficiently large $n$), then 
$$
\frac{\log(\mathbf C[k,S',c])}{m_{S'}}\le \frac c{t-2\eps}\le \frac c{t-2\eps}\frac{(t+\eps)|S'|}{m_{S'}}
$$
(note that, by \eqref{obie}, we have $\frac{(t+\eps)|S'|}{m_{S'}}\ge1$). Multiplying both sides by $\frac{m_{S'}}{|S'|}$ we obtain
$$
\frac{\log(\mathbf C[k,S',c])}{|S'|}\le c\cdot\frac{t+\eps}{t-2\eps}.
$$
Since $\eps$ is arbitrarily small, the right hand side is arbitrarily close to $c$ and the proof is finished.

\end{proof}

\subsection{The second main result and its proof}\label{secmain}

We are ready to formulate and prove our complexity-based criteria for determinism.

\begin{thm}\label{chara}
An element $x\in\Lambda^G$ is strongly deterministic if and only if its rate of growth of complexity (see Definition \ref{rof}) equals zero:
$$
\mathsf{Rate}(x)=0. 
$$
An element $x\in\Lambda^G$ is $\F$-deterministic if and only if  for every $\eps>0$ the rate of growth of $(\F,\eps)$-complexity of $x$ (see Definition \ref{rof}) equals zero:
$$
\mathsf{Rate}_{\F,\eps}(x)=0. 
$$
\end{thm}

\begin{proof}
The case of strong determinism requires no proof. Indeed, the equivalence between strong determinism and zero rate of growth of complexity follows from the well-known fact that topological entropy of the subshift generated by $x$ (i.e.\ its orbit closure) equals the rate of growth of the complexity of $x$. 

We pass to proving that $\F$-determinism implies subexponential $(\F,\eps)$-complexity. We temporarily assume that $G$ is a group. Pick an $\F$-deterministic element $x\in\Lambda^G$ and an $\eps>0$. We let $\mathfrak K=(K_m)_{m\in\N}$ be a centered F\o lner \sq, on which the supremum in the definition of the rate of growth of the $(\F,\eps)$-complexity of $x$ is attained (see Proposition~\ref{limex}~(2)). Our goal is to show that for any $\delta>0$ and $m$ sufficiently large, 
$$
\frac{\log \mathbf C_{x,\F,\eps}(K_m)}{|K_m|}<\delta,
$$
which means that there exists a set $A_m$ such that $\underline d_\F(A_m)\ge1-\eps$ and
$$
\frac{\log \mathbf C_x(K_m|A_m)}{|K_m|}<\delta.
$$

We use a deterministic, F\o lner system of tilings $\TTT$ with \tl\ entropy zero to create the system $\bar X=\Lambda^G\times\TTT$. Let $\bar x=(x, \boldsymbol\CT)\in\bar X$, where $x$ is our $\F$-deterministic symbolic element, while $\boldsymbol\CT=(\CT_1,\CT_2,\dots)\in\TTT$ is arbitrary, but fixed throughout the rest of the proof. 
%Since $\TTT$ has \tl\ entropy zero, $\bar x$ is also $\F$-deterministic, i.e.\ if we let $\M$ be the set of all measures for which $\bar x$ is $\F$-quasi-generic, then all $\mu\in\M$ have entropy zero.
We will now successively select three parameters, $k,k'$ (indexes of the system of tilings) and $m$ (the index of $K_m$). 

Pick a positive number 
\begin{equation}\label{gam}
\gamma<\min\bigl\{\eps,\tfrac\delta{7\log|\Lambda|}\bigr\}.
\end{equation} 
We now invoke the notion of tile-entropy. By Corollary \ref{unif}, there exists $k\in\N$ such that $H_{\T_k}(\mu,\Lambda)<\gamma^3$ for all $\mu\in\M$. By continuity of the function $\mu\mapsto H_{\T_k}(\mu,\Lambda)$ on measures (including empirical measures), this inequality holds for all empirical measures associated to the blocks $\bar x|_{F_n}$ for large enough $n$. We note this fact as follows:
\begin{equation}\label{totu}
H_{\T_k}(\bar x|_{F_n},\Lambda)<\gamma^3.
\end{equation}
By the Counting Lemma \ref{counting}, there exists $k'>k$ so large that all shapes $S'\in\CS_{k'}$ satisfy 
\begin{equation}\label{cl}
\frac{\log(\mathbf C[k,S',\gamma])}{|S'|}\le 2\gamma.
\end{equation}
Because $\T_{k'}$ has \tl\ entropy zero and $\CT_{k'}\in\T_{k'}$, we can now choose $m$ large enough so that
\begin{equation}\label{dv}
|\{D\in{\rm V}_{k'}^{K_m}: (\exists g\in G)\ \CT_{k'}|_{K_mg}\approx D\}|<2^{\gamma|K_m|}.
\end{equation}
(in words: the number of blocks $D$ over the alphabet ${\rm V}_{k'}=\{``S\,": S\in\CS_{k'}\}\cup\{``0"\}$, and having the domain $K_m$, which appear in $\CT_{k'}$ is less than $2^{\gamma|K_m|}$). 
\smallskip

Observe that the \sq\ of $\CT_{k'}$-saturations $(F_n^{\CT_{k'}})_{n\in\N}$ is a F\o lner \sq\ equivalent to $\F$. To simplify the notation, from now on $F_n$ will stand for $F_n^{\CT_{k'}}$ and $\F$ will denote $(F_n^{\CT_{k'}})_{n\in\N}$ (note that \eqref{totu} still applies for $n$ large enough). 

We will call a tile $T'$ of $\CT_{k'}$ ``good'' if $H_{\T_k}(\bar x|_{T'},\Lambda)\le\gamma$. By \eqref{cl}, we know that for each shape $S'\in\CS_{k'}$,
\begin{equation}\label{gb}
|\{B\in\Lambda^{S'}: (\exists c\in G)\ S'c = T'\in\CT_{k'}, T' \text{ is ``good'' }, x|_{T'}\approx B\}|\le 2^{2\gamma|S'|}
\end{equation}
(in words: the number of different blocks appearing in $x$ over the ``good'' tiles with shape $S'$ does not exceed $2^{2\gamma|S'|}$).

Using the elementary fact that the $k$th tile-entropy is concave on empirical measures associated to disjoint $\CT_k$-saturated sets, we get 
$$
\sum_{T'\subset F_n}\frac{|T'|}{|F_n|}H_{\T_k}(\bar x|_{T'},\Lambda)\le H_{\T_k}(\bar x|_{F_n},\Lambda)<\gamma^3,
$$
which immediately implies that the union $U_{\mathsf{good}}$ of the ``good'' tiles occupies a fraction larger than $(1-\gamma^2)$ of $F_n$. In this manner, we have shown that 
\begin{equation}\label{ess}
\underline{d}_{(F_n)}(U_{\mathsf{good}})\ge 1-\gamma^2. 
\end{equation}

For a large F\o lner set $F_n$, let $F'_n$ denote its $K_m^{-1}$-core. Let
$$
\alpha_n = \sum_{g\in F_n} |U_{\mathsf{good}}\cap K_mg\cap F'_n|.
$$ 
Since each $g'\in F'_n$ belongs to $K_mg$ for exactly $|K_m|$ elements $g\in F_n$, we have
$$
\alpha_n = |K_m|\cdot |U_{\mathsf{good}}\cap F_n'|.
$$
By \eqref{ess} and since $\frac{|F_n'|}{|F_n|}\to 1$, we easily get that, for $n$ large enough,
$$
\alpha_n>|K_m||F_n|(1-2\gamma^2).
$$ 
Writing on the left hand side the formula defining $\alpha_n$, and dividing both sides by $|K_m||F_n|$, we obtain
$$
\frac1{|F_n|}\sum_{g\in F_n} \frac1{|K_m|}|U_{\mathsf{good}}\cap K_mg\cap F_n'|\ge 1-2\gamma^2,
$$
which implies that at least as many as $(1-\gamma)|F_n|$ elements $g\in F_n$ satisfy
$$
\frac1{|K_m|}|U_{\mathsf{good}}\cap K_mg|\ge \frac1{|K_m|}|U_{\mathsf{good}}\cap K_mg\cap F_n'|\ge 1-2\gamma.
$$
We have shown that the set 
$$
A_m=\{g\in G: \tfrac1{|K_m|}|U_{\mathsf{good}}\cap K_mg|\ge 1-2\gamma\}
$$
has lower $\F$-density at least $1-\gamma>1-\eps$ (as it is required for $A_m$). 
\smallskip

It remains to estimate the cardinality 
$$
\mathbf C_x(K_m|A_m)=|\{B\in\Lambda^{K_m}: (\exists_g\in A_m)\ x|_{K_mg}\approx B\}|. 
$$

Fix a block $D\in{\rm V}_{k'}^{K_m}$. For $g\in G$, the fact that $\CT_{k'}|_{K_mg}\approx D$,
determines (up to shifting, of course) the partition by the tiles of $\CT_{k'}$ of the set 
$$
(K_mg)'=\bigcup_{S'\in\CS_{k'}}\ \bigcup_{c\in K_mg\cap C_{S'}(\CT_{k'})}S'c
$$ 
(in words: $(K_mg)'$ is the union of all the tiles of $\CT_{k'}$ whose centers lie in $K_mg$). Note that by choosing $m$ large enough we can assume that, regardless of $g$, 
\begin{equation}\label{fv}
\frac{|K_mg\setminus(K_mg)'|}{|K_m|}<\gamma.
\end{equation}

Throughout this paragraph we restrict our attention to only such elements $g$ for which $\CT_{k'}|_{K_mg}\approx D$. We first estimate the number of blocks $x|_{U_{\mathsf{good}}\cap (K_mg)'}$. For every $S'\in\CS_{k'}$, the number of possible blocks $x|_{S'c}$, where $S'c$ is a ``good'' tile, has been estimated in \eqref{gb} by $2^{2\gamma|S'|}$. So, the number of blocks of the form $x|_{U_{\mathsf{good}}\cap (K_mg)'}$ does not exceed  
\begin{equation}\label{eee}
2^{2\gamma|U_{\mathsf{good}}\cap (K_mg)'|}\le 2^{2\gamma|K_m|(1+\gamma)}. 
\end{equation}
Cardinality of $K_mg\setminus (U_{\mathsf{good}}\cap (K_mg)')$ can be estimated using \eqref{fv} and \eqref{ess} (for large $m$), by $2\gamma|K_m|$. Thus, the number of possible blocks $x|_{K_mg\setminus (U_{\mathsf{good}}\cap(K_mg)')}$ does not exceed
$|\Lambda|^{2\gamma|K_m|}$. Combining this result with \eqref{eee}, we conclude that the number of blocks of the form $x|_{K_mg}$ is bounded by 
$$
2^{2\gamma|K_m|(1+\gamma)}\cdot|\Lambda|^{2\gamma|K_m|}\le |\Lambda|^{6\gamma|K_m|}.
$$

Recall, that this estimate was done after restricting $g$ to such elements that $\CT_k|_{K_mg}\approx D$. So, in order to obtain the final estimate of $\mathbf C_x(K_m|A_m)$, we must multiply the above number $|\Lambda|^{6\gamma|K_m|}$ by the number of possible blocks $D$, which, as shown in  \eqref{dv}, is less than $2^{\gamma|K_m|}$. For simplicity, we will multiply by $|\Lambda|^{\gamma|K_m|}$ and conclude
$$
\mathbf C_x(K_m|A_m)\le |\Lambda|^{7\gamma|K_m|}\le 2^{\delta|K_m|},
$$
by the choice of $\gamma$ \eqref{gam}. This ends the proof of the implication, as we have found a set $A_m$ with $\underline d_\F(A_m)\ge1-\eps$, such that $\frac{\log \mathbf C_x(K_m|A_m)}{|K_m|}\le\delta$.

\medskip
We pass to the proof of the converse implication. Let $x\in\Lambda^G$ be such that $\mathsf{Rate}_{\F,\eps}(x)=0$ for any $\eps>0$. That is, for any $\eps,\delta>0$ there exists $m_0\in\N$ such that for every $m\ge m_0$ there exists a set $A_m$ of lower $\F$-density at least $1-\eps$ such that 
\begin{equation}\label{eqq}
\frac{\log \mathbf C_x(K_m|A_m)}{|K_m|}<\delta.
\end{equation}
(Here again we let $\mathfrak K=(K_m)_{m\in\N}$ be a centered F\o lner \sq, on which the supremum in the definition of the $(\F,\eps)$-complexity is attained.) We need to show that any \im\ $\mu$, for which $x$ is $\F$-quasi-generic, has entropy zero. 

We start by fixing an $\eps>0$ and $\delta=\eps$, and choosing $m$ large enough, so that \eqref{eqq} holds for $K_m$ with a suitable set $A_m$. We will show that 
\begin{equation}\label{topok}
\frac1{|K_m|}H(\mu,\Lambda^{K_m})<O(\eps), 
\end{equation}
where $O(\eps)$ depends only on $\eps$ and tends to $0$ as $\eps\to0$. Then, by passing to the limit as $m\to\infty$, we will conclude that $h(X,\Sigma,\mu,G)=h(\mu,\Lambda)<O(\epsilon)$ for every $\eps>0$, i.e. that $h(X,\Sigma,\mu,G)=0$, and the proof will be finished.

Since from now on the index $m$ is fixed, we will skip it in the notation and write just $K$ and $A$. Denote by $\B$ the family of blocks $\{B\in\Lambda^K: (\exists g\in A)\ x|_{Kg}\approx B\}$. By \eqref{eqq}, $|\B|\le 2^{\eps|K|}$, while the fact that $\underline d_\F(A)\ge 1-\eps$ implies that $\mu(Q)<\eps$, where
$$
Q=\bigcup_{B\in\Lambda^K\setminus\B}[B].
$$
Denoting by $\Q$ the partition $\Lambda^G=Q\cup Q^c$, we can write
\begin{multline*}
H(\mu,\Lambda^K)\le \\
H(\mu,\Lambda^K|\Q)+H(\mu,\Q)=
\mu(Q^c)H(\mu_{Q^c},\Lambda^K)+\mu(Q)H(\mu_Q,\Lambda^K)+ H(\mu,\Q)\le\\ 1\cdot\eps|K|+\eps|K|\log|\Lambda|+H(\eps),
\end{multline*}
where $H(\eps)=-\eps\log\eps-(1-\eps)\log(1-\eps)$. Note that $H(\eps)\to0$ as $\eps\to0$.

Dividing by $|K|$, we obtain
$$
\frac1{|K|}H(\mu,\Lambda^K)<\eps+\eps\log(|\Lambda|) + H(\eps) = O(\eps).
$$
This ends the proof of the theorem for groups.
\smallskip

We will now explain how to extend the theorem to the case of a general countable, cancellative, amenable semigroup $G$. As in the proof of Theorem \ref{generic}, we embed $G$ in a group $\bar G$ so that our F\o lner \sq\ $\F$ in $G$ is also a F\o lner \sq\ in $\bar G$. We extend $x$ to $\bar x\in\Lambda^{\bar G}$ by assigning the values $\bar x(g)$ for $g\in\bar G\setminus G$ completely arbitrarily. Notice that by doing so we may increase the complexity of $\bar x$ (in comparison to that of $x$) but not the $(\F,\eps)$-complexity (the infimum in Definition \ref{coml} will not change when restricted to subsets of the subsemigroup $G\subset \bar G$). So, 
\begin{equation}\label{rir}
\mathsf{Rate}_{\F,\eps}(\bar x)=\mathsf{Rate}_{\F,\eps}(x).
\end{equation} 
On the other hand, any $G$-\im\ $\mu$ on $\Lambda^G$ extends uniquely to an \im\ $\bar\mu$ on $\Lambda^{\bar G}$. Since $G$ and $\bar G$ have a common F\o lner \sq\ (and since for the calculation of entropy the F\o lner \sq\ may be chosen arbitrarily), $h(\mu,G)=h(\bar\mu,\bar G)$. Moreover, a point $\bar x\in\Lambda^{\bar G}$ is $\F$-quasi-generic for $\bar\mu$ if and only $x$ is $\F$-quasi-generic for $\mu$. This implies that $x$ is $\F$-deterministic if and only if so is $\bar x$. In view of \eqref{rir},
the characterization of $\F$-determinism in terms of $(\F,\eps)$-complexity extends from groups to semigroups.
\end{proof}

\section{Examples of deterministic sets}\label{nine}

Recall (see Remark \ref{dd}) that sets with Banach density 0 or 1 are strongly deterministic for trivial reasons. Similarly, sets with $\F$-density 0 or 1 are $\F$-deterministic for trivial reasons: the indicator function of such a set is $\F$-quasi-generic for the point mass concentrated at a fixpoint. In view of Theorem \ref{main}, we are mostly interested in examples of $\F$-deterministic sets with positive lower $\F$-density (in particular, strongly deterministic sets with positive lower Banach density), because they automatically become examples of $\F$-normality preserving sets. In this section we describe several classes of examples of $\F$-deterministic sets with lower and upper $\F$-densities strictly between 0 and 1, and of strongly deterministic sets with lower and upper Banach densities strictly between 0 and 1. Before we proceed, let us make some general observations concerning the existence of deterministic sets in countable cancellative amenable semigroups $G$.

\begin{prop}\label{trivia}
\begin{enumerate}
\item In any countable amenable group $G$ there exists a strongly deterministic subset with both lower and upper Banach densities strictly between 0 and 1.
\item For any F\o lner \sq\ $\F$ in $G$, there exists an $\F$-deterministic subset of $G$ which is not strongly deterministic.
\item For any F\o lner \sq\ $\F$ in $G$, the collection of all $\F$-deterministic subsets of $G$, viewed as a subset of $\{0,1\}^G$, is of first category.
\end{enumerate}
\end{prop}

\begin{proof} (1) As shown in \cite{DHZ}, there exists a free action \xg\ of entropy zero  on a zero-dimensional space $X$ (free means that $g(x)=x$, for some $g\in G$ and $x\in X$, implies $g=e$). Let $(X',G)$ be a minimal subsystem of \xg. Since the action of $G$ on $X$ is free, $X'$ is not a single point. Clearly, $X'$ is zero-dimensional, $(X',G)$ is free and has \tl\ entropy zero. The desired set can be obtained as the set of visiting times of the orbit of some $x\in X'$ to some clopen proper (neither empty nor equal to $X'$) subset $U$ of $X'$. By minimality, the visiting times to both $U$ and its complement form syndetic sets, hence have positive lower densities. 

(2) Note that there exists another F\o lner \sq\ $\F'=(F_n')_{n\in\N}$, such that the union $U=\bigcup_{n\in\N} F'_n$ has zero $\F$-density and the cardinalities $|F'_n|$ strictly increase. Now, any point $z\in\{0,1\}^G$ of the form
$$
z_g = 
\begin{cases}
y_g & g\in U,\\
x_g & \text{otherwise},
\end{cases}
$$
where $y$ is $\F'$-normal and $x$ is strongly deterministic, is $\F$-deterministic but not strongly deterministic (as it is $\F'$-normal).

(3) By passing to a sub\sq\ of $\F$ we can assume that the cardinalities $|F_n|$ strictly increase. Then, by \cite[Theorem 4.2]{BDM}, there exist many $\F$-normal elements in $\{0,1\}^G$. Although the collection $\mathcal N(\F)$ of all such elements is of first category (see \cite[Proposition 4.7]{BDM}), the collection $\mathcal N_{\mathsf q}(\F)$ of elements $x\in\{0,1\}^G$, which are $\F$-quasi-normal (i.e. $\F$-quasi-generic for the Bernoulli measure), is residual. Indeed, it is dense (the smaller set $\mathcal N(\F)$ is not only dense but has full Bernoulli measure with full \tl\ support), and we have
$$
\mathcal N_{\mathsf q}(\F)=\bigcap_{k\in\N}\ \bigcap_{i\le k}\ \bigcap_{B\in\{0,1\}^{K_i}}\ \bigcap_{n_0\in\N}\ \bigcup_{n\ge n_0}
W(F_n,B,\tfrac1k),
$$
where $(K_i)_{i\in\N}$ is a \sq\ consisting of all finite subsets of $G$, and 
$W(F_n,B,\frac1k)$ with $B\in\{0,1\}^{K_i}$ is the union of all cylinders $[C]$ with $C\in \{0,1\}^{F_n}$ such that
$$
2^{-|K_i|}-\tfrac1k<\Fr_C(B)<2^{-|K_i|}+\tfrac1k.
$$
Since each set $W(F_n,B,\frac1k)$ is clopen, $\mathcal N_{\mathsf q}(\F)$ is a $G_\delta$-set. Clearly, the residual set $\mathcal N_{\mathsf q}$ is disjoint from the set of $\F$-deterministic elements.
\end{proof}

\subsection{Strongly deterministic sets associated to tilings}
It follows from Theorem \ref{dhz} that every countable amenable group $G$ admits a dynamical tiling $\T$ of entropy zero. Then, for any $\CT\in\T$ and any shape $S\in\CS(\CT)$, the set $C_S(\CT)$ (of centers of the tiles $T\in\CT$ of shape $S$) is strongly deterministic, as it coincides with the set of visiting times of the orbit of $\CT$ in a clopen set (see Theorem \ref{vit}). 

We can apply this fact to a more specific case. Suppose the group $G$ contains a finite subgroup $S$. The equivalence relation $g\sim h\iff gh^{-1}\in S$ partitions $G$ into
countably many equivalence classes. If $C_S\subset G$ is a set containing exactly one element from each class, then the family $\{S c:c\in C_S\}$ is a monotiling of $G$ with one shape $S$.  A priori the centers $c\in C_S$ can be selected as arbitrary elements of the tiles, and there seems to be no specific preference to any particular selection (except for the central tile $S$ where it is natural to choose $e$). However, the general theory of tilings tells us that there exists an ``intelligent'' selection, for which the resulting set of centers $C_S$ is strongly deterministic. See \cite[Section 6]{DHZ}, which includes more details and uses this idea to construct a zero-entropy free $G$-action for any countable amenable group $G$. Clearly, the set $C_S$ has Banach density $1/|S|$.

We illustrate this idea on a concrete example. %By a \emph{finite permutation} of $\N$ we understand a bijection $g:\N\to\N$ with $g(n)=n$ a cofinite set of $n\in\N$.
In the following proposition we consider the group $G$ of all finite permutations of $\N$, i.e. bijections $g:\N\to\N$ which move only finitely many elements, with the multiplication defined as the composition in the reversed order: $gh = h\circ g$. (This convention is frequently used in the permutation group literature.) 
\begin{prop}
Let $G$ be the group of all finite permutations of $\N$.
For every $k\ge2$, the set 
$$
C^{\mathsf{incr}}_k = \{g\in G: g(1)<g(2)<\cdots<g(k)\}
$$ 
is strongly deterministic in $G$ (and has positive Banach density $\frac1{k!}$).
\end{prop}
\begin{proof}For each finite permutation $g\in G$, there exists $n_g\in\N$ such that $g(n)=n$ for all $n>n_g$. For each $k\ge 2$, the permutations $g$ with $n_g\le k$ form a finite subgroup $S_k$. For any permutation $c\in G$, there exists a unique $s\in S_k$ such that $sc\in C^{\mathsf{incr}}_k$. Indeed, $s$ can be identified as follows: if we write the numbers $c(1),c(2),\dots,c(k)$ in the increasing order, then the arguments $1,2,\dots,k$  appear in the order $s(1),s(2),\dots,s(k)$. We have shown that $C^{\mathsf{incr}}_k$ contains exactly one element from each of the equivalence classes defined by $S_k$, implying that the family $\CT_k=\{S_k c:c\in C^{\mathsf{incr}}_k\}$  is a monotiling with shape $S_k$. In particular, this implies that the Banach density of $C^{\mathsf{incr}}_k$ equals $\frac1{|S_k|}=\frac1{k!}$. We will now show that $C^{\mathsf{incr}}_k$ is a strongly deterministic set. 

Let $\boldsymbol\CT=(\CT_k)_{k\in\N}$ and let $\TTT$ be the orbit closure of $\boldsymbol\CT$ in $\prod_{k\in\N}{\rm V}_k^G$. Observe that $(S_k)_{k\ge 2}$ is a F\o lner \sq\ in $G$, hence $\TTT$ is a F\o lner system of monotilings. Note that, for each $k\in\N$, we have
$$
S_{k+1}=\bigcup_{i=1}^{k+1} S_kc_i,
$$
where 
\begin{equation}\label{incr}
c_i=(1,2,\dots,i-1,i+1,i+2,\dots,k+1,i,k+2,k+3,\dots)
\end{equation}
(note that the centers $c_i$ belong to $C^{\mathsf{incr}}_k$). We have shown that the system of monotilings $\TTT$ is congruent. 

We claim that the system $\TTT$ is uniquely congruent. We need to show that whenever $T=S_{k+1}c$ is a tile of $\CT_{k+1}$ with $c\in C^{\mathsf{incr}}_{k+1}$ and we represent it as the union of tiles of $\CT_k$---namely, as  $T=\bigcup_{i=1}^{k+1} S_kc_ic$, where, for each $i=1,2,\dots,k+1$, $c_i\in C^{\mathsf{incr}}_k$ (i.e.\ $c_i$ is as in \eqref{incr})---then the elements $c_ic$ also belong to $C^{\mathsf{incr}}_k$. Recall, that the product $c_ic$ is realized as the composition $c\,\circ c_i$. This composition is increasing on $\{1,2,\dots,k\}$ because the image $c_i(\{1,2,\dots,k\})$ is ordered increasingly and is contained in the set $\{1,2,\dots,k+1\}$ on which $c$ is increasing.

The proof is finished by invoking the general fact that the entropy of any deterministic system of monotilings has \tl\ entropy zero. Indeed, the \tl\ entropy of a monotiling with a shape $S$ of cardinality $r$ does not exceed the Shannon entropy of a two-element partition with probabilities $\frac1r$ and $1-\frac1r$, which equals
$$
H(\tfrac1r)=-\tfrac1r\log\tfrac1r - (1-\tfrac1r)\log(1-\tfrac1r).
$$
In the deterministic case, $\TTT$ is the inverse limit $\overset{\leftarrow}\lim_{k\to\infty}\T_k$ (see Remark \ref{remd}), so its entropy equals
$$
\htop(\TTT,G)=\lim_{k\to\infty} \htop(\T_k,G)\le \lim_{k\to\infty} H(\tfrac1{k!})=0.
$$
(This also implies that $\htop(\T_k,G)=0$ for each $k\in\N$.)
We conclude that each of the sets $C^{\mathsf{incr}}_k$, being the set of visiting times of $\CT_k\in\T_k$ to the clopen set $[``S_k"]$, is a strongly deterministic subset of~$G$.
\end{proof}

\subsection{Strong determinism via distality}\label{sdistal}
In 1968 W. Parry \cite{Pa} proved that distal $\Z$-actions have \tl\ entropy zero. The proof passes to actions of countable amenable groups essentially without modifications (see also \cite{Ya} 2015, where a more general result is obtained via a different approach). 
Because Parry's proof is very short, for reader's convenience we include a version adapted to amenable group actions. 
\begin{prop}\label{distal} Let $G$ be a countable amenable group.
Any distal system $(X,G)$ has \tl\ entropy zero. 
\end{prop}
\begin{proof}
By a result of Ellis \cite{El} (1958), any distal action $(X,G)$ (of any group $G$ on a compact metric space $X$) is \emph{pointwise almost periodic}, i.e.\ $X$ is a disjoint union of minimal sets. This implies that any ergodic $G$-\inv\ measure on $X$ is supported by a minimal set. (We remark in passing that, by a theorem of Furstenberg \cite{Fu1}, any distal action of a general, not necessarily amenable, group has an \inv, and thus also an ergodic, measure.) Now, in the amenable case, the ergodic version of the variational principle (see Subsection \ref{2.5}) allows us to restrict the proof to ergodic measures and minimal distal actions. 
So without loss of generality, we assume that $(X,G)$ is a minimal distal system. Let $\mu$ be an ergodic \im\ on $X$. We can also assume that $\mu$ is nonatomic, as otherwise we are dealing with a finite space $X$ on which any action trivially has entropy zero. Fix a point $x_0\in X$ and a positive number $r$. Let $U_0=X$ and let $(U_n)_{n\in\N}$ be a nested \sq\ of open balls around $x_0$ such that $\mu(U_n)\le r^n$. This is possible as $\mu$ is nonatomic. Let $V_n=U_{n-1}\setminus U_n$ ($n\ge 1$). Then $\P=\{V_n:n\ge 1\}\cup\{x_0\}$ is a countable, measurable partition of $X$. Since the function $t\mapsto -t\log t$ is increasing for small $t$ and decreasing for $t$ close to $1$, the Shannon entropy of this partition is at most
$$
\hspace{25pt}-\mu(V_1)\log\mu(V_1)+\sum_{n=2}^\infty -nr^n\log r\le -(1-r)\log(1-r)+\frac {-r\log r}{(1-r)^2}.
$$
The right hand side of the above formula can be made arbitrarily small by the choice of $r$. Finally observe that the partition $\P$ is generating (i.e.\ separates orbits). Indeed, by minimality, each orbit visits arbitrarily small balls $U_n$, and, by distality, any two distinct orbits cannot visit all such balls at the same ``times'' $g\in G$. Now we use two facts about measure entropy that are valid for any action of a (countable) amenable group: 
\begin{itemize} 
	\item the dynamical entropy of a measure preserving action is attained on any (finite or 
	countable with finite Shannon entropy) generating partition,
	\item the dynamical entropy of a (finite or countable) partition is dominated by the Shannon entropy of this partition.
\end{itemize}
So the measure-preserving system $(X,\mu,G)$ has zero entropy, and hence the \tl\ distal system $(X,G)$ has \tl\ entropy zero.
\end{proof}
In view of Theorem~\ref{vit}, Proposition \ref{distal} leads to a rather large family of strongly deterministic sets. Indeed, take any distal action $(X,G)$, any point $x\in X$ and any set $U$ with small boundary, whose interior is not disjoint from the orbit of $x$ (equivalently, any minimal distal action, any point and any set with nonempty interior and small boundary), and let $A=\{g\in G: g(x)\in U\}$. Then $A$ is strongly deterministic.

This approach is especially useful when applied to so-called generalized polynomials. By 
\emph{generalized polynomials in $d$ variables} $(d\in\N)$ we mean the elements of the smallest class of functions $u:\R^d\to\R$ which contains the coordinate projections and constants, and is closed under addition, multiplication and the ``integer part'' operation. The following fact holds:

\begin{prop}\label{poly}
Consider a function $u=(u_1,u_2,\dots,u_l):\R^d\to[0,1)^l$ where $u_i$ is a generalized polynomial in $d$ variables, for each $i=1,2,\dots,l$. Let $W\subset[0,1)^l$ be an open set, such that $\partial W$ has Lebesgue measure zero. Then the set
$$
A=\{n\in\Z^d: u(n)\in W\}
$$
is a strongly deterministic subset of $\Z^d$.

%\begin{enumerate} 
%\item Assume $u$ sends $\Z^d$ to $\Z^l$. For any probability measure preserving system \xmt, any set $A\in\Sigma$ with $\mu(A)>0$, and any $\eps\in[0,1)$, the set 
%$$
%\{ n\in\Z^d:\mu(A\cap T^{u_1( n)}(A)\cap\cdots\cap T^{u_l( n)}(A))>\eps\} 
%$$
%is strongly deterministic in $\Z^d$.  
%If $u$ is bounded then denoting by $\|\cdot\|$ the distance to the nearest integer lattice point in $\Z^l$, we have, for any $\eps>0$, that 
%$$
%\{n\in\Z^d: \|u(n)\|<\eps\}
%$$
%is a strongly deterministic subset of $\Z^d$.
%\item If $d=l$, $P$ is linear and sends $\Z^l$ injectively to $\Z^l$, then the image $ P(\Z^l)$ is a strongly deterministic subset of $\Z^l$.
%\end{enumerate}
\end{prop}
 
\begin{proof}
It is proved in \cite{BL} that any bounded generalized polynomial $u:\Z^d \to \R^l$ has a representation $u(n) = f(\varphi(n)x)$, $n\in\Z^d$, where $f$ is a piecewise polynomial function on a compact nilmanifold $X$, $x\in X$ and $\varphi:\Z^d\times X\to X$ is an ergodic $\Z^d$-action by translations on $X$. Thus the set $A$ equals
$$
\{n\in\Z^d: \varphi(n)x\in f^{-1}(W)\}.
$$
The set $f^{-1}(W)$ is open and has small boundary in $X$. It is known that any $\Z^d$-action by translations on a nilmanifold is distal (see, e.g., \cite[Theorem 2.14]{Le1}; also cf. \cite[Ch. 4, Theorem 3]{AHG}). 
It now follows from Theorem \ref{vit} that the set $A$ is strongly deterministic.
\end{proof}

We say that a function $F:\mathbb{Z}^k \to \mathbb{Z}$ is a \emph{generalized linear function} if $F$ belongs to the smallest class of functions which contains the coordinate projections and constants, and is closed under addition, multiplication by constants, and the ``integer part'' operation. We say that a function $F=(F_1,F_2,\dots,F_k):\mathbb{Z}^k\to \mathbb{Z}^k$ is a $k$-dimensional generalized linear function if each component is a generalized linear function in $k$ variables. 

\begin{prop}
If $F=(F_1,F_2,\dots,F_k)$ is a $k$-dimensional generalized linear function such that $F_i$ is unbounded for each $i=1,2,\dots,k$, then $F(\Z^k)$ is a strongly deterministic subset in $\Z^k$. 
\end{prop}  
\begin{proof}[Sketch of proof]  
It was proved in \cite{BLS} that the indicator function of the image set of any unbounded $1$-dimensional generalized linear function can be written as the visiting times of a point to a ``polygonal set'' under a toral translation (which obviously has zero entropy). This proof generalizes to $k$ variables with no substantial changes. Thus, applying our characterization of deterministic functions in Theorem~\ref{vit} gives the desired result.
\end{proof}

\subsection{Strong determinism via automatic sets I} Von Haeseler \cite{H} extended the notion of an automatic \sq\ from $\Z$ to a general group setting. We recount the basic definitions here. Let $G$ be a finitely generated group equipped with a norm $\|\cdot \|$. In this context a norm is a map from $G$ to $[0,\infty)$ satisfying $\|g\|=0$ if and only if $g=e$, $\|g\|=\|g^{-1}\|$ for all $g\in G$ and $\|gg'\|\le \|g\|+\|g'\|$ for all $g,g'\in G$. A group endomorphism $H:(G,\|\cdot\|)\to (G,\|\cdot\|)$ is said to be expanding if there exists $C>1$ such that $\|H(g)\| \ge C\|g\|$ for all non-unit $g\in G$. Suppose now that $H$ is an expanding endomorphism such that $H(G)$ is a subgroup of finite index. We then call $V\in G$ a residue set (w.r.t. $H$) if $e\in V$ and if for any $g\in G$ there exists a unique $v\in V$ and a unique $g'\in $ such that $g= v H(g')$. We can then iterate this procedure to find the corresponding $v'$ and $g''$ such that $g'=v' H(g'')$, and so on. We thus say that a residue set $V$ is a \emph{complete digit set} if each $g\in G\setminus\{e\}$ has a finite representation as 
\begin{equation}\label{eqn1}
    g = v_0 H(v_1)H^2 (v_2) \cdots H^n (v_{n}), 
\end{equation}
with each $v_i\in V$ ($i=0,1,\dots,n$) and $v_n\neq e$. 

As an example illustrated in \cite{H}, if $G=\langle x \rangle$, i.e.\ $G$ is isomorphic to $\mathbb{Z}$, then $H(x^j)=x^{3j}$ is expanding and $V=\{x^{-1}, x^0, x^1\}$ is a complete digit set. In \cite[Theorem 2.2.7]{H}, von Haeseler gives a fairly simple criterion for a complete digit set to exist, which encompasses interesting non-Abelian examples such as the discrete Heisenberg group.

We need some notation not present in von Haeseler. The \sq\ $\Rep(g)=(v_0(g),v_1(g),\dots)$, where, for $i=0,1,\dots,n$, $v_i(g)=v_i$ are the elements of $V$ appearing in \eqref{eqn1} and $v_i(g)=e$ for $i>n$, will be called the \emph{digit representation} of $g$ (with respect to $(V,H)$). The number $n$ will be called the \emph{order} (of magnitude) of $g$ and denoted by $\mathsf{Ord}(g)$. Additionally, we let $\bar v(e)=(e,e,\dots)$ and $\mathsf{Ord}(e)=0$. Note that $\bar v$ is a bijection between $G$ and $\bigcup_{n\ge0}\bigl(V^{\times n}\times(V\setminus\{e\})\times\{e\}\times\{e\}\times\cdots\bigr)\cup\{\bar v(e)\}$ (where $V^{\times n}$ denotes the Cartesian power). 

We will say that a complete digit set is ``good'' if there exists $n_0$ such that whenever both $\mathsf{Ord}(g)\le n$ and $\mathsf{Ord}(g')\le n$, then we have that $\mathsf{Ord}(gg')\le n+n_0$. This replicates our intuition from decimal expansions that the sum of two 5-digit numbers will have no more than 6 digits.

Let $\Lambda$ be a finite set and $x\in \Lambda^G$. We say that $x$ is $(V,H)$-automatic if there exists a finite set $S$ called \emph{the set of states} with a distinguished \emph{initial state} $s_0\in S$, maps $\alpha_v: S\to S$ indexed by $v\in V$ and a map $\omega:S\to \Lambda$ such that the following conditions hold:
\begin{itemize}
    \item $x(e)=\omega(s_0)$,
    \item For $g\in G\setminus\{e\}$ with representation $g= v_0 H(v_1)H^2 (v_2) \cdots 
    H^{n-1}(v_n)$, as above, we have that
    \begin{equation}\label{eqof}
    x(g)=\omega(\alpha_{v_0} \circ \alpha_{v_1}\dots \circ \alpha_{v_n} (s_0)).
    \end{equation}
\end{itemize}
We note that in \eqref{eqof}, we have reversed the order of the $\alpha_{v_i}$'s from what von Haeseler has. This makes his definitions more consistent with classical notions of automaticity.

While von Haeseler does not compute $\mathsf{Rate}(x)$ for $(V,H)$-automatic $x\in \Lambda^G$, we will do so now under the additional assumption that $V$ is ``good''. 

\begin{prop}\label{76}
Suppose $(G,\|\cdot\|)$ is a finitely generated group equipped with a norm, that $H$ is an expanding group endomorphism with $H(G)$ of finite index and that $V$ is a ``good'' digit set. If $x\in \Lambda^G$ is $(V,H)$-automatic, then $\mathsf{Rate}(x)=0$.
\end{prop}

\begin{cor}Applying Theorem \ref{chara}, we see that if $G$ is also amenable then any $(V,H)$-automatic $x$ is strongly deterministic.
\end{cor}

\begin{proof}[Proof of Proposition \ref{76}]
In the proof we will essentially mimic the technique of Cobham's classic result \cite{C} on the complexity of automatic sequences over $\mathbb{N}$.

We will use one more notation: for $0\le k\le l$ and $g\in G$ we let
$$
g_{[k,l]}=H^k(v_k(g))H^{k+1}(v_{k+1}(g))\cdots H^l(v_l(g)). 
$$ 
In other words, $g_{[k,l]}$ is the element of $G$ whose digit representation equals the restriction of $\Rep(g)$ to the interval of integers $[k,l]\cap\Z$ (all other digits are set to $e$). Note that for any $k\ge 0$ and $h,g\in G$ with $\mathsf{Ord}(h)\le k$, we have
\begin{gather}
g=g_{[0,k]}\cdot g_{[k+1,\infty)},\label{split}\\
(hg)_{[0,k]}=(h_{[0,k]}\cdot g_{[0,k]})_{[0,k]}.\label{kit}
\end{gather}

\smallskip
By Proposition \ref{limex}, it suffices to prove that 
\[
\lim_{m} \frac{\log \mathbf C_x(K_m)}{|K_m|} =0
\]
for some sequence of finite sets $(K_m)_{m\in\N}$. We will let $K_m$ consist of all $h\in G$ with $\mathsf{Ord}(h)\le m$. Clearly, $|K_m|=|V|^m$. We need to count how many different blocks $B_g\in\Lambda^{K_m}$, given by $B_g(h)=x(hg)$ ($h\in K_m$), will one obtain by varying $g$ over the group $G$. To this end we represent each $g\in G$ as $g_{[0,m]}\cdot g_{[m+1,\infty)}$ (using \eqref{split}). Then $hg = h\cdot g_{[0,m]}\cdot g_{[m+1,\infty)}$. Since both $\mathsf{Ord}(h)$ and $\mathsf{Ord}(g_{[0,m]})$ are bounded by $m$, we have $\mathsf{Ord}(h\cdot g_{[0,m]})\le m+n_0$ (since $V$ is ``good''). Thus,
\begin{gather}
h\cdot g_{[0,m]}=(h\cdot g_{[0,m]})_{[0,m]}\cdot(h\cdot g_{[0,m]})_{[m+1,m+n_0]}, \text{ and}\\
hg=(h\cdot g_{[0,m]})_{[0,m]}\cdot(h\cdot g_{[0,m]})_{[m+1,m+n_0]}\cdot g_{[m+1,\infty)}.\label{rep}
\end{gather}
In the main counting argument we will ignore the elements $g$ for which $\mathsf{Ord}(g)\le m+n_0$. They belong to $|K_{m+n_0}|$ and thus can produce at most $|K_{m+n_0}|=|V|^{m+n_0}$ different blocks $B_g$. (We will add this number at the end.) For the remaining elements $g$ it is seen from the formula \eqref{eqof} that for any $h\in K_m$, $x(hg)$ is determined by two objects:
\begin{enumerate}
	\item[(a)] the element $s_{n-m}(hg)=\alpha_{v_{m+1}(hg)}\circ\alpha_{v_{m+2}(hg)}\dots\circ 
	\alpha_{v_n(hg)}(s_0)$, where $n$ stands for $\mathsf{Ord}(hg)$,
	\item[(b)] the function $\alpha_{v_0(hg)}\circ\alpha_{v_1(hg)}\dots\circ\alpha_{v_m(hg)}:S\to 
	S$.
\end{enumerate}
The function in (b) is determined by $(hg)_{[0,m]}$.
Since $\mathsf{Ord}(h)\le m$, we have, by \eqref{kit}, that $(hg)_{[0,m]}=(h\cdot g_{[0,m]})_{[0,m]}$ (which is the first item in \eqref{rep}). Since $G$ is a group and $V$ is ``good'', we know that for all $g\in G\setminus K_{m+n_0}$, we have $\mathsf{Ord}(hg)\ge m$, so that the function in (b) is well-defined. If $n=\mathsf{Ord}(hg)=m$, then in (a) we define
$s_0(hg)=s_0$. Until declared otherwise, we restrict out attention to elements $g\in G\setminus{K_{m+n_0}}$ with a fixed restriction $g_{[0,m]}$. The element $s_{n-m}(hg)$ in~(a) depends on $(hg)_{[m+1,\infty)}=(h\cdot g_{[0,m]})_{[m+1,m+n_0]}\cdot g_{[m+1,\infty)}$ (the second and third items in \eqref{rep}). We classify the elements $h\in K_m$ according to the value of $(h\cdot g_{[0,m]})_{[m+1,m+n_0]}$. In this manner, we have partitioned $K_m$ into $|V|^{n_0}$ subsets. We call this partition $\P$ and we restrict our attention to one atom $P$ of $\P$. On $P$, $s_{m-n}(hg)$ depends only on $g_{[m+1,\infty)}$. Since $s_{m-n}(hg)\in S$, all elements $g$ (with the given restriction $g_{[0,m]}$) can be classified into $|S|$ classes according to the value $s_{n-m}(hg)$ on $P$. That is to say, all elements $g$ in one class determine the same (constant) assignment $h\mapsto s_{n-m}(hg)$ on $P$. Thus, for every $g$ (still with the restriction $g_{[0,m]}$), the assignment $h\mapsto s_{n-m}(gh)$ is one of $|S|^{|V|^{n_0}}$ functions from $K_m$ to $S$, constant on the atoms of $\P$. Taking into account that there are $|V|^m$ choices for $g_{[0,m]}$ and adding $|V|^{m+n_0}$ for the elements $g\in K_{m+n_0}$, we obtain that there are not more than $|V|^m(|S|^{|V|^{n_0}}+|V|^{n_0})$ blocks $B_g$ as $g$ ranges over $G$. Clearly, this number estimates $\mathbf C_x(K_m)$ as the function $\omega:S\to\Lambda$ can only lower the complexity. Taking logarithms and dividing by $|K_m|=|V|^m$, we get a \sq\ which tends to $0$.
\end{proof}

\subsection{Strong determinism via automatic sets, II} We consider a second variation of automatic sequences, this time applied to the semigroup  $(\N,\times)$. To differentiate between natural numbers viewed as elements of the above multiplicative semigroup and the same natural numbers viewed additively (as they were up to now), the elements of $(\N,\times)$ will be denoted by capital letters $N,M$. This semigroup is isomorphic to the additive semigroup $\ell_1(\N_0)$ of all finitely supported $\N_0$-valued \sq s, simply by mapping each $N\in\N$ to the sequence of powers for each prime in the prime factor decomposition of $N$. That is, $N=\prod_{j\in\N} p_j^{e_j}$ corresponds to the sequence $(e_j)_{j\in\N}$. We will let $\pi$ denote the isomorphism from $(\N,\times)$ to $(\ell_1(\N_0),+)$, and when needed, use $e_j(N)$ instead of $e_j$.

We may then fix an integer $b\ge 2$, let $V= \{0,1,\dots,b-1\}$ and consider a projection $\phi$ from $\ell_1(\N_0)$ to $V^*$ (the collection of all finite words over the alphabet $V$) that acts by concatenating the base-$b$ expansions $\bar e_j$ of the numbers $e_j$, in decreasing order ($e_j=0$ being represented by the empty string), so that $\phi((e_j)_{j\in\N}) = \dots \bar e_3\bar e_2\bar e_1$. For example, if $b=10$, then $\phi(7^{903}5^{2}3^{0}2^{1})=90321$. Note that since all but finitely many numbers $e_j$ must be $0$, the string resulting from $\phi$ will always be finite. For convenience, we will abuse notation and denote $\phi\circ \pi$ by just $\phi$. If $\phi(N)=v_nv_{n-1}\dots v_0$ with $v_n\neq 0$, then we will say that the order of $N$, denoted by $\mathsf{Ord}(N)$, is $n$.

\begin{defn}
Let $\Lambda$ be a finite alphabet and let $x\in \Lambda^\N$. We say that $x$ is \emph{automatic} (with respect to $(\N,\times)$ and $b$) if there exists a finite set $S$ of states with a distinguished initial state $s_0\in S$, maps $\alpha_v: S\to S$ indexed by $v\in V$ and a map $\omega:S\to\Lambda$ such that
\begin{itemize}
    \item For $N\in\N$ with $\phi(N)=v_nv_{n-1}v_{n-2}\dots v_0$, ($v_i\in V$, 
    $i=0,1,\dots,n$, $n=\mathsf{Ord}(N)$), we have that
    \begin{equation}\label{eq:omegadef}
    x(N)=\omega(\alpha_{v_0} \circ \dots \circ \alpha_{v_n} (s_0)). 
    \end{equation}
\end{itemize}
\end{defn}

\begin{prop}\label{seven}
If $x\in\Lambda^\N$ is automatic with respect to $(\N,\times)$ and $b$, then $x$ is strongly deterministic.
\end{prop}

\begin{rem}Note that since $(\N,\times)$ is not finitely generated, this example is distinct from the example of the previous section.
\end{rem}

\begin{proof}[Proof of Proposition \ref{seven}]
We will apply Theorem \ref{chara} and Proposition \ref{limex}, so that it suffices to show, for some sequence of sets $(K_m)_{m\in\N}$, that 
\[
\lim_{m\to\infty} \frac{\log \mathbf C_x(K_m)}{|K_m|}=0.
\]
For this we will again mimic the method of Cobham \cite{C}. 

We will require an additional definition. If $\phi(N)=v_n v_{n-1}\dots v_0$ with $n=\mathsf{Ord}(N)$ and if $k,l\in\N$, $k\le l$, then $\phi_{[k,l]}(N) = v_lv_{l-1}\dots v_k$, where if $i>n$ then we assume $v_i$ represents the empty string. 

We will let $K_m\subset \mathbb{N}$ consist of all powers of 2 of the form $M=2^{e_1}$ with $e_1\in [b^{m+1},2b^{m+1}-1]$. Consider $K_m N$ for some $N\in\N$. For any $M\in K_m$, the value of $x(M\!N)$ is determined by two objects:
\begin{enumerate}
	\item[(a)] the element $s_{n-m}(M\!N)=\alpha_{v_{m+1}(M\!N)}\circ\alpha_{v_{m+2}(M\!N)}\dots\circ 
	\alpha_{v_n(M\!N)}(s_0)$, where $n$ stands for $\mathsf{Ord}(M\!N)$, 
	\item[(b)] the function $\alpha_{v_0(M\!N)}\circ\alpha_{v_1(M\!N)}\dots\circ\alpha_{v_m(M\!N)}:S\to
	S$.
\end{enumerate}
By the definition of $K_m$, we always have that $\mathsf{Ord}(M\!N)\ge m$ for all $M\in K_m$, so the functions $\alpha_{v_i(M\!N)}$ in (b) above are well-defined.
If $n=\mathsf{Ord}(M\!N)=m$ then in (a) we let $s_0(M\!N)=s_0$.
Now, let $e_0\in [0,b^{m+1}-1]$ be the integer which is congruent to $e_1(N)$ modulo $b^{m+1}$, and let $I_m^1 := [b^{m+1},2b^{m+1}-e_0-1]$ and $I_m^2 := [2b^{m+1}-e_0-1,2b^{m+1}-1]$. Partition $K_m$ into $K_m^1$ and $K_m^2$ so that $M\in K_m^\iota$ if $M=2^{e_1}$ with $e_1\in I_m^\iota$, $\iota=1,2$. This partition is important since the digits in the base-$b$ expansion of of $e_1(M\!N)$ for $M\in K_m^\iota$ can only differ in their last $m+1$ places ($\iota=1,2$). Since  we also have that all elements in $\pi(K_m N)$ agree on $e_j$ for $j\ge 2$, we know that $\phi_{[m+1,\infty)}$ (and hence $s_{n-m}$) is constant on each of the two sets $K_m^\iota N$. Thus, for every $N$ with the same $e_0$, the assignment $M\mapsto s_{n-m}(M\!N)$ is one of $|S|^2$ functions from $K_m$ to $S$. Moreover, since $e_1(M)\ge b^{m+1}$ for each $M\in K_m$, we have that $\phi_{[0,m]}(M\!N)$ is dependent only on $M$ and $e_0$, so the assignment $M\mapsto \alpha_{v_0(M\!N)}\circ\alpha_{v_1(M\!N)}\dots\circ\alpha_{v_m(M\!N)}$ is completely determined by $e_0$. Since there are $b^{m+1}$ possibilities for $e_0$, we conclude that there are at most $|S|^2 b^{m+1}=|S|^2|K_m|$ possibilities for blocks $B_N\in \Lambda^{K_m}$, given by $B_N =x(M\!N)$ ($M\in K_m$). As before, this shows that $\mathbf C_x(K_m)$ grows at most linearly in $|K_m|$, so taking its logarithm and dividing by $|K_m|$ completes the proof.
\end{proof}

As an explicit example, with $b=2$, we may consider a ``multiplicative Thue-Morse sequence" $x\in \{0,1\}^\N$ defined by the rule: $x(N)=1$ if and only if $\phi(N)$ has an odd number of $1$'s. This is automatic by considering $S=\{0,1\}$, $s_0=0$, $\alpha_0(s)=s$, $\alpha_1(s)=1-s$, and $\omega(s)=s$, which is the same automaton set-up used for the classical Thue-Morse sequence (see, for example, \cite{ARS}).

\subsection{$\F$-determinism via generalizations of $k$-free numbers} 
As we have already mentioned, the set of square-free numbers is completely deterministic (i.e.\ $\F$-deterministic for the classical F\o lner \sq\ $F_n=\{1,2,\dots,n\}$ in $\N$) but not strongly deterministic. There are several generalizations of this fact, some of them still concerning $\N$, some, $\Z^n$. 
\begin{enumerate} 
	\item Consider any subset $\B\subset\N$. A number $n$ is \emph{$\B$-free} if no $b\in\B$ 
	divides $n$. In \cite{ALR} it is shown that if $\B$ satisfies Erd\H os's condition: $\B$ is infinite, consists of pairwise relatively prime numbers and satisfies $\sum_{b\in\B}\frac1b<\infty$, then the set of $\B$-free numbers is completely deterministic but not strongly deterministic.
	
	\item Consider the additive group of $\mathcal{O}_K$, the ring of integers of some algebraic extension of $\mathbb{Q}$.  An integer $a\in \mathcal{O}_K$ is said to be $k$-free if the principal ideal $(a)$ generated by $a$ does not contain the $k$th power of any prime ideal. The set of $k$-free integers is $\F$-deterministic for $\F$ being the F\o lner sequence of cubes, centered at the origin, with sides of size $2n+1$. This is a consequence of Corollary 1.2 in \cite{CV}.
    
  \item Let $\Lambda$ be a lattice on $\mathbb{R}^m$ ($m\in\N$) equipped with addition and for $\ell\neq 0$ in $\Lambda$ define its $k$-content $c_k(\ell)$ as the largest integer $c$ such that $\ell\in c^k \Lambda$. (Extend $c_k$ to $0$ by defining $c_k(0)=\infty$.) The set of $k$-free points $V=V(\Lambda,k)$ is the set of points with $c_k(\ell)=1$. This set is again $\F$-deterministic for $\F=(F_n)_{n\in\N}$ being the F{\o}lner sequence of balls of radius $n$ centered at the origin, intersected with $\Lambda$. This follows from the work of Pleasants and Huck \cite{PH}: they do not explicitly state that the characteristic function of $V(\Lambda,k)$ is $\F$-generic, but this follows easily from their work, and they show that the entropy for the resulting measure is $0$. They also show that the rate of growth of the complexity (what they call the patch-counting entropy) is non-zero, so this is an example of an $\F$-deterministic set that is not strongly deterministic. 
\begin{rem}
Since the lattice $\Lambda$ in example (3) is the image of $\Z^m$ by a linear change of coordinates, the set $V(\Lambda,k)$ is a linear image of the set of $k$-free elements in $\Z^m$, which equals 
$$
\{(k_1,k_2,\dots,k_m)\in\Z^m: \mathsf{gcd}(k_1,k_2,\dots,k_m) \text{ is $k$-free in $\N$}\}.
$$
The change of coordinates results in changing the F\o lner \sq\ consisting of balls in $\Lambda$ to ellipsoids in $\Z^m$.
\end{rem}
\end{enumerate}

\end{document}